\documentclass[11pt]{article}
\usepackage{authblk}
\usepackage[toc,page]{appendix}
\usepackage[top=2cm, bottom=2cm, left=2cm, right=2cm]{geometry}

\usepackage{color}
\usepackage{helvet}         
\usepackage{courier}        
\usepackage{type1cm}        

\usepackage{framed} 
\usepackage{tikz}
\usepackage{makeidx}         
\usepackage{graphicx}        
\usepackage{multicol}        
\usepackage[bottom]{footmisc}

\usepackage{amsmath}
\usepackage{amssymb}
\usepackage{bbold}
\usepackage{amsthm}
\usepackage{subcaption}
\usepackage{sidecap}
\usepackage{floatrow}
\usepackage{pdflscape}
\usepackage{comment}
\usepackage[font=small]{caption}
\usepackage{enumitem}
\usepackage{esint}

\usepackage[hidelinks]{hyperref}

\usepackage{scalerel}
\usepackage{shuffle}
\usepackage{mathrsfs}

\newtheorem{theorem}{Theorem}

\newtheorem{lemma}[theorem]{Lemma}

\newtheorem{proposition}[theorem]{Proposition}

\theoremstyle{definition}
\newtheorem{definition}[theorem]{Definition}

\theoremstyle{remark}

\newtheorem{remark}[theorem]{\bf Remark}

\numberwithin{theorem}{section}
\numberwithin{figure}{section}
\numberwithin{equation}{section}

\begin{document}

\title{Multi-time Loewner energy: rate function for large deviation}
\bigskip{}
\author[1]{Mo Chen\thanks{chenmothuprobab@gmail.com}}
\author[1]{Chongzhi Huang\thanks{huangchzh2001prob@gmail.com}}
\author[1]{Hao Wu\thanks{hao.wu.proba@gmail.com.}}
\affil[1]{Tsinghua University, China}
\date{}

%
%


\global\long\def\qnum#1{\left[#1\right]_q }
\global\long\def\qfact#1{\left[#1\right]_q! }
\global\long\def\qbin#1#2{\left[\begin{array}{c}
	#1\\
	#2 
	\end{array}\right]_q}

\global\long\def\defpatt{\shuffle}
\global\long\def\rainbow{\includegraphics[scale=0.15]{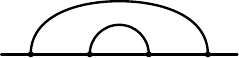}}
\global\long\def\rainbowBig{\includegraphics[scale=0.3]{figures/link-2}}

\newcommand{\nradpartfn}[2]{\mathcal{Z}^{#2}_{#1\mathrm{\textnormal{-}rad}}}
\global\long\def\covmap{h}

\global\long\def\mslitdriv{\omega}
\global\long\def\nnofloops{\mathscr{L}}

\global\long\def\Selberg{S}
\global\long\def\Diff{\Theta}
\global\long\def\SinDiff{\Xi}

\global\long\def\contour{\mathscr{C}}


\global\long\def\U{\mathbb{U}}
\global\long\def\T{\mathbb{T}}
\global\long\def\HH{\mathbb{H}}
\global\long\def\R{\mathbb{R}}
\global\long\def\C{\mathbb{C}}
\global\long\def\N{\mathbb{N}}
\global\long\def\Z{\mathbb{Z}}
\global\long\def\E{\mathbb{E}}
\global\long\def\PP{\mathbb{P}}
\global\long\def\rate{\mathcal{J}}
\global\long\def\QQ{\mathbb{Q}}
\global\long\def\A{\mathbb{A}}
\global\long\def\one{\mathbb{1}}

\newcommand{\PPspiral}[1]{\mathbb{P}^{\mu}_{#1\mathrm{\textnormal{-}rad}}}
\newcommand{\PPnospiral}[1]{\mathbb{P}^{0}_{#1\mathrm{\textnormal{-}rad}}}
\newcommand{\PPnospiralrho}{\mathbb{P}^{0; \bs{\rho}}}
\newcommand{\PPspiralrho}{\mathbb{P}^{\mu; \bs{\rho}}}

\global\long\def\CR{\mathrm{CR}}
\global\long\def\ST{\mathrm{ST}}
\global\long\def\SF{\mathrm{SF}}
\global\long\def\cov{\mathrm{cov}}
\global\long\def\dist{\mathrm{dist}}
\global\long\def\SLE{\mathrm{SLE}}
\global\long\def\hSLE{\mathrm{hSLE}}
\global\long\def\CLE{\mathrm{CLE}}
\global\long\def\GFF{\mathrm{GFF}}
\global\long\def\inte{\mathrm{int}}
\global\long\def\ext{\mathrm{ext}}
\global\long\def\inrad{\mathrm{inrad}}
\global\long\def\outrad{\mathrm{outrad}}
\global\long\def\dimH{\mathrm{dim}}
\global\long\def\capa{\mathrm{cap}}
\global\long\def\diam{\mathrm{diam}}
\global\long\def\sign{\mathrm{sgn}}
\global\long\def\cat{\mathrm{Cat}}
\global\long\def\cst{\mathrm{C}}
\global\long\def\ck{\mathrm{C}_{\kappa}}
\global\long\def\free{\mathrm{free}}
\global\long\def\hF{{}_2\mathrm{F}_1}
\global\long\def\simple{\mathrm{simple}}
\global\long\def\even{\mathrm{even}}
\global\long\def\odd{\mathrm{odd}}
\global\long\def\st{\mathrm{ST}}
\global\long\def\usf{\mathrm{USF}}
\global\long\def\Leb{\mathrm{Leb}}
\global\long\def\LP{\mathrm{LP}}
\global\long\def\I{\mathrm{I}}
\global\long\def\II{\mathrm{II}}
\global\long\def\hcap{\mathrm{hcap}}
\global\long\def\Poisson{\mathrm{P}}

\global\long\def\LA{\mathcal{A}}
\global\long\def\LB{\mathcal{B}}
\global\long\def\LC{\mathcal{C}}
\global\long\def\LD{\mathcal{D}}
\global\long\def\LF{\mathcal{F}}
\global\long\def\LK{\mathcal{K}}
\global\long\def\LE{\mathcal{E}}
\global\long\def\LG{\mathcal{G}}
\global\long\def\LGmu{\mathcal{G}_{\mu}}
\global\long\def\LI{\mathcal{I}}
\global\long\def\LJ{\mathcal{J}}
\global\long\def\LL{\mathcal{L}}
\global\long\def\LM{\mathcal{M}}
\global\long\def\LN{\mathcal{N}}
\global\long\def\OO{\mathcal{O}}
\global\long\def\LQ{\mathcal{Q}}
\global\long\def\LR{\mathcal{R}}
\global\long\def\LT{\mathcal{T}}
\global\long\def\LS{\mathcal{S}}
\global\long\def\LU{\mathcal{U}}
\global\long\def\LV{\mathcal{V}}
\global\long\def\LW{\mathcal{W}}
\global\long\def\LX{\mathcal{X}}
\global\long\def\LY{\mathcal{Y}}
\global\long\def\PartF{\mathcal{Z}}
\global\long\def\LH{\mathcal{H}}
\global\long\def\LJ{\mathcal{J}}

\global\long\def\blm{m}

\global\long\def\LZ{\mathcal{Z}}
\global\long\def\LZrp{\mathcal{Z}_{\alpha; \bs{s}}^{(p)}}
\global\long\def\LJrp{\mathcal{J}_{\alpha; \bs{s}}^{(p)}}
\global\long\def\chamberrp{\chamber_{\alpha; \bs{s}}^{(p)}}
\global\long\def\LErp{\mathcal{E}_{\alpha; \bs{s}}^{(p)}}
\global\long\def\Greenrp{G_{\alpha; \bs{s}}^{(p)}}
\global\long\def\Prp{P_{\alpha; \bs{s}}^{(p)}}
\global\long\def\norcst{\mathrm{C}_{\kappa}^{(\mathfrak{r})}}
\global\long\def\LZalphar{\mathcal{Z}_{\alpha}^{(\mathfrak{r})}}

\newcommand{\LZtwo}{\mathcal{Z}_{\includegraphics[scale=0.15]{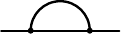}}^{(\kappa)}}
\newcommand{\Etwo}{\mathbb{E}_{\includegraphics[scale=0.15]{figures/link-0}}}
\newcommand{\PPtwo}{\PP_{\includegraphics[scale=0.15]{figures/link-0}}^{(\kappa)}}
\newcommand{\PPtworho}{\PP_{\includegraphics[scale=0.15]{figures/link-0}}^{(\kappa; \boldsymbol{\rho})}}
\newcommand{\LEtwo}{\mathcal{E}_{\includegraphics[scale=0.15]{figures/link-0}}}
\newcommand{\Etworho}{\mathbb{E}_{\includegraphics[scale=0.15]{figures/link-0}}^{(\kappa; \boldsymbol{\rho})}}
\newcommand{\LItwo}{\mathcal{I}_{\includegraphics[scale=0.15]{figures/link-0}}}
\newcommand{\LUtwo}{\mathcal{U}_{\includegraphics[scale=0.15]{figures/link-0}}}
\newcommand{\LZtwor}{\LZtwo^{(\mathfrak{r})}}
\newcommand{\LHtwo}{\mathcal{H}_{\includegraphics[scale=0.15]{figures/link-0}}}
\newcommand{\LFtwo}{\mathcal{F}_{\includegraphics[scale=0.15]{figures/link-0}}}
\newcommand{\chambertwo}{\chamber_{\includegraphics[scale=0.15]{figures/link-0}}}

\newcommand{\PPkapparho}{\PP_{\includegraphics[scale=0.15]{figures/link-0}}^{(\kappa; \bs{\rho})}}

\newcommand{\LZfoura}{\mathcal{Z}_{\includegraphics[scale=0.15]{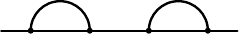}}}
\newcommand{\LZfourb}{\mathcal{Z}_{\includegraphics[scale=0.15]{figures/link-2}}}
\newcommand{\LHfoura}{\mathcal{H}_{\includegraphics[scale=0.15]{figures/link-1}}}
\newcommand{\LHfourb}{\mathcal{H}_{\includegraphics[scale=0.15]{figures/link-2}}}
\newcommand{\LFfoura}{\mathcal{F}_{\includegraphics[scale=0.15]{figures/link-1}}}
\newcommand{\LFfourb}{\mathcal{F}_{\includegraphics[scale=0.15]{figures/link-2}}}
\newcommand{\LFfouraRenorm}{\widehat{\mathcal{F}}_{\includegraphics[scale=0.15]{figures/link-1}}}
\newcommand{\LFfourbRenorm}{\widehat{\mathcal{F}}_{\includegraphics[scale=0.15]{figures/link-2}}}

\newcommand{\QQrainbow}[1]{\mathbb{Q}_{\rainbow_{#1}}}
\newcommand{\LZrainbow}[1]{\mathcal{Z}_{\rainbow_{#1}}}
\newcommand{\Frainbow}[1]{\mathscr{F}_{\rainbow_{#1}}}
\newcommand{\chamberrainbow}[1]{\chamber_{\rainbow_{#1}}}

\newcommand{\PPfusion}[1]{\mathbb{P}_{\includegraphics[scale=0.5]{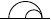}_{#1}}^{(\kappa)}}
\newcommand{\Efusion}[1]{\mathbb{E}_{\includegraphics[scale=0.5]{figures/link4fusion}_{#1}}^{(\kappa)}}
\newcommand{\LZfusion}[1]{\mathcal{Z}_{\includegraphics[scale=0.5]{figures/link4fusion}_{#1}}^{(\kappa)}}
\newcommand{\ratefusion}[1]{\mathcal{J}_{\includegraphics[scale=0.5]{figures/link4fusion}_{#1}}}
\newcommand{\LIfusion}[1]{\mathcal{I}_{\includegraphics[scale=0.5]{figures/link4fusion}_{#1}}}
\newcommand{\tilderatefusion}[1]{\tilde{\mathcal{J}}_{\includegraphics[scale=0.5]{figures/link4fusion}_{#1}}}
\newcommand{\hatratefusion}[1]{\hat{\mathcal{J}}_{\includegraphics[scale=0.5]{figures/link4fusion}_{#1}}}
\newcommand{\chamberfusion}[1]{\chamber_{\includegraphics[scale=0.5]{figures/link4fusion}_{#1}}}
\newcommand{\barchamberfusion}[1]{\overline{\chamber}_{\includegraphics[scale=0.5]{figures/link4fusion}_{#1}}}
\newcommand{\LUfusion}[1]{\mathcal{U}_{\includegraphics[scale=0.5]{figures/link4fusion}_{#1}}}
\newcommand{\LZfusionInv}[1]{\mathcal{Z}_{\rotatebox{180}{\scalebox{1}[-1]{\includegraphics[scale=0.5]{figures/link4fusion}}}_{#1}}}
\newcommand{\Ffusion}[1]{\mathscr{F}_{\includegraphics[scale=0.5]{figures/link4fusion}_{#1}}}

\newcommand{\chamberradial}[1]{\chamber_{#1\mathrm{\textnormal{-}rad}}}
\newcommand{\PPradial}[1]{\mathbb{P}_{#1\mathrm{\textnormal{-}rad}}}
\newcommand{\Eradial}[1]{\mathbb{E}_{#1\mathrm{\textnormal{-}rad}}}
\newcommand{\rateradial}[1]{\mathcal{J}_{#1\mathrm{\textnormal{-}rad}}}
\newcommand{\LIradial}[1]{\mathcal{I}_{#1\mathrm{\textnormal{-}rad}}}
\newcommand{\LVradial}[1]{\mathcal{V}_{#1\mathrm{\textnormal{-}rad}}}
\newcommand{\tilderateradial}[1]{\tilde{\mathcal{J}}_{#1\mathrm{\textnormal{-}rad}}}
\newcommand{\hatrateradial}[1]{\hat{\mathcal{J}}_{#1\mathrm{\textnormal{-}rad}}}
\newcommand{\LZradial}[1]{\mathcal{Z}_{#1\mathrm{\textnormal{-}rad}}}

\newcommand{\coulombGasHRenorm}{\widehat{\coulombGasH}}
\newcommand{\LZthree}{\mathcal{Z}_{\includegraphics[scale=0.8]{figures/link211}}}

\global\long\def\coulomb{\LH}
\global\long\def\auxcoulomb{\hat{\coulomb}}
\global\long\def\coulombGas{\LF}
\global\long\def\coulombnew{\LK}
\global\long\def\coulombLine{\LG}
\global\long\def\kfunc{p}

\global\long\def\eps{\epsilon}
\global\long\def\ov{\overline}
\global\long\def\QQrp{\QQ_{\alpha; \bs{s}}^{(p)}}

\global\long\def\bn{\mathbf{n}}
\global\long\def\MR{MR}
\global\long\def\cond{\,|\,}
\global\long\def\bigcond{\,\big|\,}
\global\long\def\Bigcond{\;\Big|\;}
\global\long\def\la{\langle}
\global\long\def\ra{\rangle}
\global\long\def\tree{\Upsilon}
\global\long\def\prob{\mathbb{P}}
\global\long\def\hm{\mathrm{Hm}}
%

\global\long\def\Im{\operatorname{Im}}
\global\long\def\Re{\operatorname{Re}}

\global\long\def\ud{\mathrm{d}}
\global\long\def\pder#1{\frac{\partial}{\partial#1}}
\global\long\def\pdder#1{\frac{\partial^{2}}{\partial#1^{2}}}
\global\long\def\pddder#1{\frac{\partial^{3}}{\partial#1^{3}}}
\global\long\def\der#1{\frac{\ud}{\ud#1}}

\global\long\def\bZnn{\mathbb{Z}_{\geq 0}}
\global\long\def\bZpos{\mathbb{Z}_{> 0}}
\global\long\def\bZneg{\mathbb{Z}_{< 0}}

\global\long\def\Vfunc{\LG}
\global\long\def\gfunc{g^{(\rr)}}
\global\long\def\hfunc{h^{(\rr)}}

\global\long\def\SimplexInt{\rho}
\global\long\def\CubeInt{\widetilde{\rho}}

\global\long\def\ii{\mathrm{i}}
\global\long\def\ee{\mathrm{e}}
\global\long\def\rr{\mathfrak{r}}
\global\long\def\chamber{\mathfrak{X}}
\global\long\def\Wchamber{\mathfrak{W}}

\global\long\def\SimplexIntKappa8{\SimplexInt}

\global\long\def\nested{\boldsymbol{\underline{\Cap}}}
\global\long\def\unnested{\boldsymbol{\underline{\cap\cap}}}
\global\long\def\unnested{\boldsymbol{\underline{\cap\cap}}}

\global\long\def\acycle{\vartheta}
\global\long\def\bcycle{\tilde{\acycle}}

\global\long\def\metric{\mathrm{dist}}

\global\long\def\adj#1{\mathrm{adj}(#1)}

\global\long\def\bs{\boldsymbol}

\global\long\def\edge#1#2{\langle #1,#2 \rangle}
\global\long\def\graph{G}

\newcommand{\conn}{\varsigma}
\newcommand{\realacycle}{\smash{\mathring{\acycle}}}
\newcommand{\realpt}{\smash{\mathring{x}}}
\newcommand{\corrind}{\LC}
\newcommand{\bssymb}{\pi}
\newcommand{\coeff}{p}
\newcommand{\MainConst}{C}

\global\long\def\removeLink{/}

\global\long\def\domainofdef{\mathfrak{U}}
\global\long\def\Test_space{C_c^\infty}
\global\long\def\Distr_space{(\Test_space)^*}

\global\long\def\bs{\boldsymbol}
\global\long\def\cst{\mathrm{C}}

\newcommand{\red}{\textcolor{red}}
\newcommand{\blue}{\textcolor{blue}}
\newcommand{\green}{\textcolor{green}}
\newcommand{\magenta}{\textcolor{magenta}}
\newcommand{\cyan}{\textcolor{cyan}}

\newcommand{\coulombGasH}{\mathcal{H}}
\newcommand{\secondbeta}{\intloop}

\newcommand{\cev}[1]{\reflectbox{\ensuremath{\vec{\reflectbox{\ensuremath{#1}}}}}}

\global\long\def\anticonf{\zeta}
\global\long\def\intloop{\varrho}
\global\long\def\Gloop{\smash{\mathring{\intloop}}}

\global\long\def\SLEmeasure{\mathrm{P}}
\global\long\def\SLEmeasureEx{\mathrm{E}}

\global\long\def\fugacity{\nu}
\global\long\def\meanderMat{\mathcal{M}}
\global\long\def\LM{\mathcal{M}}
\global\long\def\meanderMatrix{\meanderMat_{\fugacity}}
\global\long\def\meanderMatrixPrime{\meanderMat_{\fugacity(\kappa')}}
\global\long\def\meanderRenorm{\widehat{\mathcal{M}}}

\global\long\def\PartFRenorm{\widehat{\PartF}}
\global\long\def\coulombGasRenorm{\widehat{\coulombGas}}

\global\long\def\hexa{\scalebox{1.3}{\hexagon}}

\global\long\def\np{p}

\global\long\def\FKdual{\mathcal{L}}

\global\long\def\fixedindex{\flat}

\global\long\def\Dirichletradial{\mathcal{I}_{\mathrm{rad}}}
\global\long\def\LVone{\mathcal{V}_{1\mathrm{\textnormal{-}rad}}}
\maketitle
\vspace{-1cm}
\begin{center}
\begin{minipage}{0.95\textwidth}
\abstract{
The classification of probability measures that satisfy both conformal invariance and domain Markov property is equivalent to characterizing solutions to the Belavin--Polyakov--Zamolodchikov (BPZ) equations, as established by Dub\'edat~[Dub07]. In this context, the partition functions for half-watermelon SLE and for multi-radial SLE serve as fundamental solutions to the BPZ equations. In this article, we investigate the large deviation principle for both half-watermelon SLE and multi-radial SLE. The associated rate function is given by the multi-time Loewner energy, introduced in~[CHPW26]. 
As applications, we provide an alternative proof of the large deviation principle for Dyson Brownian motion, as well as a new derivation of the boundary perturbation property of the multi-time Loewner energy.}
\medbreak
\end{minipage}
\end{center}

\newpage
\tableofcontents
\newpage
\section{Introduction}
When examining the scaling limit of interfaces in two-dimensional critical lattice models, conformal invariance and domain Markov property arise naturally. In 1999, Schramm~\cite{SchrammFirstSLE} introduced the Schramm–Loewner evolution as a single random curve that possesses both conformal invariance and domain Markov property. For multiple interfaces in polygons, the combination of these two properties leads to Dub\'edat's commutation relation~\cite{DubedatCommutationSLE}. 

We say that $(\Omega; \bs{x})=(\Omega; x_1, \ldots, x_n)$ 
is a (topological) $n$-\emph{polygon} if $\Omega\subsetneq\C$ is simply connected, $\partial\Omega$ is locally connected, and $x_1, x_2, \ldots, x_n \in \partial\Omega$ are distinct points lying counterclockwise along the boundary. 
We say that $(\Omega; \bs{x})=(\Omega; x_1, \ldots, x_n)$ 
is a \emph{nice} polygon if we assume further that the marked boundary points $x_1, x_2, \ldots, x_n$ lie on $C^{1+\eps}$-boundary segments, for some $\eps>0$, so that derivatives of conformal maps on $\Omega$ are defined there.

For $n\ge 1$ and $n$-polygon $(\Omega; \bs{x})=(\Omega; x_1, \ldots, x_n)$, 
suppose $\{\PP(\Omega; \bs{x})\;|\; (\Omega; \bs{x})\}$ is a family of probability measures on $n$-tuple of disjoint continuous simple curves $\bs{\eta}=(\eta^1, \ldots, \eta^n)$ such that $\eta^j_0=x_j$ for $1\le j\le n$. Under the assumption of conformal invariance, domain Markov property, and a technical condition concerning absolute continuity, Dub\'edat~\cite{DubedatCommutationSLE} established that such a family must be encoded by a partition function $\LZ$ satisfying Belavin--Polyakov--Zamolodchikov (BPZ) equations. 
Consequently, the classification of probability measures enjoying conformal invariance and domain Markov property reduces to characterizing solutions of the BPZ equations.

In practice, when we parameterize the continuous curves $\bs{\eta}=(\eta^1, \ldots, \eta^n)$ by time, we may either take the normalization at a fixed boundary point (the chordal setting, see Section~\ref{subsec::chordalSLE}) or take the normalization at a fixed interior point (the radial setting, see Section~\ref{subsec::pre_radialLoewner}). 
The two normalizations lead to two versions of BPZ equations--the chordal version~\eqref{eqn::BPZ_H} and the radial version~\eqref{eqn::BPZ_U}. 
See also generalizations of Dub\'edat's argument in~\cite{BauerBernardKytolaMultipleSLE, GrahamSLE, FloresKlebanPDE1, KytolaPeltolaPurePartitionFunctions, 
KytolaPeltolaConformalCovBoundaryCorrelation,zhang2025multiplechordalslekappaquantum} for the chordal setting and in~\cite{KrusellWangWuCommutationRelation, zhang2025multipleradialslekappaquantum} for the radial setting.

\paragraph*{Chordal BPZ equations.}
Define 
\begin{equation}
\LX_n^{\HH}=\{(x_1, \ldots, x_n) \in \R^n\;|\; x_1<\cdots<x_n\}, 
\end{equation}
and consider functions $\LZ: \LX_n^{\HH}\to \R$. 
The system of chordal BPZ equations
\begin{equation}\label{eqn::BPZ_H}
	\frac{\kappa}{2} \frac{\partial_{j}^2\LZ}{\LZ} + \sum_{1\le i \ne j \le n } \left( \frac{2}{x_i - x_j} \frac{\partial_{i}\LZ}{\LZ} - \frac{(6-\kappa)/\kappa}{(x_i - x_j)^2} \right) = 0, \qquad \text{for all }1\le j\le n, 
\end{equation}
has a fundamental solution: 
\begin{equation}\label{eqn::halfwatermelon_pf_H}
	\LZ_{\shuffle_n}^{(\kappa)}(\bs{x}):=\prod_{1\le i<j\le n} (x_j-x_i)^{\frac{2}{\kappa}}, \qquad \text{for }\bs{x}=(x_1, \ldots, x_n)\in \LX_n^{\HH}. 
\end{equation}

\paragraph*{Radial BPZ equations.}
Define 
\begin{equation}
\LX_n^{\U}=\{(\theta_1, \ldots, \theta_n) \in \R^n\;|\; \theta_1<\cdots<\theta_n<\theta_1+2\pi\},
\end{equation}
and consider functions $\LZ: \LX_n^{\U}\to \R$. 
Fix a constant $\mu\in\R$. 
The system of radial BPZ equations
\begin{equation}\label{eqn::BPZ_U}
\frac{\kappa}{2} \frac{\partial_j^2 \LZ}{\LZ} 
+ \underset{1\leq i\neq j \leq n}{\sum} \left( \cot \left(\frac{\theta^{i}-\theta^{j}}{2}\right)
\frac{\partial_{i} \LZ}{\LZ} -\frac{(6-\kappa)/\kappa}{ 4\sin^2 \left(\frac{\theta^{i}-\theta^{j}}{2}\right)}\right)=\frac{\mu^2-n^2+1}{2\kappa}, 
 \qquad \text{for all }1\le j\le n, 
\end{equation}
has a fundamental solution:
\begin{equation}\label{eqn::nradial_pf_U}
	\LZradial{n}^{(\kappa; \mu)}(\bs{\theta}) := \prod_{1\le i<j\le n} |\ee^{\ii\theta_i} - \ee^{\ii\theta_j} |^{\frac{2}{\kappa}} \times \exp\left( \frac{\mu}{\kappa}\sum_{j=1}^{n} \theta_j \right), \qquad \text{for }\bs{\theta}=(\theta_1, \ldots, \theta_n) \in \LX_n^{\U}. 
\end{equation}

By fundamental solutions, we mean that more solutions of BPZ equations can be generated by combination of fundamental solutions of smaller size, see~\cite{KytolaPeltolaConformalCovBoundaryCorrelation, PeltolaBasisPDE} for the chordal setting.
The probability measure corresponds to the partition function $\LZ_{\shuffle_n}^{(\kappa)}$ in~\eqref{eqn::halfwatermelon_pf_H} is half-watermelon SLE which will be defined in Definition~\ref{def::halfwatermelonSLE}, and the probability measure for the partition function $\LZradial{n}^{(\kappa; \mu)}$ in~\eqref{eqn::nradial_pf_U} is multi-radial SLE with spiral which will be defined in Definition~\ref{def::multiradialSLEspiral}. 
These two models play as the building blocks of probability measures on multiple SLE curves. 
In~\cite{HuangPeltolaWuMultiradialSLEResamplingBP}, the authors derived resampling property and boundary perturbation of these two models. 
In the current article, we derive large deviation principle of these two models in Theorems~\ref{thm::halfwatermelon_LDP} and~\ref{thm::radialSLE_LDP}. 
The rate function for the large deviation is multi-time Loewner energy introduced in~\cite{ChenHuangPeltolaWumultitimeenergy}. 
As consequences of Theorems~\ref{thm::halfwatermelon_LDP} and~\ref{thm::radialSLE_LDP}, we provide an alternative proof for  large deviation for Dyson Brownian motion (see Propositions~\ref{prop::LDP_DysonBM_chordal} and~\ref{prop::LDP_DysonBM_radial}) and an alternative proof for boundary perturbation property of the multi-time Loewner energy (see Propositions~\ref{prop::bp_chordal} and~\ref{prop::bp_radial}). 

\subsection{Large deviation for half-watermelon SLE}

\paragraph{Curve Spaces. } For $2$-polygon $(\Omega; x, y)$, we denote by $\chamber(\Omega; x, y)$ the space of continuous simple curves $\eta\subset\overline{\Omega}$ from $x$ to $y$ such that $\eta\cap\partial\Omega=\{x,y\}$. In general, for $(n+1)$-polygon $(\Omega; \bs{x}, y)=(\Omega; x_1, \ldots, x_n, y)$, we denote by $\chamberfusion{n}(\Omega; \bs{x}, y)$ the space of $n$-tuples of continuous simple curves $\bs{\eta}=(\eta^1, \ldots, \eta^n)$ such that $\eta^j\in\chamber(\Omega; x_j, y)$ for all $1\le j\le n$ and that $\eta^i\cap\eta^j=\{y\}$ for all $i\neq j$, see Fig.~\ref{fig::multipleSLE}~(a). 

For $\bs{\eta}=(\eta^1, \ldots, \eta^n)\in\chamberfusion{n}(\Omega; \bs{x}, y)$, let $\varphi_{\HH}: \Omega\to \HH$ be a conformal map with $\varphi_{\HH}(y)=\infty$ and let $\varphi_{\U}: \Omega\to \U$ be a conformal map. We parameterize $\bs{\eta}\in\chamberfusion{n}(\Omega; \bs{x}, y)$ using $n$-time parameter of $\varphi_{\HH}(\bs{\eta})=(\varphi_{\HH}(\eta^1), \ldots, \varphi_{\HH}(\eta^n))$ (see~\eqref{eqn::multitime_def}). Define 
\begin{equation}\label{eqn::chamberfusion_metric}
\dist_{\chamber}(\bs{\eta}, \tilde{\bs{\eta}})= \sup_{\bs{t}\in [\bs{0}, \bs{\infty}]}\sup_{1\le j\le n}|\varphi_{\U}(\eta^j_{t_j})-\varphi_{\U}(\tilde{\eta}^j_{t_j})|, \qquad\text{for }\bs{\eta}, \tilde{\bs{\eta}}\in\chamberfusion{n}(\Omega; \bs{x}, y). 
\end{equation}
In this way, we obtain an incomplete metric space $(\chamberfusion{n}(\Omega; \bs{x}, y), \dist_{\chamber})$.
One example of curves in $\chamberfusion{n}(\Omega; \bs{x}, y)$ is half-watermelon SLE. 

\begin{definition}[Half-watermelon SLE]\label{def::halfwatermelonSLE}
Fix $\kappa\in (0,4]$ and $n\ge 1$ and $(n+1)$-polygon $(\Omega; \bs{x}, y) = (\Omega; x_1, \ldots, x_n, y)$.  
The half-$n$-watermelon $\SLE_{\kappa}$ is a probability measure on $\bs{\eta}=(\eta^1, \ldots, \eta^n)\in\chamberfusion{n}(\Omega; \bs{x}, y)$ which is characterized by the following recursion relation:
\begin{itemize}
\item $\eta^1$ is chordal $\SLE_{\kappa}(2, \ldots, 2)$ in $\Omega$ from $x_1$ to $y$ with force points $(x_2, \ldots, x_n)$; 
\item given $\eta^1$, the conditional law of $(\eta^2, \ldots, \eta^n)$ is half-$(n-1)$-watermelon $\SLE_{\kappa}$ in $(\Omega\setminus\eta^1; x_2, \ldots, x_n, y)$. 
\end{itemize}
We denote the law of half-$n$-watermelon $\SLE_{\kappa}$ in $(\Omega; \bs{x}, y )$ by $\PPfusion{n}(\Omega; \bs{x}, y)$. 
\end{definition}

It is proved in~\cite[Section~4]{MillerSheffieldIG2} that half-watermelon $\SLE_{\kappa}$ is the unique probability measure on $\bs{\eta}=(\eta^1,\ldots, \eta^n)\in \chamberfusion{n}(\Omega; \bs{x}, y)$ with the following resampling property: for each $j\in \{1, \ldots, n\}$, the conditional law of $\eta^{j}$ given $\{\eta^{i}: i\neq j\}$ is chordal $\SLE_{\kappa}$ from $x_j$ to $y$ 
in the connected component $\Omega_j$ of $\Omega\setminus\cup_{i\neq j}\eta^{i}$ having $x_j$ on its boundary. 
The solution $\LZ_{\shuffle_n}^{(\kappa)}$ in~\eqref{eqn::halfwatermelon_pf_H} is the partition function for half-$n$-watermelon $\SLE_{\kappa}$ in $(\HH; \bs{x}, \infty)$ with $\bs{x}=(x_1, \ldots, x_n)\in\LX_n^{\HH}$, in the sense that the driving function $W_t$ for $\eta^j$ solves the SDE
\begin{align}\label{eqn::halfwatermelon_marginal_SDE}
\begin{split}
\ud W_t=&\sqrt{\kappa}\ud B_t+\kappa\partial_j\left(\log\LZ_{\shuffle_n}^{(\kappa)}\right)(g_t(x_1), \ldots, g_t(x_{j-1}), W_t, g_t(x_{j+1}), \ldots, g_t(x_n))\ud t\\
=& \sqrt{\kappa}\ud B_t+\sum_{i\neq j}\frac{2\ud t}{W_t-g_t(x_i)},
\end{split}
\end{align}
where $B_t$ is standard one-dimensional Brownian motion (see Section~\ref{subsec::chordalSLE} for chordal Loewner chain). 

\begin{figure}[ht!]
\begin{subfigure}[t]{0.45\textwidth}
\begin{center}
\includegraphics[width=0.6\textwidth]{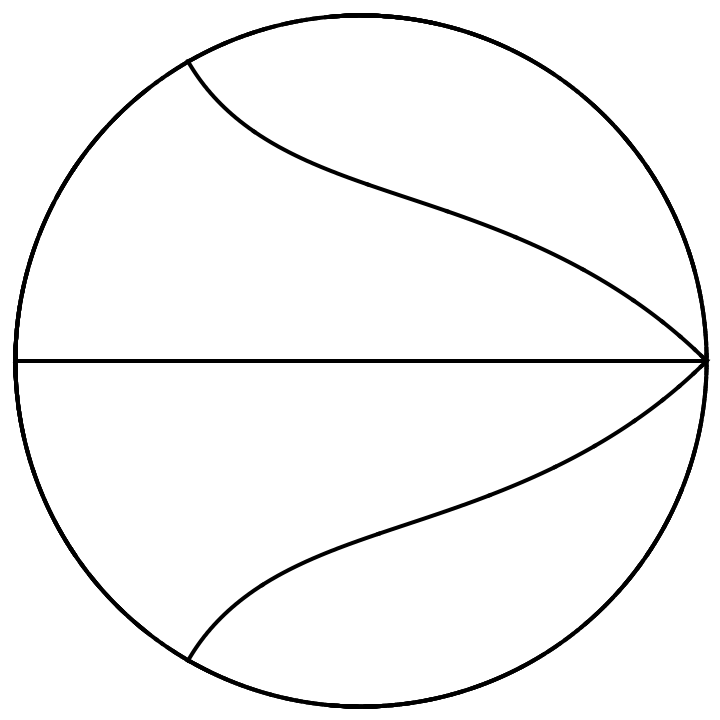}
\end{center}
\caption{Half-3-watermelon SLE$_0$.}
\end{subfigure}
\hfill
\begin{subfigure}[t]{0.45\textwidth}
\begin{center}
\includegraphics[width=0.6\textwidth]{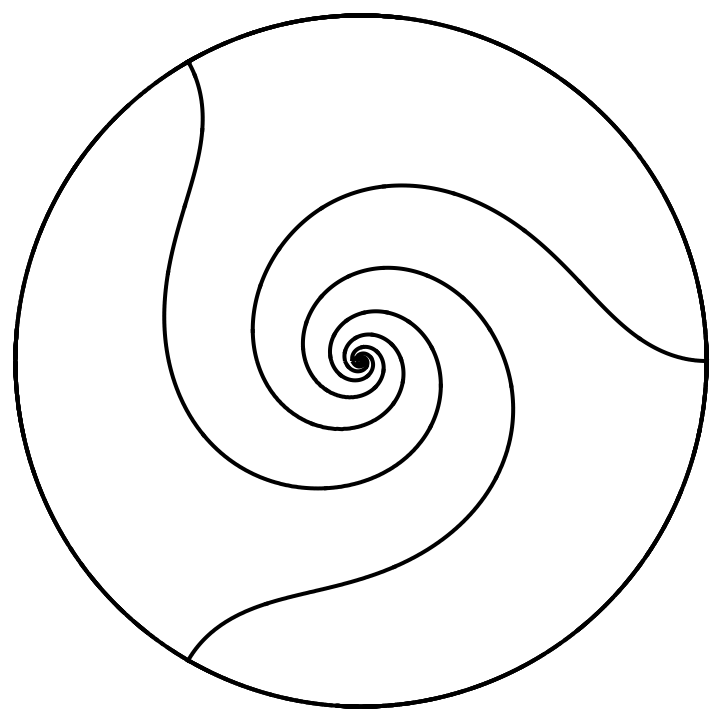}
\end{center}
\caption{3-radial SLE$_0$ with spiral. }
\end{subfigure}
\caption{\label{fig::multipleSLE} Examples of the curve spaces.}
\end{figure}

\begin{theorem}
\label{thm::halfwatermelon_LDP}
Fix $n\ge 1$ and $(n+1)$-polygon $(\Omega; \bs{x}, y)$. 
The family $\{ \PPfusion{n}(\Omega; \bs{x}, y)\}_{\kappa\in (0,4]}$ of laws of half-$n$-watermelon $\SLE_{\kappa}$   satisfies large deviation principle in the space $(\chamberfusion{n}(\Omega; \bs{x}, y), \dist_{\chamber})$ as $\kappa\to 0+$ with good rate function $\ratefusion{n}(\Omega; \bs{x}, y; \cdot)$ which will be defined in~\eqref{eqn::multi_time_energy_chordal} in Definition~\ref{Def::multitime_energy_chordal}, that is, 
\begin{align*}
	\liminf_{\kappa\to 0+} \kappa \log \PPfusion{n}(\Omega; \bs{x}, y) [O] \ge& -\inf_{\bs{\eta}\in O} \ratefusion{n}(\Omega; \bs{x}, y; \bs{\eta}), \quad \text{for any open }O\subset \chamberfusion{n}(\Omega; \bs{x}, y),\\
	\limsup_{\kappa\to 0+} \kappa \log \PPfusion{n}(\Omega; \bs{x}, y) [F] \le& -\inf_{\bs{\eta}\in F} \ratefusion{n}(\Omega; \bs{x}, y; \bs{\eta}), \quad \text{for any closed }F\subset \chamberfusion{n}(\Omega; \bs{x}, y).
\end{align*}
\end{theorem}

Large deviation principle for chordal SLE curve was first established in a weak, finite-time form in \cite{Wang:EnergyLoewnerChain}. This result was later strengthened to a large deviation principle for the entire curve with respect to the Hausdorff metric; see \cite[Theorem~1.5]{PeltolaWangSLELDP}.
The large deviation in Theorem~\ref{thm::halfwatermelon_LDP} when $n=1$ is proved in~\cite[Theorem~1.2]{AbuzaidPeltolaLargeDeviationCapacityParameterization}. 
We will prove it for $n\ge 1$ in Section~\ref{sec::LDP_halfwatermelon}.

\subsection{Large deviation for multi-radial SLE with spiral}
\paragraph{Curve Spaces. }
For $1$-polygon $(\Omega;x)$ with $z\in\Omega$, 
we denote by $\chamber(\Omega;x;z)$ the space of continuous simple curves $\gamma\subset \overline{\Omega}$ from $x$ to $z$, such that $\gamma\cap \partial\Omega = \{x\}$.
In general, for $n$-polygon $(\Omega; \bs{x})=(\Omega; x_1, \ldots, x_n)$ with $z\in\Omega$, we denote by $\chamberradial{n}(\Omega; \bs{x}; z)$ the space of $n$-tuples of continuous simple curves $\bs{\gamma}=(\gamma^1, \ldots, \gamma^n)$ such that $\gamma^j\in\chamber(\Omega; x_j; z)$ for all $1\le j\le n$ and that $\gamma^i\cap\gamma^j=\{z\}$ for all $i\neq j$, see Fig.~\ref{fig::multipleSLE}~(b). 

For $\bs{\gamma}=(\gamma^1, \ldots, \gamma^n)\in\chamberradial{n}(\Omega; \bs{x}; z)$, let $\varphi: \Omega\to \U$ be a conformal map with $\varphi(z)=0$. 
We parameterize $\bs{\gamma}\in \chamberradial{n}(\Omega; \bs{x}; z)$ using $n$-time parameter of $\varphi(\bs{\gamma})=(\varphi(\gamma^1), \ldots, \varphi(\gamma^n))$ (see~\eqref{eqn::multitime_def_radial}). Define
\begin{equation} \label{eqn::dist_chamber}
	\dist_{\chamber}(\bs{\gamma},\tilde{\bs{\gamma}})=\sup_{\bs{t}\in [\bs{0},\bs{\infty}]} \sup_{1\le j\le n} |\varphi(\gamma_{t_j}^j)-\varphi(\tilde{\gamma}_{t_j}^j)|,
\qquad \text{for }\bs{\gamma}, \tilde{\bs{\gamma}}\in\chamberradial{n}(\Omega; \bs{x}; z).
\end{equation}
In this way, we obtain an incomplete metric space $(\chamberradial{n}(\Omega; \bs{x}; z), \dist_{\chamber})$.
One example of curves in $\chamberradial{n}(\Omega; \bs{x}; z)$ is multi-radial SLE with spiral.

\begin{definition}[Multi-radial SLE with spiral]
\label{def::multiradialSLEspiral}
Fix $\kappa\in (0,4], \mu\in\R$ and $n\ge 1$ and $n$-polygon $(\Omega; \bs{x})=(\Omega; x_1, \ldots, x_n)$ with $z\in\Omega$. The $n$-radial $\SLE_{\kappa}^{\mu}$ is a probability measure on $\bs{\gamma}=(\gamma^1, \ldots, \gamma^n)\in\chamberradial{n}(\Omega; \bs{x}; z)$ which is characterized by the following properties:
\begin{itemize}
\item $\gamma^1$  is radial $\SLE_{\kappa}^{\mu}(2,\ldots,2)$ in $\Omega$ from $x_1$ to $z$ with force points $(x_2, \ldots, x_n)$; 
\item given $\gamma^1$, the conditional law of $(\gamma^2, \ldots, \gamma^n)$ is half-$(n-1)$-watermelon $\SLE_{\kappa}$ in $(\Omega\setminus\gamma^1; x_2, \ldots, x_n, z)$. 
\end{itemize}
We denote the law of $n$-radial $\SLE_{\kappa}^{\mu}$ in $(\Omega; \bs{x}; z)$ by $\PPradial{n}^{(\kappa; \mu)}(\Omega; \bs{x}; z)$. 
\end{definition}

The $n$-radial $\SLE_{\kappa}$ (without spiral) was originally constructed in~\cite{HealeyLawlerNSidedRadialSLE} by weighting independent radial $\SLE_{\kappa}$'s by common-time martingale. Later, the $n$-radial $\SLE_{\kappa}^{\mu}$ was constructed in~\cite[Section~3]{KrusellWangWuCommutationRelation} for $n=2$ and in~\cite{HuangPeltolaWuMultiradialSLEResamplingBP} for $n\ge 2$ as a process by weighting independent radial $\SLE_{\kappa}$'s by multi-time martingale (see \eqref{eqn::nradial_multitime_mart}). The definition we present here is an equivalent description proved in~\cite[Theorem~1.3]{HuangPeltolaWuMultiradialSLEResamplingBP}.
 The solution to radial BPZ equations $\LZradial{n}^{(\kappa; \mu)}$ in~\eqref{eqn::nradial_pf_U} is the partition function for $n$-radial $\SLE_{\kappa}^{\mu}$ in $(\U; \ee^{\ii\bs{\theta}}; 0)$ with $\bs{\theta}=(\theta_1, \ldots, \theta_n)\in\LX_n^{\U}$, in the sense that the driving function $\xi_t$ for $\gamma^j$ solves the SDE
 \begin{align}\label{eqn::multiradial_marginal_SDE}
 \begin{split}
\ud \xi_t=&\sqrt{\kappa}\ud B_t+\kappa\partial_j\left(\log\LZradial{n}^{(\kappa; \mu)}\right)(\covmap_t(\theta_1), \ldots, \covmap_t(\theta_{j-1}), \xi_t, \covmap_t(\theta_{j+1}), \ldots, \covmap_t(\theta_n))\ud t\\
 =&\sqrt{\kappa}\ud B_t+\sum_{i\neq j}\cot\left(\frac{\xi_t-\covmap_t(\theta_i)}{2}\right)\ud t+\mu\ud t, 
 \end{split}
 \end{align}
 where $B_t$ is standard one-dimensional Brorwnian motion (see Section~\ref{subsec::pre_radialLoewner} for radial Loewner chain). 
 

\begin{theorem}
\label{thm::radialSLE_LDP}
Fix $\mu\in\R$ and $n\ge 1$ and $n$-polygon $(\Omega; \bs{x})$ with $z\in\Omega$. 
The family $\{\PPradial{n}^{(\kappa; \mu)}(\Omega; \bs{x}; z)\}_{\kappa\in (0,4]}$ of laws of $n$-radial $\SLE_{\kappa}^{\mu}$
 satisfies large deviation principle in the space $(\chamberradial{n}(\Omega; \bs{x}; z), \dist_{\chamber})$ as $\kappa\to 0+$ with good rate function $\rateradial{n}^{(\mu)}(\Omega; \bs{x}; z; \cdot)$ which will be defined in~\eqref{eqn::multi_time_energy_radial} in Definition~\ref{def::multitime_energy_radial}, that is,
\begin{align*}
	\liminf_{\kappa\to 0+} \kappa \log \PPradial{n}^{(\kappa; \mu)}(\Omega; \bs{x}; z) [O] \ge& -\inf_{\bs{\gamma}\in O} \rateradial{n}^{(\mu)}(\Omega; \bs{x}; z; \bs{\gamma}), \quad \text{for any open }O\subset \chamberradial{n}(\Omega; \bs{x}; z),\\
	\limsup_{\kappa\to 0+} \kappa \log \PPradial{n}^{(\kappa; \mu)}(\Omega; \bs{x}; z) [F] \le& -\inf_{\bs{\gamma}\in F} \rateradial{n}^{(\mu)}(\Omega; \bs{x}; z; \bs{\gamma}), \quad \text{for any closed }F\subset \chamberradial{n}(\Omega; \bs{x}; z). 
\end{align*}
\end{theorem}

The large deviation Theorem~\ref{thm::radialSLE_LDP} when $n=1, \mu=0$ is proved in~\cite[Theorem~1.2]{AbuzaidPeltolaLargeDeviationCapacityParameterization}. We will prove it for $n\ge 1$ and general $\mu\in\R$ in Section~\ref{sec::LDP_radialSLE}.
\medbreak
The proof for large deviation in Theorems~\ref{thm::halfwatermelon_LDP} and~\ref{thm::radialSLE_LDP} has two main ingredients. 
\begin{itemize}
\item First, we derive return estimate for SLE in Propositions~\ref{prop::SLE_return_chordal} and~\ref{prop::radialSLE_return}. The proof for the return estimates 
consists of two parts: we split the estimates into estimates on ``good event" and estimates on ``bad event". 
On the good event, the return estimates are comparable to the return estimates for SLE without force points or spiral, which was proved in~\cite{field2015escape,LawlerMinkowskiSLERealLine,PeltolaWangSLELDP} for the chordal case and in~\cite{LawlerContinuityofradialSLE,AbuzaidPeltolaLargeDeviationCapacityParameterization} for the radial case. The control of the ``bad event'' essentially reduces to establishing the transience of an associated Bessel-type process.
\item Second, we establish a large deviation principle for finite-time setting in Propositions~\ref{prop::finite_watermelon_LDP} and~\ref{prop::finitetime_LDP_radial}. The proof for Propositions~\ref{prop::finite_watermelon_LDP} and~\ref{prop::finitetime_LDP_radial} relies on multi-time martingales developed in~\cite{HuangWuYangMultipleSLEsDysonBM, HuangPeltolaWuMultiradialSLEResamplingBP} and a generalized Varadhan's lemma. 
\end{itemize}
With these two ingredients at hand, we extend the large deviation principle from finite-time setting to the entire curve using a generalized contraction principle developed in~\cite{AbuzaidPeltolaLargeDeviationCapacityParameterization}.

\subsection{Large deviation for Dyson Brownian motion}
In this section, we give a first application of the large deviation principle. 
By changing the multi-time parameter to common-time parameter, the driving function of half-watermelon $\SLE_{\kappa}$ becomes Dyson Brownian motion~\cite[Lemma~3.4]{HuangWuYangMultipleSLEsDysonBM} (see Section~\ref{subsec::common_time_chordal}), and the driving function of multi-radial $\SLE_{\kappa}^{\mu}$ becomes Dyson circular ensemble (see Section~\ref{subsec::common_time_radial}). Thus, we recover large deviation principle for Dyson Brownian motion and Dyson circular ensemble from Theorems~\ref{thm::halfwatermelon_LDP} and~\ref{thm::radialSLE_LDP}. 

\paragraph*{Dyson Brownian motion.}
Fix $\kappa\in (0,4]$. 
Dyson Brownian motion with parameter $\beta = 8/\kappa$ is the solution to the SDE system:
\begin{equation}\label{eqn::DysonBM}
	\ud X_t^j = \sqrt{\kappa} \, \ud B_t^j + \sum_{i \neq j} \frac{4}{X_t^j - X_t^i} \, \ud t, \quad 1 \leq j \leq n,
\end{equation}
where $\bs{X}_0 \in \LX_n^{\HH}$, and $\{B^j\}_{1\leq j \leq n}$ are independent standard one-dimensional Brownian motions. When $\kappa \le 4$, we have $\beta \ge 2$ and the solution to \eqref{eqn::DysonBM} exists for all time, see Fig.~\ref{fig::DysonBM}~(a) and~(c) for simulations of Dyson Brownian motion.
Fix $T\in (0,\infty)$. We denote by $\mathrm{C}([0,T], \LX_n^{\HH})$  the space of continuous functions from $[0,T]$ to $\LX_n^{\HH}$ equipped with the topology of uniform convergence.
\begin{proposition}\label{prop::LDP_DysonBM_chordal}
Fix $T\in (0,\infty)$, $n\ge 1$ and $\bs{x}=(x_1, \ldots, x_n)\in\LX_n^{\HH}$.
Let $\bs{X}^\kappa$ be the unique strong solution to the SDE system~\eqref{eqn::DysonBM} with initial data $\bs{X}^\kappa_0=\bs{x}$. The family of laws of Dyson Brownian motion $\{\bs{X}^{\kappa}\}_{\kappa\in (0,4]}$ satisfies large deviation principle in the space $\mathrm{C}([0,T], \LX_n^{\HH})$ as $\kappa\to 0+$ with good rate function
\begin{equation}\label{eqn::rate_DysonBM}
\LIfusion{n}(\bs{X}_{[0,T]}):= \frac{1}{2} \int_0^T \sum_{j=1}^n \left( \dot{X}^j_t - \sum_{i\neq j} \frac{4}{X^j_t - X^i_t} \right)^2 \ud t,
\end{equation}
if $\bs{X}_t=(X^1_t, \ldots, X^n_t)\in \LX_n^{\HH}$ is absolutely continuous on $[0,T]$, and $\LIfusion{n}(\bs{X}_{[0,T]})=\infty$ otherwise. 
\end{proposition}

\begin{figure}[ht!]
\begin{subfigure}[t]{0.45\textwidth}
\begin{center}
\includegraphics[width=0.7\textwidth]{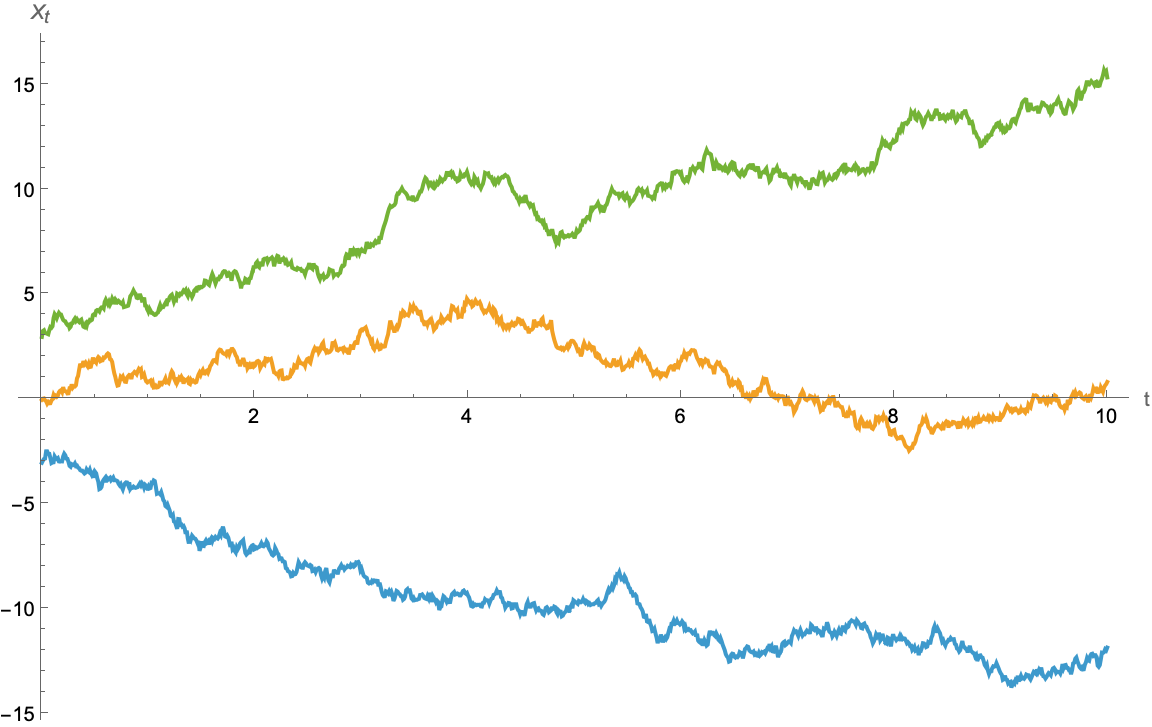}
\end{center}
\caption{Simulation for Dyson Brownian motion~\eqref{eqn::DysonBM} with $\kappa=2$ and $n=3$.}
\end{subfigure}
\hfill
\begin{subfigure}[t]{0.45\textwidth}
\begin{center}
\includegraphics[width=0.7\textwidth]{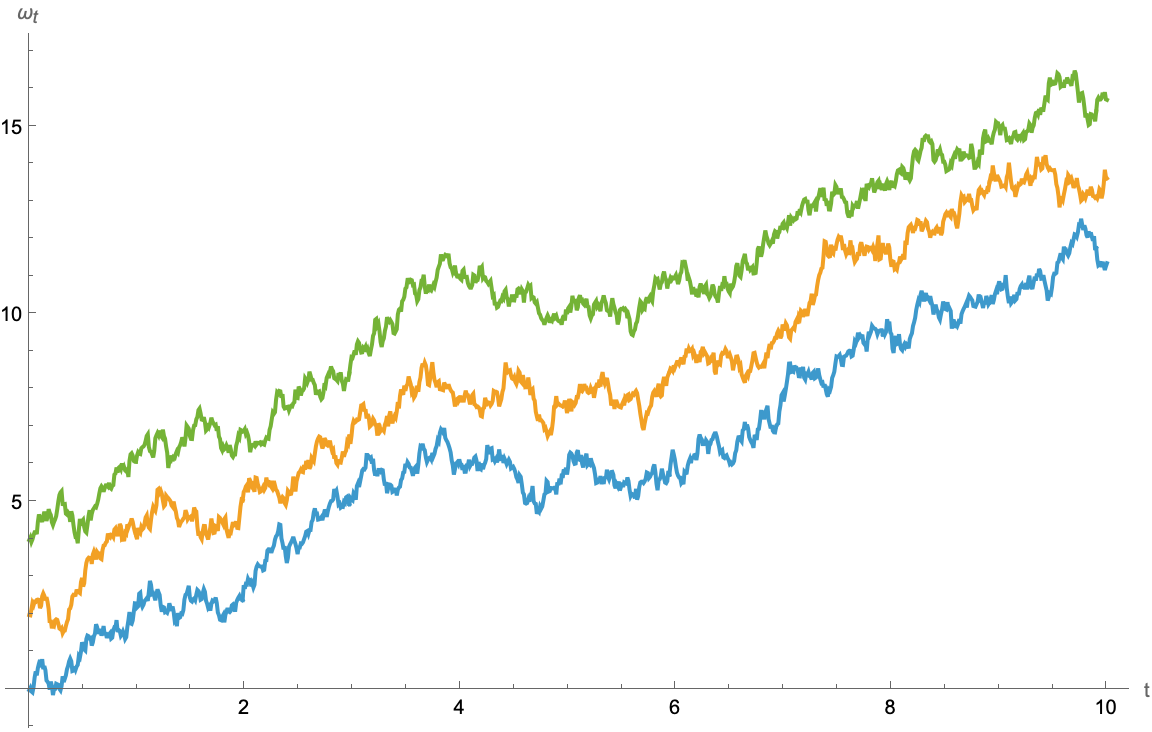}
\end{center}
\caption{Simulation for Dyson circular ensemble~\eqref{eqn::DysonBM_radial} with $\kappa=2$ and $n=3$ and $\mu=1$.}
\end{subfigure}\\
\begin{subfigure}[t]{0.45\textwidth}
\begin{center}
\includegraphics[width=0.7\textwidth]{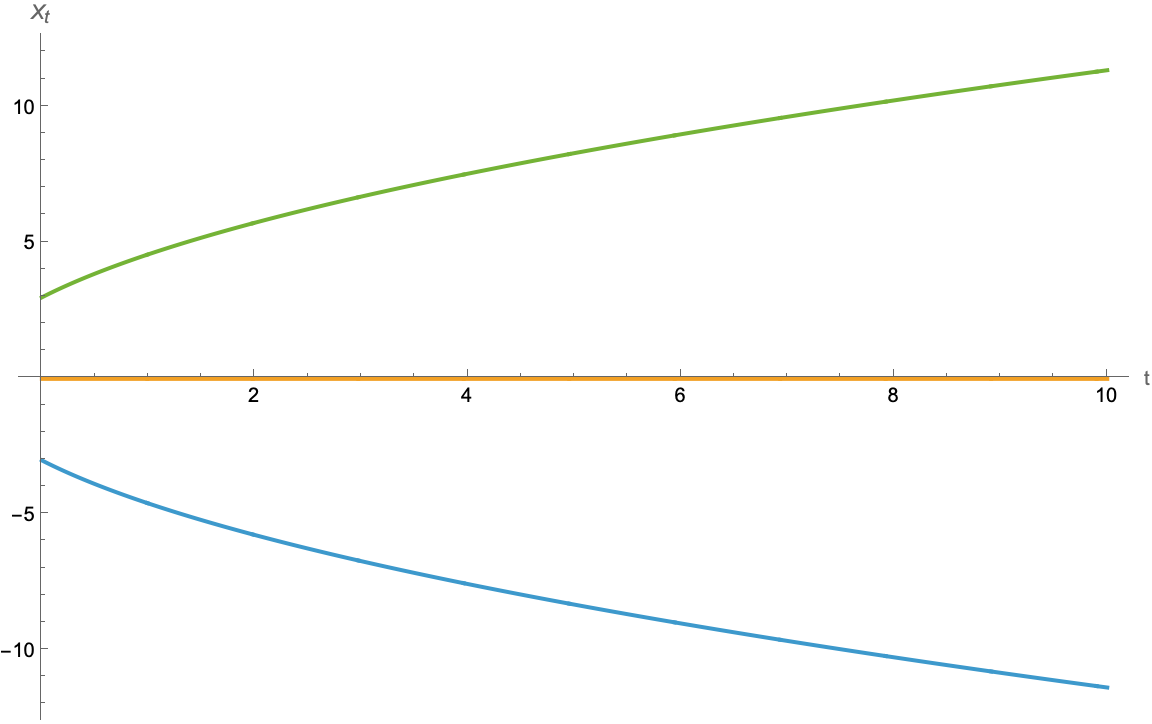}
\end{center}
\caption{Simulation for Dyson Brownian motion~\eqref{eqn::DysonBM} with $\kappa=0$ and $n=3$.}
\end{subfigure}
\hfill
\begin{subfigure}[t]{0.45\textwidth}
\begin{center}
\includegraphics[width=0.7\textwidth]{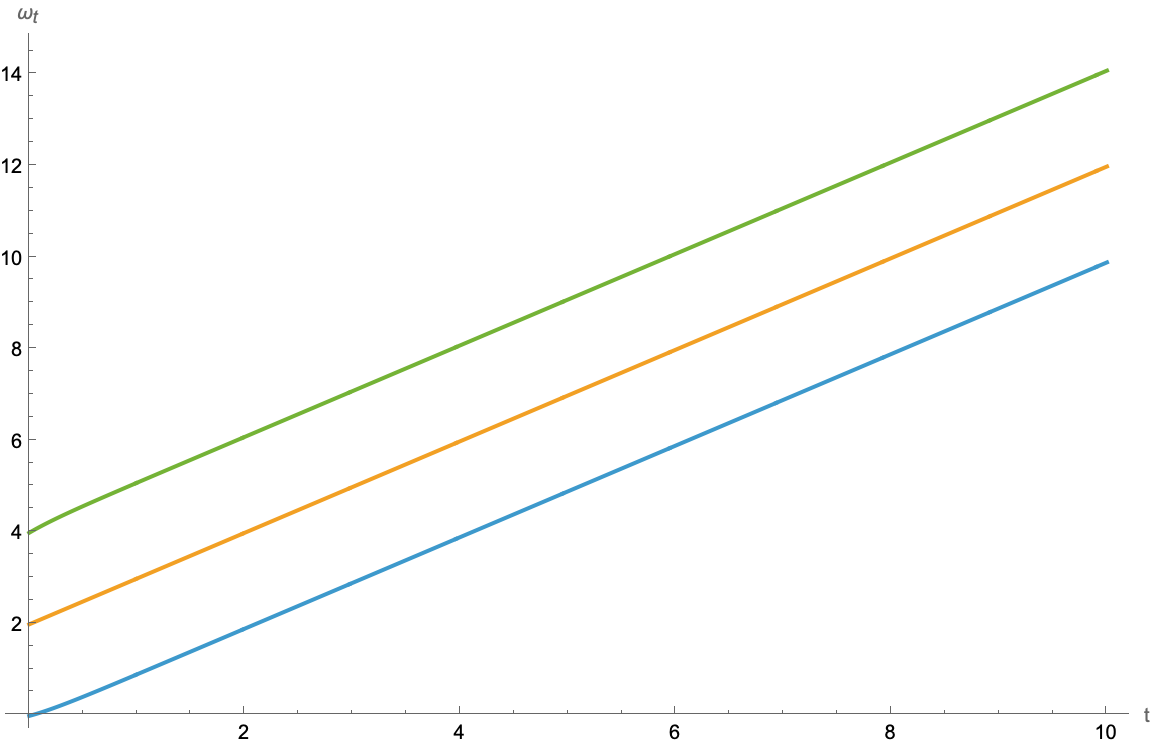}
\end{center}
\caption{Simulation for Dyson circular ensemble~\eqref{eqn::DysonBM_radial} with $\kappa=0$ and $n=3$ and $\mu=1$.}
\end{subfigure}
\caption{\label{fig::DysonBM} Simulations of Dyson Brownian motion and Dyson circular ensemble.}
\end{figure}

\paragraph*{Dyson circular ensemble.}
Fix $\kappa\in (0,4]$ and $\mu\in\R$. 
Dyson circular ensemble with parameter $\beta = 8/\kappa$ and spiraling rate $\mu$ is the solution to the SDE system:
\begin{equation}\label{eqn::DysonBM_radial}
	\ud \mslitdriv_t^j = \sqrt{\kappa} \, \ud B_t^j + \left( \mu + 2 \sum_{i \neq j} \cot \left( \frac{\mslitdriv_t^j-\mslitdriv_t^i}{2} \right) \right) \, \ud t, \quad 1 \leq j \leq n,
\end{equation}
where $\bs{\mslitdriv}_0\in \LX_n^{\U}$, and $\{B^j\}_{1\leq j \leq n}$ are independent standard one-dimensional Brownian motions. When $\kappa \le 4$, we have $\beta \ge 2$ and the solution to~\eqref{eqn::DysonBM_radial} exists for all time, see Fig.~\ref{fig::DysonBM}~(b) and~(d) for simulations of Dyson circular ensemble. 
Fix $T\in (0,\infty)$. We denote by $\mathrm{C}([0,T], \LX_n^{\U})$  the space of continuous functions from $[0,T]$ to $\LX_n^{\U}$ equipped with the topology of uniform convergence.

\begin{proposition}\label{prop::LDP_DysonBM_radial}
Fix $\mu\in \R$ and $T\in (0,\infty)$. 
Fix $n\ge 1$ and $\bs{\theta}=(\theta_1, \ldots, \theta_n)\in\LX_n^{\U}$. 
Let $\bs{\mslitdriv}^\kappa$ be the unique strong solution to the SDE system~\eqref{eqn::DysonBM_radial} with initial data $\bs{\mslitdriv}^\kappa_0=\bs{\theta}$. The family of laws of Dyson circular ensemble $\{\bs{\mslitdriv}^\kappa\}_{\kappa\in (0,4]}$ satisfies large deviation principle in the space $\mathrm{C}([0,T], \LX_n^{\U})$ as $\kappa\to 0+$ with good rate function
\begin{equation}\label{eqn::rate_DysonBM_radial}
\LIradial{n}^{(\mu)}(\bs{\mslitdriv}_{[0,T]})=\frac{1}{2}\int_{0}^{T}\sum_{j=1}^{n} \left( \dot{\mslitdriv}_t^j-  \mu- 2\sum_{i\neq j} \cot\left(\frac{\mslitdriv_{t}^j-\mslitdriv_{t}^{i}}{2}\right)  \right)^2 \ud t, 
\end{equation}
if $\bs{\mslitdriv}_t=(\mslitdriv^1_t, \ldots, \mslitdriv^n_t)$ is absolutely continuous on $[0,T]$, and $\LIradial{n}^{(\mu)}(\bs{\mslitdriv}_{[0,T]})=\infty$ otherwise. 
\end{proposition}

Proposition~\ref{prop::LDP_DysonBM_chordal} is possibly known in literature, but we did not find a proper reference for it. 
Proposition~\ref{prop::LDP_DysonBM_radial} is proved in~\cite{AbuzaidHealeyPeltolaLargeDeviationDysonBM} under the general framework of large deviations for Dyson-type diffusions on the unit circle. In Sections~\ref{subsec::common_time_chordal} and~\ref{subsec::common_time_radial}, we present an alternative proof of these propositions using the large deviation principles in Theorems~\ref{thm::halfwatermelon_LDP} and~\ref{thm::radialSLE_LDP}.

\subsection{Boundary perturbation of the energy}
Another application of the large deviation principle is boundary perturbation property of the energies in Theorems~\ref{thm::halfwatermelon_LDP} and~\ref{thm::radialSLE_LDP}. 

\begin{proposition}
\label{prop::bp_chordal}
Fix $n\ge 1$ and $(n+1)$-polygon $(\Omega; \bs{x}, y)= (\Omega; x_1, \ldots, x_n, y)$. 
Let $U\subset\Omega$ be a simply connected subdomain which coincides with $\Omega$ in a neighborhood of $\{x_1, \ldots, x_n, y\}$. The multi-time Loewner energy~\eqref{eqn::multi_time_energy_chordal} has the following boundary perturbation: 
for $\bs{\eta}\in\chamberfusion{n}(U; \bs{x}, y)$, we have
\begin{equation}\label{eqn::bp_choral}
\ratefusion{n}(U; \bs{x}, y; \bs{\eta})-\ratefusion{n}(\Omega; \bs{x}, y; \bs{\eta})=\LUfusion{n}(\Omega; \bs{x}, y) - \LUfusion{n}(U; \bs{x}, y) + 12 \blm(\Omega;\bs{\eta},\Omega\setminus U),
\end{equation}
where $\LUfusion{n}$ is given by 
\begin{align}\label{eqn::LUfusion_def}
\begin{split}
	\LUfusion{n}(\Omega;\bs{x},y)
	=& \sum_{1\le i<j\le n} \log \Poisson(\Omega;x_i,x_j) - (n+2) \sum_{\ell=1}^{n} \log \Poisson(\Omega;x_{\ell},y),
	\end{split}
\end{align}
and $\Poisson(\Omega; x, y)$ is boundary Poisson kernel (see~\eqref{eqn::bPoisson_H}-\eqref{eqn::bPoisson_cov}), 
and $\blm(\Omega;\bs{\eta},\Omega\setminus U)$ is the Brownian loop measure of the loops in $\Omega$ that intersect both $\bs{\eta}$ and $\Omega\setminus U$ (see~\eqref{eqn::blm_def}).
\end{proposition}

\begin{proposition}
\label{prop::bp_radial}
Fix $\mu\in\R$ and $n\ge 1$ and $n$-polygon $(\Omega; \bs{x})=(\Omega; x_1, \ldots, x_n)$ with $z\in\Omega$. 
Let $U\subset\Omega$ be a simply connected subdomain which coincides with $\Omega$ in a neighborhood of $\{x_1, \ldots, x_n, z\}$. The multi-time Loewner energy~\eqref{eqn::multi_time_energy_radial} has the following boundary perturbation: 
for $\bs{\gamma}\in\chamberradial{n}(U; \bs{x}, z)$, we have
\begin{equation}\label{eqn::bp_radial}
\rateradial{n}^{(\mu)}(U; \bs{x}, z; \bs{\gamma})-\rateradial{n}^{(\mu)}(\Omega; \bs{x}, z; \bs{\gamma})=\LVradial{n}^{(\mu)}(\Omega; \bs{x}, z) - \LVradial{n}^{(\mu)}(U; \bs{x}, z) + 12 \blm(\Omega;\bs{\gamma},\Omega\setminus U),
\end{equation}
where $\LVradial{n}^{(\mu)}$ is given by 
\begin{align}\label{eqn::LVradial_def}
\begin{split}
	\LVradial{n}^{(\mu)}(\Omega;\bs{x};z):=& \frac{n^2-4-\mu^2}{2} \log \CR(\Omega;z) - (n+2) \sum_{j=1}^{n} \log \Poisson(\Omega;x_j;z) + \sum_{1\le i<j\le n} \log \Poisson(\Omega;x_i,x_j)\\
	&  +\mu n \arg \varphi'(z) - \mu \sum_{j=1}^n \arg \varphi(x_j), 
\end{split}
\end{align}
and $\CR(\Omega; z)$ is conformal radius of $\Omega$ seen from $z$, 
and $\Poisson(\Omega; x_i, z)$ is Poisson kernel (see~\eqref{eqn::Poisson_H_inner}-\eqref{eqn::Poisson_polygon_inner}), 
and $\Poisson(\Omega; x_i, x_j)$ is boundary Poisson kernel (see~\eqref{eqn::bPoisson_H}-\eqref{eqn::bPoisson_cov}), 
and $\varphi$ is any conformal map from $\Omega$ onto $\U$ with $\varphi(z)=0$, and $\blm(\Omega;\bs{\gamma},\Omega\setminus U)$ is the Brownian loop measure of the loops in $\Omega$ that intersect both $\bs{\gamma}$ and $\Omega\setminus U$ (see~\eqref{eqn::blm_def}).
\end{proposition}

Proposition~\ref{prop::bp_chordal} when $n=1$ is proved in~\cite[Proposition~4.1]{Wang:EnergyLoewnerChain}. 
Both Propositions~\ref{prop::bp_chordal} and~\ref{prop::bp_radial} are proved in~\cite{ChenHuangPeltolaWumultitimeenergy} using direct analysis of the energy. 
In Sections~\ref{subsec::bp_chordal_proof} and~\ref{subsec::bp_radial_proof}, we will give an alternative proof of these two propositions using boundary perturbation property of multiple $\SLE$ proved in~\cite{HuangPeltolaWuMultiradialSLEResamplingBP} (see Lemmas~\ref{lem::halfwatermelon_bp} and~\ref{lem::bp_multiradialkappa}) and the large deviation principles in Theorems~\ref{thm::halfwatermelon_LDP} and~\ref{thm::radialSLE_LDP}. The quantities $\LUfusion{n}$ in~\eqref{eqn::LUfusion_def} and $\LVradial{n}^{(\mu)}$ in~\eqref{eqn::LVradial_def} are semi-classical limits of partition functions for half-watermelon SLE and for multi-radial SLE with spiral, see~\eqref{eqn::LUfusion_sclimit} and~\eqref{eqn::LVradial_sclimit}.

\paragraph*{Outline.} 
In Section~\ref{sec::pre}, we give preliminaries on Poisson kernels and Brownian measures; and we also give estimates on two Bessel-type processes which will be used in the proof of return estimates in Propositions~\ref{prop::SLE_return_chordal} and~\ref{prop::radialSLE_return}. 
In Section~\ref{sec::chordalSLE_return}, we give preliminaries on chordal Loewner chain and prove return estimate for chordal SLE with force points in Proposition~\ref{prop::SLE_return_chordal}. 
In Section~\ref{sec::LDP_halfwatermelon}, we define multi-time chordal Loewner energy and prove Theorem~\ref{thm::halfwatermelon_LDP} and Propositions~\ref{prop::LDP_DysonBM_chordal} and~\ref{prop::bp_chordal}. 
In Section~\ref{sec::return_radial}, we give preliminaries on radial Loewner chain and prove return estimate for radial SLE with spiral in Proposition~\ref{prop::radialSLE_return}. 
In Section~\ref{sec::LDP_radialSLE}, we define multi-time radial Loewner energy and prove Theorem~\ref{thm::radialSLE_LDP} and Propositions~\ref{prop::LDP_DysonBM_radial} and~\ref{prop::bp_radial}. 

\paragraph*{Acknowledgment.}
We thank Eveliina Peltola for helpful discussion on large deviation principle. 
All three authors are supported by Beijing Natural Science Foundation (JQ20001). 
Part of this research was performed while the three authors participated in a program hosted by the Hausdorff Research Institute for Mathematics (HIM), which is supported by the Deutsche Forschungsgemeinschaft (DFG, German Research Foundation) under Germany’s Excellence Strategy EXC-2047/1-390685813. 
H. W. is also supported by New Cornerstone Investigator Program 100001127. H.W. is partly affiliated at Yanqi Lake Beijing Institute of Mathematical Sciences and Applications, Beijing, China.

\section{Preliminaries}
\label{sec::pre}
In this section, we first give preliminaries on Poisson kernels and Brownian measures in Section~\ref{subsec::pre_Poisson}; 
and then give estimates on two Bessel-type processes in Section~\ref{subsec::pre_Bessel}. 
The estimate in Lemma~\ref{lem::chordalBessel_estimate} will be used in the proof of the return estimate for chordal SLE in Proposition~\ref{prop::SLE_return_chordal}. 
The estimate in Lemma~\ref{lem::radialBessel_estimates} will be used in the proof of return estimate for radial SLE in Proposition~\ref{prop::radialSLE_return}. 
\subsection{Poisson kernels and Brownian measures}
\label{subsec::pre_Poisson}
\paragraph*{Poisson kernel.}
For 1-polygon $(\HH;x)$ with $z\in\HH$, the Poisson kernel is defined as 
\begin{equation} \label{eqn::Poisson_H_inner}
	\Poisson(\HH;x;z):=\frac{2\Im(z)}{|z-x|^2}, \quad x\in \R,z\in \HH;
\end{equation}
and for general nice 1-polygon $(\Omega;x)$ with $z\in\Omega$, we extend its definition via conformal covariance:
\begin{equation} \label{eqn::Poisson_polygon_inner}
	\Poisson(\Omega;x;z):= |\varphi'(x)| \Poisson(\HH;\varphi(x);\varphi(z)),
\end{equation}
where $\varphi$ is any conformal map from $\Omega$ to $\HH$. 

\paragraph*{Boundary Poisson kernel.}
 For $2$-polygon $(\HH;x,y)$, the boundary Poisson kernel is defined as 
\begin{equation} \label{eqn::bPoisson_H}
	\Poisson(\HH;x,y):=\frac{1}{(y-x)^2}, \quad x,y\in \R;
\end{equation}
and for general nice 2-polygon $(\Omega;x,y)$, we extend its definition via conformal covariance:
\begin{equation} \label{eqn::bPoisson_cov}
	\Poisson(\Omega;x,y):= |\varphi'(x)| |\varphi'(y)| \Poisson(\HH;\varphi(x),\varphi(y)),
\end{equation}
where $\varphi$ is any conformal map from $\Omega$ to $\HH$.

\paragraph*{Partition function for half-watermelon SLE.} 
Fix nice $(n+1)$-polygon $(\Omega; \bs{x}, y)$, the partition function for  half-$n$-watermelon $\SLE_\kappa$ in $(\Omega;\bs{x},y)$ is defined to be
\begin{equation} \label{eqn::halfwatermelon_pf_def}
	\LZfusion{n}(\Omega;\bs{x},y):= \left(\prod_{1\le i<j\le n} \Poisson(\Omega;x_i,x_j)^{-\frac{1}{\kappa}} \right) \times \left( \prod_{\ell=1}^{n} \Poisson(\Omega;x_{\ell},y)^{\frac{n+2}{\kappa}-\frac{1}{2}} \right).
\end{equation}
The semi-classical limit of $\LZfusion{n}$ is $\LUfusion{n}$ defined by~\eqref{eqn::LUfusion_def}, i.e. 
\begin{align}\label{eqn::LUfusion_sclimit}
	\LUfusion{n}(\Omega;\bs{x},y)=&-\lim_{\kappa\to 0}\kappa\log\LZfusion{n}(\Omega; \bs{x}, y). 
\end{align}
Note that 
$\LZfusion{n}(\HH; \bs{x}, y)$ is well-defined for nice $(n+1)$-polygon $(\HH; \bs{x}, y)$ when $y$ is finite. However, it degenerates when $y\to\infty$. In order to define partition function for half-watermelon $\SLE_{\kappa}$ in $(\HH; \bs{x}, \infty)$, we need to renormalize it properly. The proper renormalization results in $\LZ_{\shuffle_n}^{(\kappa)}$ given by~\eqref{eqn::halfwatermelon_pf_H}.

\paragraph*{Partition function for multi-radial SLE with spiral.}
Fix nice $n$-polygon $(\Omega;\bs{x})$ with $z\in \Omega$.  The partition function for $n$-radial $\SLE_{\kappa}^{\mu}$ in $(\Omega; \bs{x}; z)$ is defined to be
\begin{equation} \label{eqn::nradial_pf}
	\LZradial{n}^{(\kappa; \mu)}(\Omega; \bs{x}; z)= \LZradial{n}^{(\kappa; 0)}(\Omega; \bs{x}; z) \times \CR(\Omega;z)^{\frac{\mu^2}{2\kappa}} \exp\left( -\frac{\mu}{\kappa} n \arg \varphi'(z) + \frac{\mu}{\kappa} \sum_{j=1}^{n} \arg \varphi(x_j) \right),
\end{equation}
with
\begin{equation*}\label{eqn::nradial_pf_nospiral}
	\LZradial{n}^{(\kappa; 0)}(\Omega; \bs{x}; z)= \CR(\Omega;z)^{\frac{(\kappa-4)^2-4n^2}{8\kappa}}\times \prod_{j=1}^n \Poisson(\Omega;x_j;z)^{\frac{2n+4-\kappa}{2\kappa}}\times \prod_{1\le i<j\le n} \Poisson(\Omega;x_i,x_j)^{-\frac{1}{\kappa}}  ,
\end{equation*}
where $\CR(\Omega; z)$ is conformal radius of $\Omega$ seen from $z$, 
and $\Poisson(\Omega; x_i, z)$ is Poisson kernel~\eqref{eqn::Poisson_H_inner}-\eqref{eqn::Poisson_polygon_inner}, 
and $\Poisson(\Omega; x_i, x_j)$ is boundary Poisson kernel~\eqref{eqn::bPoisson_H}-\eqref{eqn::bPoisson_cov}, 
and $\varphi$ is any conformal map from $\Omega$ onto $\U$ with $\varphi(z)=0$. 
The semi-classical limit of $\LZradial{n}^{(\kappa, \mu)}$ is $\LVradial{n}^{(\mu)}$ defined by~\eqref{eqn::LVradial_def}, i.e. 
\begin{align}\label{eqn::LVradial_sclimit}
	\LVradial{n}^{(\mu)}(\Omega;\bs{x};z):=& - \lim_{\kappa\to 0} \kappa \log \LZradial{n}^{(\kappa; \mu)}(\Omega; \bs{x}; z). 
\end{align}
When $(\Omega; \bs{x}; z)=(\U; \ee^{\ii\bs{\theta}}; 0)$, the partition function $\LZradial{n}^{(\kappa; \mu)}(\U; \ee^{\ii\bs{\theta}}; 0)$ becomes the one in~\eqref{eqn::nradial_pf_U}. 

\paragraph*{Brownian excursion measure.}
Consider a domain $\Omega\subsetneq \C$ with smooth boundary. For two disjoint subsets $\Gamma_1,\Gamma_2\subset \partial \Omega$ of the boundary, the Brownian excursion measure between $\Gamma_1$ and $\Gamma_2$ in $\Omega$ is defined as 
\begin{equation} \label{eqn::Def_Brownian excursion measure}
	\LE_{\Omega}(\Gamma_1,\Gamma_2):=\int_{\Gamma_1} \int_{\Gamma_2} \Poisson(\Omega;x,y) |\ud x| |\ud y|,
\end{equation}
where $\Poisson(\Omega;x,y)$ is the boundary Poisson kernel. The Brownian excursion measure is conformally invariant (see~\cite[Proposition~5.8]{LawlerConformallyInvariantProcesses}). Thus the Brownian excursion measure is also well defined when $\partial \Omega$ is not smooth. For two disjoint sets $\Gamma_1$ and $\Gamma_2$, we denote $\LE_{\Omega\setminus (\Gamma_1\cup\Gamma_2)}(\Gamma_1,\Gamma_2)$ by $\LE_{\Omega}(\Gamma_1,\Gamma_2)$, where $\Omega\setminus (\Gamma_1\cup\Gamma_2)$ is the connected component of $\Omega\setminus (\Gamma_1\cup\Gamma_2)$ that intersects both $\partial \Gamma_1$ and $\partial \Gamma_2$. We refer to~\cite[Section~5]{LawlerConformallyInvariantProcesses} and~\cite[Page~4 and Lemma~3.3]{field2015escape} for its useful properties.

\paragraph*{Brownian loop measure.}
Brownian loop measure $\blm^\mathrm{loop}$ is a $\sigma$-finite measure on planar unrooted Brownian loops
--- see ~\cite{LawlerWernerBrownianLoopsoup} for its definition and properties. 
While the total mass of $\blm^\mathrm{loop}$ is infinite, the mass on macroscopic loops is finite: 
if $\Omega$ is a domain and $K_1, K_2 \subset \overline{\Omega}$ are two disjoint compact subsets, 
then the total mass $\blm(\Omega; K_1, K_2)$ of Brownian loops that stay in $\Omega$ and intersect both $K_1$ and $K_2$ is finite.
In general, for $n\ge 2$ disjoint compact subsets $K_1, \ldots, K_n$ of $\overline{\Omega}$, we denote
\begin{align}\label{eqn::blm_def}
	\blm(\Omega; K_1, \ldots, K_n) := \sum_{j=2}^n \blm^{\mathrm{loop}} \big[ \ell\subset\Omega: \ell\cap K_i\neq \emptyset\textnormal{ for at least }j \textnormal{ of the }i\in\{1, \ldots, n\} \big].
\end{align}
See~\cite{LawlerPartitionFunctionsSLE,PeltolaWangSLELDP} for more properties and~\cite{DubedatEulerIntegralsCommutingSLEs, DubedatCommutationSLE, KozdronLawlerMultipleSLEs, PeltolaWuGlobalMultipleSLEs} 
for alternative forms for~\eqref{eqn::blm_def}.

\subsection{Estimates on Bessel-type processes}
\label{subsec::pre_Bessel}

\begin{lemma}\label{lem::chordalBessel_estimate}
Fix $0<\kappa<4$ and $\alpha\ge 1$. Suppose $X_t$ is the solution to the following SDE: 
		\begin{equation}\label{eqn::Zhan_calculation}
			\ud X_t=-\sqrt{\kappa X_t(1-X_t)}\ud B_t+2\left(1-(1+\alpha) X_t\right) \ud t, \qquad X_0\in (0,1],
		\end{equation}
		where $B_t$ is standard one-dimensional Brownian motion. 
		\begin{itemize}
		\item[(1)] Suppose $X_0=1$. For all $t>0$ and $\eps\in (0,1/4)$, 
		\begin{equation}\label{eqn::chordalBessel_estimate1}
		\PP[X_t\le \eps]\le \frac{\Gamma\left({4(\alpha+1)}/{\kappa}\right)}{\Gamma\left({4}/{\kappa}\right)\Gamma\left({4\alpha}/{\kappa}\right)}\frac{\kappa}{4} \eps^{{4}/{\kappa}}, 
		\end{equation}
		where $\Gamma$ is the Gamma function. 
		\item[(2)] Suppose $X_0\in (0,1)$. For any $t_0>0$ and $0<\eps<2^{-\frac{4+2\kappa}{4-\kappa}} X_0$, there exists a constant $C_{\eqref{eqn::finite_time_upper_bound}}\in (-\infty,\infty)$ depending on $\alpha$ such that
		\begin{equation}\label{eqn::finite_time_upper_bound}
			\PP\left[\inf_{0\leq t \leq t_0} X_t \leq \eps \right] \leq \exp\left(C_{\eqref{eqn::finite_time_upper_bound}}/\kappa\right) t_0 X_0^{-{4}/{\kappa}}\eps^{{4}/{\kappa}-1}.
		\end{equation}
		\end{itemize}
\end{lemma}
Both estimates~\eqref{eqn::chordalBessel_estimate1} and~\eqref{eqn::finite_time_upper_bound} will be used in the proof of Lemma~\ref{lem::bad_event_estimate} which is a key estimate in the proof the return estimate in Proposition~\ref{prop::SLE_return_chordal}. 
\begin{proof}[Proof of Lemma~\ref{lem::chordalBessel_estimate}~(1)]
From~\cite[Proposition~2.20]{ZhanSLEkappaRhoBubbleMeasures}, the process $(X_t)_{t\geq 0}$ has an invariant distribution with density
		\begin{equation}\label{eqn::invariant_distribution}
			w(x) = \frac{\Gamma\left({4(\alpha+1)}/{\kappa}\right)}{\Gamma\left({4}/{\kappa}\right)\Gamma\left({4\alpha}/{\kappa}\right)}x^{{4}/{\kappa}-1} (1-x)^{{4\alpha}/{\kappa}-1}, \quad x\in (0,1).
		\end{equation}
Let $(\hat{X}_t)_{t\ge 0}$ be a process satisfying~\eqref{eqn::Zhan_calculation} with $\hat{X}_0$ distributed as the invariant density~\eqref{eqn::invariant_distribution}. For all $t>0$ and $\eps\in (0,1)$, we have
\begin{equation*}
	\PP[ X_t\le \eps ] \le \PP[ \hat{X}_t\le \eps ] =\int_{0}^{\eps} w(x) \ud x\le \frac{\Gamma\left({4(\alpha+1)}/{\kappa}\right)}{\Gamma\left({4}/{\kappa}\right)\Gamma\left({4\alpha}/{\kappa}\right)}\frac{\kappa}{4} \eps^{{4}/{\kappa}},
\end{equation*}
which gives~\eqref{eqn::chordalBessel_estimate1} as desired.
\end{proof}
\begin{proof}[Proof of Lemma~\ref{lem::chordalBessel_estimate}~(2)]		
		Define 
$T_{\eps} = \inf\{t: X_t\leq \eps\}$ and set		\[
		f(x):= \frac{2}{\kappa}\int_x^{X_0} \ud y \int_y^{X_0} \ud u \frac{u^{{4}/{\kappa} -1}(1-u)^{{4\alpha}/{\kappa} -1}}{y^{{4}/{\kappa}}(1-y)^{{4\alpha}/{\kappa}}}.
		\]
		Then $f(x)\geq 0$ for $x\in (0,1)$ and $f(x)$ satisifes the following ODE:
		\[
		\frac{\kappa}{2}x(1-x) f''(x) + 2(1-(1+\alpha)x) f'(x) = 1, \quad f(X_0)=0, \quad f'(X_0)=0.
		\]
		It\^o's formula tells that the process $\{f(X_{t\wedge t_0\wedge T_\eps}) - t\wedge t_0\wedge T_\eps\}_{t\geq 0}$ is a bounded martingale. Optional stopping theorem gives
		\[
		\E[f(X_{t_0\wedge T_\eps})] - \E[t_0\wedge T_\eps] = 0.
		\] 
		Since 
		\[
		\E[f(X_{t_0\wedge T_\eps})] \geq \PP[T_\eps\leq t_0] f(\eps),  
		\]
		we have 
		\begin{equation}\label{eqn::Ys_upper_bound_auxiliary}
			\PP\left[\inf_{0\leq t \leq t_0} X_t \leq \eps \right] =  \mathbb{P}[T_\eps\leq t_0] \leq \frac{\E[t_0\wedge T_\eps]}{f(\eps)} \leq \frac{t_0}{f(\eps)}.
		\end{equation}
		\medbreak
		Next, we give a lower bound of $f(\eps)$ as $\eps\to 0$. Observe that 
		\begin{align*}
			f(\eps) = & \frac{2}{\kappa}\int_{\eps}^{X_0} \ud z \int_z^{X_0} \ud u \frac{u^{{4}/{\kappa} -1}(1-u)^{{4\alpha}/{\kappa} -1}}{z^{{4}/{\kappa}}(1-z)^{{4\alpha}/{\kappa}}} \\
			\geq& \frac{2}{\kappa}\int_{\eps}^\frac{X_0}{2} \ud z \int_z^\frac{X_0}{2} \ud u \frac{u^{{4}/{\kappa} -1}(1-u)^{{4\alpha}/{\kappa} -1}}{z^{{4}/{\kappa}}(1-z)^{{4\alpha}/{\kappa}}} \\
			\geq& \frac{2^{{4\alpha}/{\kappa}}}{\kappa} \int_{\eps}^\frac{X_0}{2} \ud z \int_z^\frac{X_0}{2} \ud u \frac{u^{{4}/{\kappa} -1}}{z^{{4}/{\kappa}}}  \\
			=& {2^{{4\alpha}/{\kappa}-2}}\left( \frac{\eps^{1-{4}/{\kappa}}}{(4/\kappa-1)} \left(\frac{X_0}{2}\right)^{{4}/{\kappa}} - \frac{4}{(4-\kappa)}\frac{X_0}{2}+\eps\right)\\
			\geq & \frac{2^{{4\alpha}/{\kappa}-2}\eps^{1-{4}/{\kappa}}}{(4/\kappa-1)}\times \frac{1}{2} \left(\frac{X_0}{2}\right)^{{4}/{\kappa}}\\
			=& \frac{\kappa}{(4 -\kappa)}2^{{4(\alpha-1)}/{\kappa}  -3}X_0^{{4}/{\kappa}} \eps^{1-{4}/{\kappa}}, 
		\end{align*}
where the last inequality is due to $0<\eps<2^{-\frac{4+2\kappa}{4-\kappa}} X_0$. 
		Combining the lower bound of $f(\eps)$ with~\eqref{eqn::Ys_upper_bound_auxiliary}, we obtain~\eqref{eqn::finite_time_upper_bound} with
		\[
		C_{\eqref{eqn::finite_time_upper_bound}} =\sup_{0<\kappa<4} \kappa\log \left(\frac{4-\kappa}{\kappa} 2^{-{4(\alpha-1)}/{\kappa}+3}\right) \in (-\infty, +\infty).
		\]
	\end{proof}

\begin{lemma} \label{lem::radialBessel_estimates}
Fix $\kappa>0$, and $\mu\in\R$, and $0<\kappa<4\alpha$. 
Suppose $X_t$ is the solution to the following SDE:
\begin{equation}\label{eqn::radialBessel_SDE}
	\ud X_t=-\sqrt{\kappa}\ud B_t+\left( \alpha \cot(X_t/2)-\mu\right)\ud t, \qquad X_0\in (0,2\pi),
\end{equation}
where $B_t$ is standard one-dimensional Brownian motion. For any $t_0>0$, there exists a constant $C_{\eqref{eqn::radialBessel_estimates}}\in (0,\infty)$ depending on $\alpha,\mu,X_0$ such that 
\begin{equation}\label{eqn::radialBessel_estimates}
	\PP\left[ \left\{ \min_{0 \le t \le t_0} X_t \le \eps \right\} \bigcup \left\{ \max_{0 \le t \le t_0} X_t \ge 2\pi-\eps \right\}\right] \le \exp\left( C_{\eqref{eqn::radialBessel_estimates}}/\kappa \right) t_0 \eps^{4\alpha/\kappa-1},
\end{equation}
\end{lemma}

\begin{proof}
Define
$T_{\eps}=\inf\{t: X_t\le \eps \text{ or } X_t\ge 2\pi-\eps \}$
and set
\begin{equation*}
	f(\theta):=\frac{2}{\kappa} \int_{\theta}^{X_0} \ud y \int_{y}^{X_0} \ud u \, \frac{\sin(u/2)^{4\alpha/\kappa} \exp(-2\mu u/\kappa)}{\sin(y/2)^{4\alpha/\kappa} \exp(-2\mu y/\kappa)}.
\end{equation*}
Then $f(\theta)\ge 0$ for $\theta\in (0,2\pi)$ and $f(\theta)$ satisfies the following ODE:
\begin{equation*}
	\frac{\kappa}{2} f''(\theta)+ f'(\theta) \left( \alpha \cot(\theta/2)-\mu \right)-1=0, \qquad f(X_0)=0, \qquad f'(X_0)=0.
\end{equation*}
It\^o's formula tells that the process $\{f(X_{t\wedge t_0\wedge T_\eps})-t\wedge t_0\wedge T_\eps\}_{t\ge 0}$ is a bounded martingale. Optional stopping theorem gives
\begin{equation} \label{eqn::radialBessel_estimates_aux1}
	\E[f(X_{t_0\wedge T_\eps})]-\E[t_0\wedge T_\eps]=0.
\end{equation}
Since
\begin{equation*}
	\E[f(X_{t_0\wedge T_\eps})]\ge \PP\left[T_\eps\le t_0\right] \left( f(\eps) \wedge f(2\pi-\eps) \right),
\end{equation*}
combining with~\eqref{eqn::radialBessel_estimates_aux1}, we have
\begin{equation} \label{eqn::radialBessel_estimates_aux2}
	\PP\left[ \left\{ \min_{0 \le t \le t_0} X_t \le \eps \right\} \bigcup \left\{ \max_{0 \le t \le t_0} X_t \ge 2\pi-\eps \right\} \right]= \PP\left[T_\eps\le t_0\right] \le \frac{\E[t_0\wedge T_\eps]}{f(\eps) \wedge f(2\pi-\eps)} \le \frac{t_0}{f(\eps) \wedge f(2\pi-\eps)}.
\end{equation}
Note that
\[
	\frac{\sin(X_0/2)}{X_0} u \le \sin (u/2) \le u/2, \text{ for } 0\le u\le X_0.
\]
Thus for $0<\eps<X_0$, we have
\begin{align} \label{eqn::radialBessel_estimates_aux3}
	f(\eps) \ge& \frac{2}{\kappa} \exp(-8|\mu| \pi/\kappa) \left( \frac{2\sin (X_0/2)}{X_0} \right)^{4\alpha/\kappa} \int_{\eps}^{X_0} \ud y \int_{y}^{X_0} \ud u \left( \frac{u}{y} \right)^{4\alpha/\kappa} \notag\\
	=& \frac{2}{4\alpha+\kappa} \exp(-8|\mu| \pi/\kappa) \left( \frac{2\sin (X_0/2)}{X_0} \right)^{4\alpha/\kappa}	\left( \frac{\kappa X_0^{4\alpha/\kappa+1}}{4\alpha-\kappa} \eps^{1-4\alpha/\kappa} - \frac{\kappa X_0^2}{4\alpha-\kappa}+ \frac{\eps^2-X_0^2}{2}  \right)\notag\\
	\ge & \frac{2}{4\alpha+\kappa} \exp(-8|\mu| \pi/\kappa) \left( \frac{2\sin (X_0/2)}{X_0} \right)^{4\alpha/\kappa}	\left( \frac{\kappa X_0^{4\alpha/\kappa+1}}{4\alpha-\kappa} \eps^{1-4\alpha/\kappa} - \frac{4\alpha+\kappa}{8\alpha-2\kappa}X_0^2 \right).
\end{align}
Similarly, for $0<\eps<2\pi-X_0$, we have
\begin{equation} \label{eqn::radialBessel_estimates_aux4}
	f(2\pi-\eps) \ge \frac{2}{4\alpha+\kappa} \exp(-8|\mu| \pi/\kappa) \left( \frac{2\sin (X_0/2)}{2\pi-X_0} \right)^{4\alpha/\kappa}	\left( \frac{\kappa (2\pi-X_0)^{4\alpha/\kappa+1}}{4\alpha-\kappa} \eps^{1-4\alpha/\kappa} - \frac{4\alpha+\kappa}{8\alpha-2\kappa}(2\pi-X_0)^2  \right).
\end{equation}
Plugging~\eqref{eqn::radialBessel_estimates_aux3} and~\eqref{eqn::radialBessel_estimates_aux4} into~\eqref{eqn::radialBessel_estimates_aux2}, we obtain~\eqref{eqn::radialBessel_estimates} as desired.
\end{proof}

\section{Return estimate for chordal SLE with force points}
\label{sec::chordalSLE_return}
The goal of this section is to derive a return estimate for chordal SLE with force points. 
Fix $\kappa>0$ and $\ell, r\geq 0$. Suppose 
\begin{align}\label{eqn::forcepoints_weights}
\begin{split}
\bs{\rho}=(\bs{\rho}^L; \bs{\rho}^R)\in\mathbb{R}_{\geq 0}^{\ell+r}, \qquad &\bs{x}^L = (x^L_\ell, \ldots, x^L_1)\in \mathbb{R}^\ell, \qquad \bs{x}^R = (x^R_1, \ldots, x^R_r)\in \mathbb{R}^r,\\
&\text{where }x_\ell^L<\cdots<x_1^L<0<x_1^R<\cdots<x_r^R. 
\end{split}
\end{align}
The chordal $\SLE_{\kappa}(\bs{\rho}^L; \bs{\rho}^R)$  in $\HH$ from $0$ to $\infty$ with force points $(\bs{x}^L; \bs{x}^R)$ is defined as the Loewner chain $(K_t)_{t\ge 0}$ driven by a continuous function $W: [0,\infty)\to \R$ satisfying the SDE system 
\begin{equation}\label{eqn::chordalSLEkappa_rho_sde}
	\begin{cases}
\displaystyle 	\ud W_t = \sqrt{\kappa} \ud B_t + \sum_{j=1}^{\ell} \frac{\rho_j^L}{W_t - V_t^{L, j}} \ud t + \sum_{j=1}^{r} \frac{\rho_j^R}{W_t - V_t^{R, j}} \ud t, \qquad W_0 = 0, \\
\displaystyle 	\ud V_t^{L, j} = \frac{2}{V_t^{L, j}-W_t} \ud t, \qquad V_0^{L, j}=x_j^L, \qquad 1\le j\le \ell, \\
\displaystyle 	\ud V_t^{R, j} = \frac{2}{V_t^{R, j}-W_t} \ud t, \qquad V_0^{R, j}=x_j^R, \qquad 1\le j\le r.
\end{cases}
\end{equation}
where $B_t$ is a standard one-dimensional Brownian motion. We denote the law of $(K_t)_{t\ge 0}$ by $\PPtworho$. 
In this article, we only consider the case when $\kappa\in (0,4]$ and all weights are non-negative. In such case, the process is well-defined for all time and is almost surely generated by a continuous curve $\eta$, see e.g.~\cite{DubedatSLEDuality, MillerSheffieldIG1}. Moreover, in this case, the curve does not touch the boundary except at the two ends, and the evolution $V_t^{L, j}$ coincides with $g_t(x^L_j)$ for $1\le j\le \ell$ and the evolution $V_t^{R, j}$ coincides with $g_t(x^R_j)$ for $1\le j\le r$, where $g_t$ is the mapping-out function of the chordal Loewner chain (see Section~\ref{subsec::chordalSLE}). The main result of this section is the following return estimate.

\begin{proposition}\label{prop::SLE_return_chordal}
Assume the same notation as in~\eqref{eqn::forcepoints_weights}. 
Suppose $\eta$ is chordal $\SLE_{\kappa}(\bs{\rho}^L;\bs{\rho}^R)$ in $\HH$ from $0$ to $\infty$ with force points $(\bs{x}^L; \bs{x}^R)$. For $s>0$, we denote
\begin{equation} \label{eqn::chordal_return_stoppingtime}
\tau_s=\inf \{ t\ge 0: |\eta_t|=s \}, \qquad \LF_s = \sigma(\{\eta_t: 0\leq t\le \tau_s\}).
\end{equation}
Then, for any $s>0$ and $M\in[0,\infty)$, there exists a constant $S=S(s,M)>s$ such that 
\begin{equation}\label{eqn::SLE_return_estimate}
	\PPtworho \left[ \eta_{[\tau_S,\infty)} \cap \partial B(0,s) \neq \emptyset \right] \le \exp(-M/\kappa), \qquad \kappa\in (0,1).
\end{equation}
\end{proposition}

Proposition~\ref{prop::SLE_return_chordal} will be a key ingredient in the proof of large deviation in Theorem~\ref{thm::halfwatermelon_LDP}. Return estimates for chordal SLE (without force point) were originally proved in~\cite[Theorem~1.1]{field2015escape}, and later adapted to the LDP framework in~\cite[Proposition~A.3]{PeltolaWangSLELDP} and~\cite[Proposition~4.6]{AbuzaidPeltolaLargeDeviationCapacityParameterization} via techniques from~\cite{LawlerMinkowskiSLERealLine}. Return estimates for chordal SLE with a single force point were established in~\cite[Prop~17]{krusell2024rholoewnerenergylargedeviations}. We prove the general case in Proposition~\ref{prop::SLE_return_chordal} in this section. 
To this end, we first establish a monotone coupling in the chordal setting in Lemma~\ref{lem::monotone_coupling} (in Section~\ref{subsec::monotone_coupling}) to reduce the general case of chordal SLE with force points to the case of one force point on each side. 
Then, in Section~\ref{subsec::return_estimate_proof}, we split the return estimates into the estimates on ``good event" and the estimates on ``bad event". 
On the one hand, on ``good event'', the Radon-Nikodym derivative between the chordal SLE with force points and the standard chordal SLE is well-controlled (see Lemma~\ref{lem::martingale_ratio_bound}), thus we can give a control for the probability of return events using the return estimate for standard chordal SLE in~\cite[Proposition~A.3]{PeltolaWangSLELDP}, see Lemma~\ref{lem::good_event_estimate}. On the other hand, the ``bad event" happens with small probability, since the force points are repulsive and the curve is unlikely to get close to the force points, see Lemma~\ref{lem::bad_event_estimate}.  

\medbreak
As a consequence of Proposition~\ref{prop::SLE_return_chordal}, we obtain large deviation principle for chordal SLE with force points. 
Fix $\ell, r\geq 0$ and $(\ell+r+2)$-polygon $(\Omega; \bs{x}^L, a, \bs{x}^R, b)$. 
Fix $\kappa>0$ and $\bs{\rho} = (\bs{\rho}^L; \bs{\rho}^R) \in \mathbb{R}_{\geq 0}^{\ell+r}$.  The $\SLE_{\kappa}(\bs{\rho}^L;\bs{\rho}^R)$ process in $\Omega$ from $a$ to $b$ with force points $\bs{x} = (\bs{x}^L; \bs{x}^R)$ is defined via conformal invariance: it is the preimage of the chordal $\SLE_{\kappa}(\bs{\rho}^L;\bs{\rho}^R)$ process in $\HH$ from $0$ to $\infty$ with force points $(\varphi(\bs{x}^L); \varphi(\bs{x}^R))$, under any conformal map $\varphi: \Omega\to \HH$ satisfying $\varphi(a)=0$ and $\varphi(b)=\infty$. 
When $\kappa\in (0,4]$ and $\bs{\rho} = (\bs{\rho}^L; \bs{\rho}^R) \in \mathbb{R}_{\geq 0}^{\ell+r}$, the curve does not touch the boundary except at the two ends, thus it is a continuous curve in $\chamber(\Omega; a, b)$.
The law of the process is denoted by 
$\PPtworho(\Omega; \bs{x}^L, a, \bs{x}^R, b)$. 
We also write $\PPtworho = \PPtworho(\HH;\bs{x}^L,0, \bs{x}^R,\infty)$ for short. For standard chordal SLE, we just call it chordal $\SLE_{\kappa}$ in $(\Omega; a, b)$.

\begin{proposition}\label{prop::LDP_SLEkapparho} 
Fix $\ell, r\geq 0$, and $(\ell+r+2)$-polygon $(\Omega; \bs{x}^L, a, \bs{x}^R, b)$. Fix $\bs{\rho} = (\bs{\rho}^L; \bs{\rho}^R) \in \mathbb{R}_{\geq 0}^{\ell+r}$.
The family 
$\{\PPtworho(\Omega; \bs{x}^L, a, \bs{x}^R, b)\}_{\kappa\in (0,4]}$
	 of laws of chordal $\SLE_{\kappa}(\bs{\rho}^L;\bs{\rho}^R)$  satisfies large deviation principle in the space $(\chamber(\Omega; a, b), \dist_{\chamber})$ as $\kappa\to 0+$ with good rate function $\rate^{(\bs{\rho})}(\Omega; \bs{x}^L, a, \bs{x}^R, b; \cdot)$ which will be defined in Definition~\ref{def::rho_energy_chordal}.
\end{proposition}

The rest of this section is organized as follows.  
In Section~\ref{subsec::chordalSLE}, we review the definition and basic properties of chordal SLE with force points.
We introduce the chordal Loewner energy $\rate^{(\bs{\rho})}$ in Definition~\ref{def::rho_energy_chordal}. We also recall the return estimate and the large deviation for standard chordal SLE in Section~\ref{subsec::chordalSLE}.
We prove Proposition~\ref{prop::SLE_return_chordal} in Sections~\ref{subsec::monotone_coupling}-\ref{subsec::chordal_bad}.
The large deviation in Proposition~\ref{prop::LDP_SLEkapparho} will be proved in Section~\ref{subsec::proof_rho_LDP}.

\subsection{Chordal Loewner chain and its energy}
\label{subsec::chordalSLE}
\paragraph*{Chordal Loewner chain.}
An $\HH$-hull is a bounded, relatively closed subset $K$ of $\HH$ such that $\HH\setminus K$ is simply connected. 
Let $g_K$ be
the conformal map from $\HH\setminus K$ onto $\HH$ satisfying the  normalization $\lim_{z\to\infty}|g_K(z)-z|=0$. We refer to $g_K$ as the mapping-out function of $K$. Note that there exists a constant $c_K\ge 0$ such that 
\[g_K(z)=z+\frac{c_K}{z}+o\left(|z|^{-1}\right), \qquad \text{as }z\to\infty.\]
The constant $c_K$ is called the half-plane capacity of $K$ and is denoted by $\mathrm{hcap}(K)$.  

Fix $x\in \mathbb{R}$ and $T\in (0,\infty]$. Let $W: [0,T) \to \R$ be a continuous function with $W_0 = x$.
A chordal Loewner chain driven by $W$ is a family of $\HH$-hulls $(K_t)_{t\in [0,T)}$ such that the mapping-out function $g_t:=g_{K_t}$ satisfies the Loewner equation
\begin{equation*}
	\partial_t g_t(z)=\frac{2}{g_t(z)-W_t}, \qquad g_0(z)=z, \qquad z\in \overline{\HH}.
\end{equation*}
For each $z\in \overline{\HH}$, the flow $t\mapsto g_t(z)$ is well-defined up to the swallowing time 
\[
    \sigma_z := \sup\left\{ t \in [0, T) : \inf_{s \in [0, t]} |g_s(z) - W_s| > 0 \right\}.
\]
The hulls are determined by  $K_t = \{z\in \HH : \sigma_z \le t\}$,
and $(K_t)_{t\in [0,T)}$ is parameterized by half-plane capacity, i.e., $\mathrm{hcap}(K_t)=2t$ for all $t\in [0,T)$. Moreover, $(K_t)_{t\in [0,T)}$ satisfies the local growth property: for $0\le s<t<T$, the diameter of $g_s(K_t\setminus K_s)$ tends to $0$ as $t\downarrow s$ uniformly over $s\in [0,T)$. 

Conversely, any increasing family of $\HH$-hulls $(K_t)_{t\in [0,T)}$ satisfying the local growth property can be represented as a chordal Loewner chain driven by some continuous function $W$. A simple curve $\eta$ in $(\HH; 0,\infty)$ satisfies the local growth property, and the driving function can be obtained easily by $W_t = g_t(\eta_t)$. For more details, refer to~\cite{Berestycki2011LecturesOS, KemppainenSchrammLoewnerevolution} for example.

\paragraph*{Chordal SLE with force points.}
We use the same notation as in~\eqref{eqn::forcepoints_weights}. 
Recall that the chordal $\SLE_{\kappa}(\bs{\rho}^L; \bs{\rho}^R)$ process in $\HH$ from $0$ to $\infty$ with force points $(\bs{x}^L; \bs{x}^R)$ is defined as the Loewner chain $(K_t)_{t\ge 0}$ driven by a continuous function $W: [0,\infty)\to \R$ satisfying the SDE system~\eqref{eqn::chordalSLEkappa_rho_sde}. Its law is absolutely continuous with respect to chordal SLE and the Radon-Nikodym derivative is given by the following lemma. 

\begin{lemma}[{\cite{SchrammWilsonSLECoordinatechanges}}]\label{lem::SLEkapparho_martingale}
Fix $\kappa\in (0,4]$ and assume the same notation as in \eqref{eqn::forcepoints_weights}. The law of chordal $\SLE_{\kappa}(\bs{\rho}^L; \bs{\rho}^R)$ in $\HH$ from $0$ to $\infty$ with force points $(\bs{x}^L; \bs{x}^R)$ is the same as the law of chordal $\SLE_{\kappa}$ in $(\HH;0,\infty)$ tilted by the following martingale, up to the first time $x^L_1$ or $x^R_1$ is swallowed:
		\begin{align} \label{eqn::SLEkapparho_martingale}
			M_t =&\prod_{j=1}^{\ell} \left( g'_t(x_j^L)^{\frac{\rho_j^L (\rho_j^L+4-\kappa)}{4\kappa}} (W_t-g_t(x_j^L))^{\frac{\rho_j^L}{\kappa}}\right) \times \prod_{j=1}^{r} \left( g'_t(x_j^R)^{\frac{\rho_j^R (\rho_j^R+4-\kappa)}{4\kappa}} (g_t(x_j^R)-W_t)^{\frac{\rho_j^R}{\kappa}}\right) \notag \\
			&\times \prod_{\substack{1\le i\le \ell \\ 1\le j\le r}} (g_t(x_j^R)-g_t(x_i^L))^{\frac{\rho_i^L \rho_j^R}{2\kappa}} \times \prod_{1\le i<j\le \ell} (g_t(x_i^L)-g_t(x_j^L))^{\frac{\rho_i^L \rho_j^L}{2\kappa}} \notag \\
			&\times \prod_{1\le i<j\le r} (g_t(x_j^R)-g_t(x_i^R))^{\frac{\rho_i^R \rho_j^R}{2\kappa}}. 
		\end{align}
	\end{lemma}

	\begin{definition}[Chordal Loewner energy]\label{def::rho_energy_chordal}
Fix $\ell, r\geq 0$ and $(\ell+r+2)$-polygon $(\Omega; \bs{x}^L, a, \bs{x}^R, b)$. Fix $\bs{\rho} = (\bs{\rho}^L; \bs{\rho}^R) \in \mathbb{R}_{\geq 0}^{\ell+r}$.
Let  $\varphi:\Omega \to \mathbb{H}$ be a conformal map  satisfying $\varphi(a)=0$ and $\varphi(b)=\infty$. 
	For $\eta\in \overline{\chamber}(\Omega; a,b)$, we parameterize $\eta$ by half-plane capacity of $\varphi(\eta)$. We denote by $W_t$ the driving function of $\varphi(\eta)$. We define	the \emph{truncated $\bs{\rho}$-Loewner energy} associated with the force points $\bs{x}=(\bs{x}^L; \bs{x}^R)$ to be
	\begin{equation} \label{eqn::rho_energy_chordal}
	{\rate}^{(\bs{\rho})}(\Omega;\bs{x}^L, a, \bs{x}^R, b;\eta_{[0,T]}) \;:=\; \frac{1}{2} \int_{0}^{T} \left(\dot{W}_t-\sum_{j=1}^{\ell} \frac{\rho_j^L}{W_t-g_t(x_j^L)}-\sum_{j=1}^{r} \frac{\rho_j^R}{W_t-g_t(x_j^R)} \right)^2 \ud t,
	\end{equation}
	if $W$ is absolutely continuous; and setting $\rate^{(\bs{\rho})}\left(\Omega;\bs{x}^L, a, \bs{x}^R, b; \cdot\right):= +\infty$ otherwise. The quantity ${\rate}^{(\bs{\rho})}(\Omega;\bs{x}^L, a, \bs{x}^R, b;\eta_{[0,T]}) $ is increasing in $T$. We define the \emph{$\bs{\rho}$-Loewner energy} of $\eta\in \overline{\chamber}(\Omega; a,b)$ to be its limit:
	\begin{equation} \label{eqn::rho_energy_chordal_full}
		\rate^{(\bs{\rho})}(\Omega;\bs{x}^L, a, \bs{x}^R, b;\eta) \; := \; \lim_{T\to +\infty} \rate^{(\bs{\rho})}(\Omega;\bs{x}^L, a, \bs{x}^R, b;\eta_{[0,T]}).
	\end{equation}
\end{definition}

When all weights are zero, the $\bs{\rho}$-Loewner energy reduces to the standard Loewner energy defined in~\cite{Wang:EnergyLoewnerChain,PeltolaWangSLELDP}: 
\begin{equation}
	\label{eqn::LDP_rate_function_standard}
		\rate(\Omega; x, y; \eta_{[0,T]}):= \; \frac{1}{2} \int_{0}^{T} \left(\dot{W}_t\right)^2 \ud t,\qquad 
	\rate(\Omega; x, y; \eta):= \; \frac{1}{2} \int_{0}^{+\infty} \left(\dot{W}_t\right)^2 \ud t, 
\end{equation}
if $W$ is absolutely continuous; and setting $\rate(\Omega; x, y; \cdot):= +\infty$ otherwise. We recall the return estimate for standard chordal SLE and its large deviation below. They will be the foundation for our proof of Proposition~\ref{prop::SLE_return_chordal} and Theorem~\ref{thm::halfwatermelon_LDP}.

		\begin{lemma}[{\cite[Eq.~(A.10)]{PeltolaWangSLELDP}}]\label{lem::standard_return_estimate}
		Suppose $\eta\sim\PPtwo$ is chordal $\SLE_{\kappa}$ in $(\HH; 0,\infty)$.
		For any $S>s>0$ and $\kappa\in (0,4]$, we have 
			\begin{equation}\label{eqn::standard_chordal_return}
				\PPtwo \left[ \eta_{[\tau_S,\infty)} \cap \partial B(0,s) \neq \emptyset \right] \leq c_\kappa \left(\frac{s}{S}\right)^{{8}/{\kappa}-1},
			\end{equation}
			where $c_\kappa>0$ is a constant such that 
			\[\lim_{\kappa\to 0+} \kappa \log c_\kappa \in (-\infty, +\infty).\]
		\end{lemma}
\begin{lemma}[{\cite[Theorem~1.2]{AbuzaidPeltolaLargeDeviationCapacityParameterization}}]\label{lem::LDP_SLEkappa_standard} 
	Fix $2$-polygon $(\Omega;x,y)$. The family $\{\PPtwo(\Omega; x, y)\}_{\kappa\in (0,4]}$  of laws of  chordal $\SLE_{\kappa}$  satisfies large deviation principle in the space $(\chamber(\Omega; x, y), \dist_{\chamber})$ as $\kappa\to 0+$ with good rate function $\rate(\Omega; x, y; \eta)$ in~\eqref{eqn::LDP_rate_function_standard}.
\end{lemma}

\subsection{A monotone coupling} 
\label{subsec::monotone_coupling}
\begin{lemma}[{\cite[Lemma~3.3]{ZhanSLEkappaRhoBubbleMeasures}}]\label{lem::Zhan_calculation}
Fix $\alpha\ge 1$. Suppose $(Y_t^L; Y_t^R)_{t\ge 0}$ is the solution to the following SDE system:
			\begin{align}\label{eqn::XR_coupled_process}
				\begin{cases}
					\displaystyle \ud Y_t^L = -\sqrt{\kappa} \ud B_t + \frac{2\alpha}{Y_t^L} \ud t, \quad & Y_0^L \le 0,  \\
					\displaystyle \ud Y_t^R = -\sqrt{\kappa} \ud B_t + \frac{2}{Y_t^R} \ud t + \frac{2\alpha-2}{Y_t^L} \ud t, \quad &  Y_0^R > 0. 
				\end{cases}
			\end{align}
Set 	\[
		X_s = \frac{Y_{t(s)}^R}{Y_{t(s)}^R - Y_{t(s)}^L}\in [0,1], \qquad s(t) = \frac{1}{2} \log \frac{Y_{t}^R - Y_{t}^L}{Y_0^R-Y_0^L}.
		\] 
		Then $(X_s)_{s\geq 0}$ satisfies the SDE~\eqref{eqn::Zhan_calculation} with initial data $X_0=\frac{Y_0^R}{Y_0^R-Y_0^L}$.
\end{lemma}

	\begin{lemma}[Monotone coupling]\label{lem::monotone_coupling} 
	Fix $\kappa\in (0,4]$ and assume the same notation as in~\eqref{eqn::forcepoints_weights}.
		Suppose $(W_t, V_t^{L,\ell}, \ldots, V_t^{L,1}, V_t^{R,1}, \ldots, V_t^{R,r})$ is the solution to the SDE system~\eqref{eqn::chordalSLEkappa_rho_sde}. We define $\bs{\Delta}_t := (\Delta_t^{L,\ell}, \ldots, \Delta_t^{L,1}; \Delta_t^{R,1}, \ldots, \Delta_t^{R,r})$ where
		\[
		\Delta_t^{L,i}:= V_t^{L,i} - W_t \qquad \text{for } 1\leq i \leq \ell, \qquad \text{and} \qquad \Delta_t^{R,j}:= V_t^{R,j} - W_t \qquad \text{for } 1\leq j \leq r. 
		\]
		Suppose $(Y_t^L; Y_t^R)_{t\ge 0}$ is the solution to the SDE system~\eqref{eqn::XR_coupled_process} with $Y_0^L\in [x_1^L,0-]$ and $Y_0^R=x_1^R$ and		
		\begin{equation*}
		\alpha = 1+(\rho_1^L+\cdots+\rho_\ell^L)/2.
		\end{equation*} 
		Then there exists a coupling between $(Y^L_t; Y^R_t)_{t\geq 0}$ and $(\bs{\Delta}_t )_{t\geq 0}$ such that $Y_0^R = \Delta_0^{R,1}$ and $Y_t^R \leq \Delta_t^{R,1}$ for all $t > 0$ almost surely.
	\end{lemma}
	\begin{proof}
		From~\eqref{eqn::chordalSLEkappa_rho_sde}, the process $(\bs{\Delta}_t)_{t\geq 0}$ is governed by the following SDE system:
	\begin{align}\label{eqn::Z_process_system}
		\begin{cases}
			\displaystyle \ud \Delta_t^{L,i} = -\sqrt{\kappa} \ud B_t + \frac{\rho_i^L+2}{\Delta_t^{L,i}} \ud t + \sum_{k\neq i} \frac{\rho_k^L}{\Delta_t^{L,k}} \ud t + \sum_{j=1}^r \frac{\rho_j^R}{\Delta_t^{R,j}} \ud t, \quad & \Delta_0^{L,i} = x_i^L< 0, \quad 1\leq i \leq \ell, \\[10pt]
			\displaystyle \ud \Delta_t^{R,j} = -\sqrt{\kappa} \ud B_t + \frac{\rho_j^R+2}{\Delta_t^{R,j}} \ud t + \sum_{k\neq j} \frac{\rho_k^R}{\Delta_t^{R,k}} \ud t + \sum_{i=1}^\ell \frac{\rho_i^L}{\Delta_t^{L,i}} \ud t, \quad & \Delta_0^{R,j} = x_j^R> 0, \quad 1\leq j \leq r.
		\end{cases}    
	\end{align}
When $\kappa\in(0,4]$ and $\rho_i^L, \rho_j^R \geq 0$ for all $i,j$, the solution $\bs{\Delta}_t$ exists for all time and never hits $0$ (see, e.g.,~\cite[Lemma~15]{DubedatSLEDuality}).

	Given $(Y_t^L; Y_t^R)_{t\geq 0}$. Let $\bs{Z}_t := ( Z_t^{L,\ell}, \ldots, Z_t^{L,1}; Z_t^{R,1}, \ldots, Z_t^{R,r})$ be the solution to the following ODE system: 		 
		 for $1\leq i \leq \ell$ and $1\leq j \leq r$,
		\begin{align*}
			\begin{cases}
				\displaystyle \ud Z_t^{L,i} = \left( \frac{2}{Y_t^L}- \frac{2}{ Y_t^L-Z_t^{L,i}}  \right)\ud t + \sum_{k=1}^\ell \left(\frac{\rho_k^L}{Y_t^L}-\frac{\rho_k^L}{ Y_t^L-Z_t^{L,k}} \right) \ud t - \sum_{j=1}^r \frac{\rho_j^R}{Y^R_t-Z_t^{R,j}}\ud t, \ Z_0^{L,i} = Y_0^L-x_i^L, \\[10pt]
				\displaystyle \ud Z_t^{R,j} = \left( \frac{2}{Y_t^R}-\frac{2}{Y_t^R-Z_t^{R,j}}  \right)\ud t - \sum_{k=1}^r\frac{\rho_k^R}{ Y_t^R-Z_t^{R,k}}\ud t + \sum_{i=1}^\ell\left(\frac{\rho_i^L}{Y_t^L}-\frac{\rho_i^L}{Y_t^L-Z_t^{L,i}}  \right) \ud t, \ Z_0^{R,j} = Y_0^R- x_j^R .
			\end{cases}
		\end{align*}
		Since the coefficients in the RHS of the ODE system for $\bs{Z}_t$ are locally Lipschitz continuous, there exists a unique strong solution $(\bs{Z}_t)_{t\geq 0}$. Set $$\Delta_t^{L,i} = Y_t^L - Z_t^{L,i} \text{ for } 1\leq i\leq\ell, \qquad \Delta_t^{R,j} = Y_t^R - Z_t^{R,j} \text{ for } 1\leq j\leq r. $$
		We find that $\bs{\Delta}_t$ satisfies the SDE system~\eqref{eqn::Z_process_system}. This gives a coupling of $(Y^L_t; Y^R_t)_{t\geq 0}$ and $(\bs{\Delta}_t)_{t\geq 0}$.
		\medbreak
		It remains to verify that $Y^R_t \leq \Delta^{R,1}_t$, which is equivalent to showing $Z_t^{R,1} \leq 0$. The evolution of $Z_t^{R,1}$ is given by
		\begin{equation*}\label{eqn::DeltaR1}
			\ud Z_t^{R,1} = \left( \frac{2}{Y_t^R}-\frac{2}{Y_t^R-Z_t^{R,1}}  \right)\ud t - \sum_{k=1}^r\frac{\rho_k^R}{ Y_t^R-Z_t^{R,k}}\ud t + \sum_{i=1}^\ell\left(\frac{\rho_i^L}{Y_t^L}-\frac{\rho_i^L}{Y_t^L-Z_t^{L,i}}  \right) \ud t.
		\end{equation*}
		Since $x_\ell^L < \dots < x_1^L < 0 < x_1^R < \dots < x_r^R$, the initial data satisfies
		\[
		Z_0^{L,\ell}>\ldots>Z_0^{L,1}\geq 0=Z_0^{R,1}>\ldots>Z_0^{R,r}.
		\]
		We observe that whenever $Z_t^{R,1}=0$, the drift of $Z_t^{R,1}$ is strictly negative. Given that $Z_0^{R,1}=0$, the term $Z_t^{R,1}$ cannot become positive. Thus $Z_t^{R,1} \leq 0$ for all $t>0$. This completes the proof.
	\end{proof}
	
\subsection{Proof of Proposition~\ref{prop::SLE_return_chordal}}\label{subsec::return_estimate_proof}
\begin{proof}[Proof of Proposition~\ref{prop::SLE_return_chordal}]
Recall that the driving function for chordal $\SLE_\kappa(\bs{\rho}^L; \bs{\rho}^R)$ satisfies the SDE system~\eqref{eqn::chordalSLEkappa_rho_sde}. For $1\leq i \leq \ell$ and $1\leq j \leq n$, we denote  
	\begin{align*}\overline{\rho}^L:= \sum_{i=1}^\ell \rho_i^L,\qquad \overline{\rho}^R:= \sum_{j=1}^r \rho_j^R, \qquad  \Delta_t^{L,i} = V_t^{L,i} - W_t  \qquad \Delta_t^{R,j} = V_t^{R,j} - W_t.
	\end{align*}
	We use notations in \eqref{eqn::chordal_return_stoppingtime} and define ``bad events'' to be 
	\begin{align*} 
			E_n^R &= \left\{\exists t\in [\tau_{2^n S}, \tau_{2^{n+1}S}] \text{ s.t. } \left|\Delta_t^{R,1}\right|\leq \eps_n 2^n S\right\},\\ 
			E_n^L &= \left\{\exists t\in [\tau_{2^n S}, \tau_{2^{n+1}S}] \text{ s.t. } \left|\Delta_t^{L,1}\right|\leq \eps_n 2^n S\right\},
		\end{align*}
		for $n\in \mathbb{N}$, with $\eps_n = (n+n_0)^{-2}$ for some large constant $n_0\geq 1$ to be chosen later. Then we have
		\begin{align}
			&\PPtworho \left[ \eta_{[\tau_S,\infty)} \cap \partial B(0,s) \neq \emptyset \right] \notag \\ 
			\leq &\sum_{n=0}^{\infty} \PPtworho \left[E_n^L \cup E_n^R\right]+\sum_{n=0}^{\infty} \PPtworho\left[ \left\{\eta_{[\tau_{2^n S}, \tau_{2^{n+1}S}]}\cap \partial B(0,s) \neq \emptyset\right\}\cap \left(E_n^L \cup E_n^R\right)^c \right]. \label{eqn::return_estimate_decomposition}
		\end{align}
		We will prove in Lemma \ref{lem::good_event_estimate} that there exists a constant $C_{\eqref{eqn::good_event_estimate}}\in (-\infty, +\infty)$ depending on $n_0, \bs{\rho}^L, \bs{\rho}^R$ such that when $S>s>0$ and $\kappa\in (0,4]$, we have
		\begin{equation}\label{eqn::good_event_estimate}
			\sum_{n=0}^{\infty} \PPtworho\left[\left\{ \eta_{[\tau_{2^n S}, \tau_{2^{n+1}S}]}\cap \partial B(0,s) \neq \emptyset\right\}\cap \left(E_n^L \cup E_n^R\right)^c \right] \leq \exp\left(C_{\eqref{eqn::good_event_estimate}}/\kappa\right) \left(\frac{s}{S}\right)^{{8}/{\kappa}-1}.
		\end{equation}
		We will prove in Lemma \ref{lem::bad_event_estimate} that there exists a constant $C_{\eqref{eqn::bad_event_estimate}}\in (-\infty, +\infty)$ depending on $\bs{\rho}^L, \bs{\rho}^R$ such that when $0<\kappa<1$, we have 
		\begin{equation}\label{eqn::bad_event_estimate}
			\sum_{n=0}^{\infty} \PPtworho \left[E_n^L \cup E_n^R\right] \leq 2\exp\left(C_{\eqref{eqn::bad_event_estimate}}/\kappa\right) n_0^{1-2/\kappa}.
		\end{equation}
		Taking 
		\[
		n_0 = \left\lfloor 4 \exp\left(M + C_{\eqref{eqn::bad_event_estimate}}\right)\right\rfloor+1, \qquad S = s 2^{1/4} \exp\left(\frac{C_{\eqref{eqn::good_event_estimate}}+M}{4}\right).
		\]
		Plugging~\eqref{eqn::good_event_estimate} and~\eqref{eqn::bad_event_estimate} into~\eqref{eqn::return_estimate_decomposition}, we obtain 
		\[
		\PPtworho \left[ \eta_{[\tau_S,\infty)} \cap \partial B(0,s) \neq \emptyset \right]\leq \frac{1}{2}\exp\left(-M/\kappa\right) + \frac{1}{2}\exp\left(-M/\kappa\right) = \exp\left(-M/\kappa\right),
		\]
		as desired in \eqref{eqn::SLE_return_estimate}.
\end{proof}

		\begin{lemma}\label{lem::martingale_ratio_bound}
Fix $\kappa\in (0,4]$ and assume $M_t$ is the martingale defined in~\eqref{eqn::SLEkapparho_martingale}. Fix a large radius $S$ such that $S>\max\{|x_\ell^L|, |x_r^R|\}$. On the good event $\left(E_n^L\cup E_n^R\right)^c$, we have
			\begin{align}\label{eqn::martingale_ratio_bound}
			\frac{M_{\tau_{2^{n+1} S}}}{M_{\tau_{2^{n} S}}} \leq   \left(\frac{12}{\eps_n}\right)^{{\tilde{\rho}}/{\kappa}}, \qquad\text{where }\tilde{\rho}:= \overline{\rho}^L + \overline{\rho}^R + \frac{1}{2}\overline{\rho}^L\overline{\rho}^R.
		\end{align}
		\end{lemma}
				\begin{proof}
		From~\eqref{eqn::SLEkapparho_martingale}, we have 
		\begin{align}
		\frac{M_{\tau_{2^{n+1} S}}}{M_{\tau_{2^{n} S}}}=& \underbrace{\prod_{j=1}^{\ell} \left(\frac{g'_{\tau_{2^{n+1} S}}(x_j^L)}{g'_{\tau_{2^{n} S}}(x_j^L)}\right)^{\frac{\rho_j^L (\rho_j^L+4-\kappa)}{4\kappa}}\times \prod_{j=1}^{r} \left(\frac{g'_{\tau_{2^{n+1} S}}(x_j^R)}{g'_{\tau_{2^{n} S}}(x_j^R)}\right)^{\frac{\rho_j^R (\rho_j^R+4-\kappa)}{4\kappa}}}_{R_1:=}\notag\\
		& \times \underbrace{\prod_{1\leq i < j\leq \ell} \left(\frac{g_{\tau_{2^{n+1} S}}(x_i^L)-g_{\tau_{2^{n+1} S}}(x_j^L)}{g_{\tau_{2^{n} S}}(x_i^L)-g_{\tau_{2^{n} S}}(x_j^L)}\right)^{\frac{\rho_i^L \rho_j^L}{2\kappa}} \times \prod_{1\leq i < j\leq r} \left(\frac{g_{\tau_{2^{n+1} S}}(x_j^R)-g_{\tau_{2^{n+1} S}}(x_i^R)}{g_{\tau_{2^{n} S}}(x_j^R)-g_{\tau_{2^{n} S}}(x_i^R)}\right)^{\frac{\rho_i^R \rho_j^R}{2\kappa}}}_{R_2:=}\notag\\
		& \times \underbrace{\prod_{j=1}^{\ell} \left(\frac{W_{\tau_{2^{n+1} S}}-g_{\tau_{2^{n+1} S}}(x_j^L)}{W_{\tau_{2^{n} S}}-g_{\tau_{2^{n} S}}(x_j^L)}\right)^{\frac{\rho_j^L}{\kappa}}}_{R_3:=} \times \underbrace{\prod_{j=1}^{r} \left(\frac{g_{\tau_{2^{n+1} S}}(x_j^R)-W_{\tau_{2^{n+1} S}}}{g_{\tau_{2^{n} S}}(x_j^R)-W_{\tau_{2^{n} S}}}\right)^{\frac{\rho_j^R}{\kappa}}}_{R_4:=} \notag\\
		&\times \underbrace{\prod_{\substack{1\le i\le \ell \\ 1\le j\le r}} \left(\frac{g_{\tau_{2^{n+1} S}}(x_j^R)-g_{\tau_{2^{n+1} S}}(x_i^L)}{g_{\tau_{2^{n} S}}(x_j^R)-g_{\tau_{2^{n} S}}(x_i^L)}\right)^{\frac{\rho_i^L \rho_j^R}{2\kappa}}}_{R_5:=}\label{eqn::martingale_ratio_bound_auxiliary},
		\end{align}
	 Let us evaluate the terms in RHS of \eqref{eqn::martingale_ratio_bound_auxiliary} one by one.
		\begin{itemize}
			\item For $R_1$ and $R_2$, since $g_t'(x)$ is decreasing in $t$ and $|g_t(x)-g_t(y)|$ is also decreasing in $t$ for $x>y>0$ or $x<y<0$, we have $
			R_1\leq 1$ and $ R_2 \leq 1$.
			\item For $R_3$ and $R_4$, since $S>\max\{|x_\ell^L|, |x_r^R|\}$, by~\cite[Lemma~4.5]{KemppainenSchrammLoewnerevolution}, we have 
			\begin{align}\label{eqn::chordal_good_aux1}
			\left| g_{\tau_{2^{n+1}S}}(x_i^R) - W_{\tau_{2^{n+1}S}} \right|\leq 6S 2^{n+1}, \qquad \left| W_{\tau_{2^{n+1}S}} - g_{\tau_{2^{n+1}S}}(x_i^L) \right|\leq 6S 2^{n+1}.
			\end{align}
			On the good event $\left(E_n^L\cup E_n^R\right)^c$, we have 
			\begin{align}\label{eqn::chordal_good_aux2}
			\left|g_{\tau_{2^{n}S}}(x_i^R) - W_{\tau_{2^{n}S}}\right|\geq \eps_n 2^n S, \qquad \left| g_{\tau_{2^{n}S}}(x_i^L)-W_{\tau_{2^{n}S}} \right|\geq \eps_n 2^n S.
			\end{align}
			Thus,
			\[
			R_3 \leq \left(\frac{12}{\eps_n}\right)^{\overline{\rho}^L / \kappa}, \qquad R_4\leq \left(\frac{12}{\eps_n}\right)^{\overline{\rho}^R / \kappa}.
			\]
			\item For $R_5$, combining~\eqref{eqn::chordal_good_aux1} and~\eqref{eqn::chordal_good_aux2}, we have
			\begin{align*}
				\left|g_{\tau_{2^{n+1}S}}(x_j^R) - g_{\tau_{2^{n+1}S}}(x_i^L)\right|\leq 12 S 2^{n+1}, \qquad
				 \left|g_{\tau_{2^{n}S}}(x_j^R) - g_{\tau_{2^{n}S}}(x_i^L)\right|\geq 2S \eps_n 2^n.
			\end{align*}
			Thus
			\[
			R_5 \leq \left(\frac{12}{\eps_n}\right)^{\overline{\rho}^L \overline{\rho}^R / 2\kappa}.
			\] 
		\end{itemize}
		Combining these estimates, we obtain~\eqref{eqn::martingale_ratio_bound} as desired.
		\end{proof}

		\begin{lemma}
		\label{lem::good_event_estimate} Assume the same notation as in the proof of Proposition~\ref{prop::SLE_return_chordal}.
		The estimate~\eqref{eqn::good_event_estimate} on the good event holds for $\kappa\in (0,4]$.
		\end{lemma}
				\begin{proof}
			By the strong domain Markov property of chordal $\SLE_{\kappa}(\bs{\rho}^L; \bs{\rho}^R)$, given $\LF_{{2^n S}}$, the law of $\eta_{[\tau_{2^n S}, \tau_{2^{n+1}S}]}$ is absolutely continuous with respect to the law of chordal $\SLE_{\kappa}$ in $\HH\setminus \eta_{[0,\tau_{2^n S}]}$ with Radon-Nikodym derivative $M_{\tau_{2^{n+1} S}} / M_{\tau_{2^{n} S}}$, where $M_t$ is the martingale in~\eqref{eqn::SLEkapparho_martingale}. We can compare chordal $\SLE_{\kappa}(\bs{\rho}^L; \bs{\rho}^R)$ process to the standard chordal $\SLE_{\kappa}$ process: 
		\begin{align*}
			& \sum_{n=0}^{\infty} \PPtworho\left[ \left\{\eta_{[\tau_{2^n S}, \tau_{2^{n+1}S}]}\cap \partial B(0,s) \neq \emptyset\right\}\cap\left(E_n^L \cup E_n^R\right)^c \right] \\
			\leq& \sum_{n=0}^{\infty} \Etworho\left[ \PPtworho\left[\left[\left\{\eta_{[\tau_{2^n S}, \tau_{2^{n+1}S}]}\cap \partial B(0,s) \neq \emptyset\right\}\cap\left\{\frac{M_{\tau_{2^{n+1} S}}}{M_{\tau_{2^{n} S}}}\leq \left(\frac{12}{\eps_n}\right)^{{\tilde{\rho}}/{\kappa}}\right\} \right|\LF_{{2^n S}}\right]\right] \tag{due to~\eqref{eqn::martingale_ratio_bound}}\\
			\leq& \sum_{n=0}^{\infty} c_\kappa \left(\frac{s}{2^n S}\right)^{{8}/{\kappa}-1} \left(\frac{12}{\eps_n}\right)^{{\tilde{\rho}}/{\kappa}}  \tag{due to~\eqref{eqn::standard_chordal_return}} \\
			\leq& \underbrace{c_{\kappa} \sum_{n=0}^{\infty} \left(\frac{12}{\eps_n}\right)^{{\tilde{\rho}}/{\kappa}} 2^{-{4n}/{\kappa}}}_{\tilde{c}_{\kappa}:=} \left(\frac{s}{S}\right)^{{8}/{\kappa}-1}, 
		\end{align*}
		where in the last inequality we use $8/\kappa - 1 \ge 4/\kappa$ for $\kappa \le 4$.
		The coefficient
		\[
		\tilde{c}_\kappa= c_\kappa \sum_{n=0}^{\infty} \left(\frac{\left(12(n+n_0)^2\right)^{\tilde{\rho}}}{2^{4n}}\right)^{{1}/{\kappa}}
		\]
		satisfies $\limsup_{\kappa\to 0+} \kappa\log \tilde{c}_\kappa\in (-\infty, +\infty)$. Set 
		\[
		C_{\eqref{eqn::good_event_estimate}} :=\sup_{0<\kappa\leq 4} \kappa\log \tilde{c}_\kappa\in (-\infty, \infty).
		\]
		Then we obtain the estimate~\eqref{eqn::good_event_estimate}
		as desired.
		\end{proof}

\subsection{Estimate on the bad event}\label{subsec::chordal_bad}

To prove the estimate~\eqref{eqn::bad_event_estimate} on the bad event in the proof of Proposition~\ref{prop::SLE_return_chordal}, two essential inputs are the monotone coupling in Lemma~\ref{lem::monotone_coupling} and the Bessel-type estimates~\eqref{eqn::chordalBessel_estimate1} and~\eqref{eqn::finite_time_upper_bound} in Lemma~\ref{lem::chordalBessel_estimate}. 

	\begin{lemma}\label{lem::bad_event_estimate}
		Assume the same notation as in the proof of Proposition~\ref{prop::SLE_return_chordal}. The estimate~\eqref{eqn::bad_event_estimate} on the bad event holds for $\kappa<1$. 
	\end{lemma} 
	\begin{proof}
	As 
$\PPtworho \left[E_n^L \cup E_n^R\right]\le \PPtworho \left[E_n^L\right]+\PPtworho \left[E_n^R\right]$,
it suffices to show that there exists a constant $C_{\eqref{eqn::chordal_bad_aux0}}\in (-\infty, +\infty)$ depending on $\bs{\rho}^L$ such that when $0<\kappa<1$, we have
\begin{equation}\label{eqn::chordal_bad_aux0}
\sum_{n=0}^{\infty}\PPtworho \left[E_n^R \right]\le \exp\left(C_{\eqref{eqn::chordal_bad_aux0}}/\kappa\right) n_0^{1-2/\kappa}.
\end{equation}	
We denote by $\PP$ the measure of the coupling in Lemma~\ref{lem::monotone_coupling} with $\alpha = 1+\overline{\rho}^L/2$. Set 
		\[
		\tilde{E}_n^R:= \{\exists t\in [\tau_{2^n S}, \tau_{2^{n+1} S}] \text{ s.t. } |Y_t^R|\leq \eps_n 2^n S\},
		\]
		where $(Y_t^R)_{t\geq 0}$ is the coupled process with $Y_0^L=0-$ and $ Y_0^R = x_1^R>0$. 
		By the coupling in Lemma~\ref{lem::monotone_coupling}, we have the inclusion $E_n^R \subset \tilde{E}_n^R$.
		\medbreak
		Note that the process $g_t(0+) -g_t(0-)$ is increasing in $t$.  Consider the stopping time 
		\begin{equation*}
			\hat{\tau}_S = \inf\{t\geq 0: |g_t(0+) - g_t(0-)| \geq S\}.
		\end{equation*}
	We claim that 
	\begin{equation}\label{eqn::chordal_bad_aux1}
	\hat{\tau}_S \leq \tau_S\le \hat{\tau}_{12S}. 
	\end{equation}
	We first show $\hat{\tau}_S \leq \tau_S$.
From the monotonicity of the harmonic measure, the quantity $|g_{{\tau}_S}(0+)-g_{{\tau}_S}(0-)|$ is bounded from below by the harmonic measure of $\partial B(0,S)\cap \HH$ in $\HH$ seen from $\infty$. The mapping-out function for $\partial B(0,S)\cap \HH$ is given by $g(z)=z+S^2/z$. This tells that the harmonic measure of $\partial B(0,S)\cap \HH$ in $\HH$ seen from $\infty$ is $4S$. Thus, 
$|g_{{\tau}_S}(0+)-g_{{\tau}_S}(0-)|\ge 4S$.
This confirms $\hat{\tau}_S \leq \tau_S$. The equation~\eqref{eqn::chordal_good_aux1} indicates that
$|g_{\tau_S}(0+) - g_{\tau_S}(0-)| \leq 12S$,
This confirms $\tau_S\le \hat{\tau}_{12S}$ and completes the proof of~\eqref{eqn::chordal_bad_aux1}.
\medbreak		
We further set 
		\[
		F_n = F_n^R:= \{\exists t\in [\hat{\tau}_{2^n S}, \tau_{2^{n+1} S}] \text{ s.t. } |Y_t^R|\leq \eps_n 2^n S\}. 
		\]
Then $\tilde{E}_n^R\subset F_n$ as $\hat{\tau}_S\le \tau_S$.
		We decompose $F_n$ by setting
		\begin{equation*}
			F_n^1 = \left\{|Y_{\hat{\tau}_{2^n S}}^R|\leq \sqrt{\eps_n} 2^n S\right\}, \qquad \text{and} \qquad F_n^2 = F_n \setminus F_n^1. 
		\end{equation*}
Then 
		\begin{equation}\label{eqn::EnR_decompose}
			\sum_{n=0}^{\infty} \PPtworho \left[E_n^R \right] 
			\leq  \sum_{n=0}^{\infty} \PP \left[F_n^1 \right]+ \sum_{n=0}^{\infty} \PP \left[F_n^2 \right].
		\end{equation}
		Heuristically, the set $F_n^1$ represents the event where the configuration is ``already bad'' at the starting time $\hat{\tau}_{2^n S}$, whereas $F_n^2$ corresponds to the event that the configuration ``becomes bad'' during the time interval $[\hat{\tau}_{2^n S}, \tau_{2^{n+1}S}]$. 
		\medbreak
First,	 we estimate $\PP \left[F_n^1 \right]$. 
We set
\[
		X_s = \frac{Y_{t(s)}^R}{Y_{t(s)}^R - Y_{t(s)}^L}\in [0,1], \qquad s(t) = \frac{1}{2} \log \frac{Y_{t}^R - Y_{t}^L}{Y_0^R-Y_0^L}.
		\] 
as in Lemma~\ref{lem::Zhan_calculation}. 
Since $Y_0^L = 0-$, the process $(X_s)_{s>0}$ satisfies the SDE~\eqref{eqn::Zhan_calculation} with $X_0=1$. From~\eqref{eqn::chordalBessel_estimate1}, for all $s>0$ and $\eps\in (0,1/4)$, we have 
		\[
		\PP[X_s<\eps]\leq \frac{\Gamma\left({4(\alpha+1)}/{\kappa}\right)}{\Gamma\left({4}/{\kappa}\right)\Gamma\left({4\alpha}/{\kappa}\right)}\frac{\kappa}{4} \eps^{{4}/{\kappa}}. 
		\]
		Since $s(\hat{\tau}_S) = \frac{\kappa}{2} \log \frac{S}{x_1^R}$ is deterministic, we have 
		\begin{align}
		\sum_{n=0}^{\infty} \PP \left[F_n^1 \right] =& \sum_{n=0}^{\infty} \PP\left[X_{s(\hat{\tau}_{2^n S})} \leq \sqrt{\eps_n}\right] \nonumber \\
		\leq& \frac{\Gamma\left({4(\alpha+1)}/{\kappa}\right)}{\Gamma\left({4}/{\kappa}\right)\Gamma\left({4\alpha}/{\kappa}\right)}\frac{\kappa}{4} \sum_{n=0}^{\infty} (n+n_0)^{-{4}/{\kappa}}\leq C_{\eqref{eqn::E1R_estimate_constant}}(\kappa, \alpha) n_0^{1-{4}/{\kappa}}\label{eqn::E1R_estimate_constant},
		\end{align}
		where 
		\begin{equation*}
			C_{\eqref{eqn::E1R_estimate_constant}}(\kappa, \alpha):= \frac{\Gamma\left({4(\alpha+1)}/{\kappa}\right)}{\Gamma\left({4}/{\kappa}\right)\Gamma\left({4\alpha}/{\kappa}\right)}\frac{\kappa}{(4-\kappa)}.
		\end{equation*}
		By Stirling's formula, we have
		\[
		\limsup_{\kappa\to 0} \kappa\log C_{\eqref{eqn::E1R_estimate_constant}} = 4(\alpha+1)\log(\alpha+1) - 4\alpha\log\alpha \in (-\infty, +\infty).
		\]

		\medbreak
		Next, we estimate $\PP \left[F_n^2 \right]$. 
		Note that on the event $F_n^2$, we have $\left|Y_{\hat{\tau}_{2^n S}}^R\right| > \sqrt{\eps_n} 2^n S$. Also note that $\tau_{S}\leq \hat{\tau}_{12 S}$ due to~\eqref{eqn::chordal_bad_aux1}, thus by the strong Markov property of the It\^{o} diffusion, we have 
		\begin{align*}
			\PP \left[F_n^2 \right] \leq& \PP\left[\exists t\in [\hat{\tau}_{2^n S}, \hat{\tau}_{12S2^{n+1}}] \text{ s.t. } |Y_t^R|\leq \eps_n 2^n S; \left|Y_{\hat{\tau}_{2^n S}}^R\right| > \sqrt{\eps_n} 2^n S \right] \\
			\leq& \PP\left[\exists t\in [\hat{\tau}_{2^n S}, \hat{\tau}_{12S2^{n+1}}] \text{ s.t. } X_t \leq \eps_n ; X_{\hat{\tau}_{2^n S}} \geq  \frac{\sqrt{\eps_n}}{12}\right] \\
			\leq& \PP\left[\exists t\in [\hat{\tau}_{2^n S}, \hat{\tau}_{12S2^{n+1}}] \text{ s.t. } X_t \leq \eps_n \;\left|\; X_{\hat{\tau}_{2^n S}} = \frac{\sqrt{\eps_n}}{12} \right.\right] \\
			\leq & \PP\left[\exists s\in [0,T] \text{ s.t. } X_s \leq \eps_n \;\left |\; X_0 = \frac{\sqrt{\eps_n}}{12}\right.\right],
		\end{align*}
		where 
		\[T = \frac{1}{2}\log \frac{12S2^{n+1}}{Y^R_{\hat{\tau}_{2^n S}}-Y^L_{\hat{\tau}_{2^n S}}} = \frac{1}{2}\log 24.\]
		Now we can apply~\eqref{eqn::finite_time_upper_bound} to obtain
		\begin{align}
			\sum_{n=0}^{\infty} \PP \left[F_n^2 \right] &\leq \sum_{n=0}^{\infty} \exp\left(C_{\eqref{eqn::finite_time_upper_bound}}/\kappa\right)  T \left(\frac{\sqrt{\eps_n}}{12}\right)^{-{4}/{\kappa}} \eps_n^{{4}/{\kappa}-1} \nonumber\\
			&= \exp\left(C_{\eqref{eqn::finite_time_upper_bound}}/\kappa\right)  T 12^{{4}/{\kappa}} \sum_{n=0}^{\infty} (n+n_0)^{-{4}/{\kappa}+2} \nonumber \\
			&\leq\exp\left(C_{\eqref{eqn::finite_time_upper_bound}}/\kappa\right)  T 12^{{4}/{\kappa}} \sum_{n=0}^{\infty} (n+n_0)^{-{2}/{\kappa}}\leq C_{\eqref{eqn::E2R_estimate_constant}}(\kappa, \alpha) n_0^{1-2/\kappa} \label{eqn::E2R_estimate_constant}, 
		\end{align}
		where 
		\begin{equation*}
			C_{\eqref{eqn::E2R_estimate_constant}}(\kappa, \alpha) := \exp\left(C_{\eqref{eqn::finite_time_upper_bound}}/\kappa\right)  T 12^{{4}/{\kappa}} \frac{2}{2-\kappa},
		\end{equation*}
		and we have
		\[
		\limsup_{\kappa\to 0} \kappa\log C_{\eqref{eqn::E2R_estimate_constant}}= C_{\eqref{eqn::finite_time_upper_bound}} +  4\log 12  \in (-\infty, +\infty).
		\]
		Combining~\eqref{eqn::E1R_estimate_constant} and~\eqref{eqn::E2R_estimate_constant}, we obtain the desired estimate \eqref{eqn::chordal_bad_aux0} on the bad event $E_n^R$ with 
		\[
		C_{\eqref{eqn::chordal_bad_aux0}}(\alpha) := \sup_{0<\kappa<1}\kappa\log\left(C_{\eqref{eqn::E1R_estimate_constant}}(\kappa, \alpha) + C_{\eqref{eqn::E2R_estimate_constant}}(\kappa, \alpha)\right).
		\]
	\end{proof}

\section{Large deviation for half-watermelon SLE}
\label{sec::LDP_halfwatermelon}
The goal of this section is to prove large deviation in Theorem~\ref{thm::halfwatermelon_LDP}. 
To this end, let us first introduce the large deviation in finite-time setting.

\paragraph{Curve spaces. } Fix $n\geq 1$ and $(n+1)$-polygon $(\Omega; \bs{x}, y)$, we denote by $\chamber_{\bs{T}}=\chamber_{\bs{T}}(\Omega;\bs{x},y)$ the space of the simple multichords $\bs{\eta}_{[\bs{0},\bs{T}]}=(\eta^1_{[0,T_1]}, \ldots, \eta^n_{[0,T_n]})$ in $\Omega$ starting from $\bs{x}$, that is, 
\[
	\eta_{(0,T_j]}^j \subset \Omega \text{ for } 1\le j\le n, \quad \text{and furthermore }\quad \eta_{(0,T_j]}^j \cap \eta_{(0,T_k]}^k = \emptyset \quad \text{ for all } j\neq k.
\]
Here $\bs{\eta}_{[\bs{0}, \bs{T}]}$ is parameterized by $n$-time parameter of $\varphi_{\HH}(\bs{\eta})=(\varphi_{\HH}(\eta^1), \ldots, \varphi_{\HH}(\eta^n))$ (see~\eqref{eqn::multitime_def}), where $\varphi_{\HH}:\Omega\to \HH$ is a conformal map with $\varphi_{\HH}(y)=\infty$.
We equip a metric analogous to~\eqref{eqn::chamberfusion_metric} on the truncated curve space $\chamber_{\bs{T}}$:
\begin{equation} \label{eqn::metric}
	\dist_{\chamber}(\bs{\eta}_{[\bs{0}, \bs{T}]},\tilde{\bs{\eta}}_{[\bs{0},\bs{T}]}):=\sup_{\bs{t}\in [\bs{0}, \bs{\infty}]} \sup_{1\le j\le n} \left| \varphi_{\U} \left(\eta^j_{t_j \wedge T_j}\right)-\varphi_{\U}\left(\tilde{\eta}^j_{t_j \wedge T_j}\right) \right|, \quad \bs{\eta}_{[\bs{0}, \bs{T}]}, \tilde{\bs{\eta}}_{[\bs{0},\bs{T}]}\in \chamber_{\bs{T}}, 
\end{equation}
where $\varphi_{\U}: \Omega\to \mathbb{U}$ is a conformal map. In this way, we obtain an incomplete metric space $(\chamber_{\bs{T}}(\Omega;\bs{x},y), \dist_{\chamber})$. Note that the induced topology on $\chamber_{\bs{T}}$ does not depend on the choice of $\varphi_{\U}$, though the metric does.
\medbreak
Assume $D\subset \Omega$ is a subdomain which agrees with $\Omega$ in a neighborhood of $\bs{x}$ and has a positive distance from $y$. 
We define the exit time of $D$ for the curves:
\begin{equation*} \label{eqn::exittime}
	T_D^j:=\inf \{t_j \colon \eta_{t_j}^j\notin \overline{D}\}, \quad \bs{T}_D :=(T_D^1,\ldots,T_D^n).
\end{equation*}

\begin{proposition} \label{prop::finite_watermelon_LDP}
Fix $n\ge 1$ and $(n+1)$-polygon $(\Omega;\bs{x},y)$. Assume $D\subset \Omega$ is a simply connected subdomain which agrees with $\Omega$ in a neighborhood of $\bs{x}$ and has a positive distance from $y$.
Suppose $\bs{\eta}\sim \PPfusion{n}(\Omega;\bs{x},y)$ is half-$n$-watermelon $\SLE_{\kappa}$ in $(\Omega;\bs{x},y)$. Then the family of laws of $\bs{\eta}_{[\bs{0}, \bs{T}_D]}$ under $\PPfusion{n}(\Omega;\bs{x},y)$ satisfies large deviation principle in the space $(\chamber_{\bs{T}_D}(\Omega; \bs{x}, y), \dist_{\chamber})$ as $\kappa\to 0+$ with good rate function $\ratefusion{n} (\Omega; \bs{x}, y; \cdot)$ which will be defined in Definition~\ref{Def::multitime_energy_chordal}. 
\end{proposition}
The rest of this section is organized as follows. In Section~\ref{subsec::multitime_index}, we introduce the chordal multi-time parameter, and  recall the definition and basic properties of multi-time chordal Loewner energy from~\cite{ChenHuangPeltolaWumultitimeenergy}. In Section~\ref{sec::LDP_finite_time_halfwatermelon}, we establish the large deviation of half-watermelon $\SLE$ in finite time (Proposition~\ref{prop::finite_watermelon_LDP}). In Section~\ref{subsec::LDP_infinite_time_halfwatermelon}, we extend this result to the infinite-time setting, using Proposition~\ref{prop::finite_watermelon_LDP} and the return estimate in Proposition~\ref{prop::SLE_return_chordal}.
In Section~\ref{subsec::proof_rho_LDP}, we adapt the procedures from  Sections~\ref{sec::LDP_finite_time_halfwatermelon}-\ref{subsec::LDP_infinite_time_halfwatermelon} to prove the LDP for SLE with force points (Proposition~\ref{prop::LDP_SLEkapparho}).

The next two sections present direct consequences of large deviation of half-watermelon $\SLE$. In Section~\ref{subsec::common_time_chordal}, we prove the LDP for Dyson Brownian motion (Proposition~\ref{prop::LDP_DysonBM_chordal}), utilizing the relation between half-watermelon $\SLE$ and Dyson Brownian motion (see Lemma~\ref{lem::DysonBM_driving_function}). In Section~\ref{subsec::bp_chordal_proof}, we prove the boundary perturbation property of multi-time chordal Loewner energy (Proposition~\ref{prop::bp_chordal}) using the boundary perturbation property of half-watermelon SLE proved in~\cite[Proposition~3.4]{HuangPeltolaWuMultiradialSLEResamplingBP} (see Lemma~\ref{lem::halfwatermelon_bp}) and large deviation principle in Theorem~\ref{thm::halfwatermelon_LDP}.

\subsection{Multi-time chordal Loewner energy}\label{subsec::multitime_index}

\paragraph*{Multi-time parameter (chordal).}
Fix $n\ge 1$ and $\bs{x}=(x_1, \ldots, x_n)\in \LX_n^{\HH}$. Consider an $n$-tuple $\bs{\eta}_{[\bs{0},\bs{t}]}=(\eta^1_{[0,t_1]}, \ldots, \eta^n_{[0,t_n]})\in \overline{\chamber}_{\bs{t}}(\HH;\bs{x},\infty)$\footnote{That is, the curves are allowed to touch themselves or each other, but not to cross.}  parameterized by $\bs{t}=(t_1, \ldots, t_n)\in[0,\infty)^n$.
We define the following normalized conformal transformations:
\begin{itemize}
	\item $g_{t_j}^j$ is the conformal map from the unbounded component of $\HH\setminus\eta^j_{[0,t_j]}$ onto $\HH$ with $\lim_{z\to\infty}|g_{t_j}^j(z)-z|=0$, $1\le j\le n$.
	\item $g_{\bs{t}}$ is the conformal map from the unbounded component of $\HH\setminus\cup_{j=1}^n \eta^j_{[0,t_j]}$ onto $\HH$ with $\lim_{z\to\infty}|g_{\bs{t}}(z)-z|=0$.
	\item $g_{\bs{t}, j}$ is the conformal map from the unbounded component of $\HH\setminus g_{t_j}^j\left(\cup_{i\neq j}\eta^i_{[0,t_i]}\right)$ onto $\HH$ with $\lim_{z\to\infty}|g_{\bs{t}, j}(z)-z|=0$, $1\le j\le n$. 
\end{itemize}
Using these notations, we have $g_{\bs{t}}=g_{\bs{t}, j}\circ g_{t_j}^j$ for $1\le j\le n$. We say that the $n$-tuple $\bs{\eta}_{\bs{t}}$ of curves has $n$-time parameter if $g^{j}$ is parameterized by the half-plane capacity: 
\begin{equation}\label{eqn::multitime_def}
	g_{t_j}^j(z)=z+\frac{2t_j}{z}+o\left(|z|^{-1}\right),\qquad \text{as }z\to\infty, \qquad \text{ for } 1\le j\le n.
\end{equation}
Denote by $W^j$ the driving function of each $\eta^j$, and we define the driving function of the $n$-tuple $\bs{\eta}_{\bs{t}}$, started at $\bs{X}_0=(x_1, \ldots, x_n)\in \LX_n^{\HH}$, by
\begin{equation}\label{eqn::multitime_driving}
	\bs{X}_{\bs{t}}=(X^1_{\bs{t}}, \ldots, X^n_{\bs{t}}), \qquad \text{with }X^j_{\bs{t}}=g_{\bs{t}, j}(W^j_{t_j}), \qquad \text{for }1\le j\le n.
\end{equation}
A standard calculation shows that 
 \begin{equation}\label{eqn::halfwatermelonSLE0_minimizer_aux2}
    \ud X_{\bs{t}}^j = \left(g_{\bs{t},j}'\left(W_{t_j}^j\right)\dot{W}_{t_j}^j - 3g_{\bs{t},j}''\left(W_{t_j}^j\right)\right) \ud t_j+ \sum_{\ell\ne j} \frac{2g_{\bs{t},\ell}'\left(W_{t_\ell}^\ell\right)^2}{X_{\bs{t}}^j - X_{\bs{t}}^{\ell}} \ud t_\ell, \quad\text{for }1\le j\le n.
\end{equation}
The mapping-out function $g_{\bs{t}}$ solves the Loewner equation
\begin{equation*}\label{eqn::multitime_Loewner_chordal}
\ud g_{\bs{t}}(z)=\sum_{j=1}^n \partial_{t_j}g_{\bs{t}}(z)\ud t_j=\sum_{j=1}^n \frac{2 g_{\bs{t},j}'(W_{\bs{t}}^j)^2}{g_{\bs{t}}(z)-X_{\bs{t}}^j}\ud t_j.
\end{equation*}

Fix $n\geq 1$ and $(n+1)$-polygon $(\HH; \bs{x}, \infty)$ with $\bs{x}=(x_1, \ldots, x_n)\in \LX_n^{\HH}$. 
Consider an $n$-tuple $\bs{\eta}_{[\bs{0},\bs{t}]}=(\eta^1_{[0,t_1]}, \ldots, \eta^n_{[0,t_n]})$ of disjoint simple curves in $\HH$ parameterized by multi-time parameter. Define 
$\blm_{\bs{t}}$ to be the unique potential solving the exact differential equation (see~\cite[Lemma~2.2]{HuangWuYangMultipleSLEsDysonBM})
\begin{equation}\label{eqn::mt_def}
	\ud \blm_{\bs{t}}=\sum_{j=1}^n \partial_{t_j}\blm_{\bs{t}}\ud t_j=\sum_{j=1}^n -\frac{1}{3}\LS g_{\bs{t}, j}(W^j_{t_j})\ud t_j,\qquad \blm_{\bs{0}}=0,
\end{equation}
with $\LS g=\frac{g'''}{g'}-\frac{3}{2}\left(\frac{g''}{g'}\right)^2$ denoting the Schwarzian derivative of a function $g$. It is proved in~\cite[Lemma~2.3]{HuangWuYangMultipleSLEsDysonBM} that the solution to~\eqref{eqn::mt_def} 
	can be described in terms of Brownian loop measure as
	\begin{equation*}\label{eqn::mt_blm}
		\blm_{\bs{t}} = \blm \left(\HH;\eta_{[0,t_1]}^{1},\ldots,\eta_{[0,t_n]}^{n} \right).
	\end{equation*}
	Consequently, the measure $\blm_{\bs{t}}$ is finite as long as $\eta_{[0,t_1]}^{1},\ldots,\eta_{[0,t_n]}^{n}$ are disjoint.

\paragraph*{Semi-classical limits of BPZ equations. }
We define 
\begin{align}\label{eqn::halfwatermelonSLE0_pf_def_H}
\begin{split}
	\LU_{\shuffle_n}(\bs{x}) = -2 \sum_{1\le i<j\le n}\log (x_j-x_i), \quad \text{for }\bs{x}=(x_1, \ldots, x_n)\in\LX_n^{\HH}. 
\end{split}
\end{align}
The function $\LU_{\shuffle_n}:\LX_n^{\HH}\to \R$ is the semi-classical limit of the partition function $\LZ_{\shuffle_n}^{({\kappa})}$ in~\eqref{eqn::halfwatermelon_pf_H}: 
\[\LU_{\shuffle_n}(\bs{x})=-\lim_{\kappa\to 0}\kappa\log\LZ_{\shuffle_n}^{(\kappa)}(\bs{x});\]
and it satisfies the semi-classical limits of the BPZ equations:
\begin{equation}\label{eqn::BPZ_sc}
	(\partial_j \LU)^2 - \sum_{\ell\neq j} \left( \frac{4}{x_{\ell}-x_j} \partial_{\ell} \LU + \frac{12}{(x_\ell-x_j)^2} \right)=0, \qquad \text{for }1\le j\le n.
\end{equation}

\begin{definition}[Multi-time chordal Loewner energy]\label{Def::multitime_energy_chordal}
Fix $n\ge 2$ and $(n+1)$-polygon $(\Omega; \bs{x}, y) = (\Omega; x_1, \ldots, x_n, y)$. 
Let $\varphi: \Omega\to \HH$ be a conformal map with $\varphi(y)=\infty$. 
For $\bs{\eta}_{[\bs{0},\bs{t}]}\in\overline{\chamber}_{\bs{t}}(\Omega; \bs{x}, y)$, we denote by $W_{t_j}^j$ the driving function of $\varphi(\eta^j_{[0,t_j]})$ and denote by $\bs{X}_{\bs{t}}$ the driving function of the $n$-tuple $\varphi(\bs{\eta}_{[\bs{0},\bs{t}]})$.
We define the \emph{truncated $n$-time chordal Loewner energy} $\ratefusion{n} (\Omega; \bs{x}, y; \bs{\eta}_{[\bs{0}, \bs{t}]})$ to be 
\begin{equation} \label{eqn::multi_time_energy_chordal_def}
\begin{split}
	\ratefusion{n} (\Omega; \bs{x}, y; \bs{\eta}_{[\bs{0}, \bs{t}]}) \; =& \; \sum_{j=1}^{n} \rate(\Omega; x_j, y; \eta_{[0, t_j]}^j) +12 \blm(\Omega; \eta^1_{[0,t_1]}, \ldots, \eta^n_{[0,t_n]})\\
	& -3\sum_{j=1}^{n} \log g'_{\bs{t},j}(W^j_{t_j})+ \LU_{\shuffle_n}(\bs{X}_{\bs{t}}) -\LU_{\shuffle_n}(\bs{X}_{\bs{0}}). 
\end{split}
\end{equation}
Here, the quantity $\rate$ is the chordal Loewner energy defined in~\eqref{eqn::LDP_rate_function_standard}, and $m$ is the Brownian loop measure defined in~\eqref{eqn::mt_def}, and $\LU_{\shuffle_n}(\bs{x})$ is defined in~\eqref{eqn::halfwatermelonSLE0_pf_def_H}. 
Using~\eqref{eqn::BPZ_sc}, it is proved in~\cite{ChenHuangPeltolaWumultitimeenergy} that $\ratefusion{n} (\Omega; \bs{x}, y; \bs{\eta}_{[\bs{0}, \bs{t}]})$ is the unique solution to the following system of exact differential equations:
\begin{equation}
\begin{split}\label{eqn::multi_time_energy_chordal_def2}
	\ud \ratefusion{n}(\Omega; \bs{x}, y; \bs{\eta}_{[\bs{0}, \bs{t}]})
	= \; & \sum_{j=1}^{n} \partial_{t_j} \ratefusion{n}(\Omega; \bs{x}, y; \bs{\eta}_{[\bs{0}, \bs{t}]}) \ud t_j \\
	= \; & \frac{1}{2}\sum_{j=1}^{n} \bigg( \dot{W}_{t_j}^{j} - 2 \sum_{i\neq j} \frac{ g'_{\bs{t},j}(W_{t_j}^{j})}{X_{\bs{t}}^{j} - X_{\bs{t}}^{i}} 
	- 3 \, \frac{g''_{\bs{t},j}(W_{t_j}^{j})}{g'_{\bs{t},j}(W_{t_j}^{j})}\bigg)^2 \ud t_j ,
	\\
	\ratefusion{n}(\Omega; \bs{x}, y; \bs{\eta}_{[\bs{0}, \bs{0}]})
	= \; & 0.
\end{split}
\end{equation}
In particular, the quantity $\ratefusion{n} (\Omega; \bs{x}, y; \bs{\eta}_{[\bs{0}, \bs{t}]})$ is increasing in $\bs{t}$, and we define \emph{$n$-time chordal Loewner energy} of $\bs{\eta}\in\barchamberfusion{n}(\Omega; \bs{x}, y)$ to be its limit: 
\begin{equation}\label{eqn::multi_time_energy_chordal}
	\ratefusion{n}(\Omega; \bs{x}, y; \bs{\eta}):=\lim_{\bs{t}\to +\bs{\infty}} \ratefusion{n}(\Omega; \bs{x}, y; \bs{\eta}_{[\bs{0}, \bs{t}]}). 
\end{equation}
\end{definition}

The following basic property of the energy will be used in the proof of Proposition~\ref{prop::finite_watermelon_LDP}. 

\begin{lemma} \label{lem::finite_energy_simple_multichord}
If $\ratefusion{n}(\Omega;\bs{x},y;\bs{\eta}_{[\bs{0},\bs{T}]})<\infty$, then the component curves $\eta^j_{[0,T_j]}$ are simple and mutually disjoint for $1\le j\le n$.
\end{lemma}

The proof of Lemma~\ref{lem::finite_energy_simple_multichord} relies on the following Lemma~\ref{lem::finite_energy_simplechord}.
\begin{lemma} \label{lem::finite_energy_simplechord}
Suppose $T<\infty$ and $W_{[0,T]}$ has finite Dirichlet energy~\eqref{eqn::LDP_rate_function_standard}. Suppose $\lambda:[0,T]\to (0,\infty)$ is a weight function that $0<\inf_{t\in [0,T]} \lambda_t\le \sup_{t\in [0,T]} \lambda_t< \infty$. Consider the chordal Loewner chain with weight function $\lambda$:
\begin{equation} \label{eqn::lambda_Loewnerchain}
	\partial_t g_t(z)=\frac{2\lambda_t}{g_t(z)-W_t}, \quad g_0(z)=z, \quad 0<t\le T,
\end{equation}
then the hull $\eta_{[0,T]}$ generated by the Loewner chain~\eqref{eqn::lambda_Loewnerchain} is a simple curve in $\HH$.
\end{lemma}

\begin{proof}
Since $0<\inf_{t\in [0,T]} \lambda_t\le \sup_{t\in [0,T]} \lambda_t< \infty$, the previous result~\cite[Lemma~3.8]{AbuzaidHealeyPeltolaLargeDeviationDysonBM} implies that $\eta_{[0,T]}$ has finite Loewner energy when parameterized by half-plane capacity. Then~\cite[Theorem~2(i)]{FrizShekharSLEtracefiniteenergydrivers} implies that $\eta_{[0,T]}$ is a simple curve in $\HH$.
\end{proof}

\begin{proof}[Proof of Lemma~\ref{lem::finite_energy_simple_multichord}]
It suffices to show the conclusion for $(\Omega;\bs{x},y)=(\HH;\bs{x},\infty)$ and we eliminate them from the notation. We show that $\eta_{(0,T_1]}^1$ is a simple curve in $\HH\setminus \left( \cup_{j=2}^n \eta_{[0,T_j]}^j \right)$. To this end, we will prove that
\begin{equation} \label{eqn::finite_energy_simple_multichord_TODO}
	\inf_{\stackrel{\bs{t}\in [\bs{0},\bs{T}]}{2\le j\le n}} \left| X_{\bs{t}}^j - X_{\bs{t}}^1 \right|>0, \qquad \text{ and } \qquad \inf_{\bs{t}\in [\bs{0},\bs{T}]} g'_{\bs{t},1} (W_{t_1}^1)^2>0.
\end{equation}

Fix $(t_2,\ldots,t_n)$ with $0\le t_j\le T_j$ for $2\le j\le n$ and denote $\bs{s}=(t_2,\ldots,t_n)$. We view $\bs{t}=(t_1,\bs{s})$ as a function in $t_1\in [0,T_1]$. By \eqref{eqn::multi_time_energy_chordal_def2}, we have
\begin{align}
	\ratefusion{n}(\bs{\eta}_{[\bs{0},(T_1,\bs{s})]}) \ge & \frac{1}{2} \int_{0}^{T_1} \left( \dot{W}_{t_1}^1 - g'_{\bs{t},1} (W_{t_1}^1) \sum_{j=2}^{n} \frac{2}{X_{\bs{t}}^1 - X_{\bs{t}}^j} - 3 \frac{g''_{\bs{t},1}(W_{t_1}^1)}{g'_{\bs{t},1}(W_{t_1}^1)} \right)^2 \ud t_1 \nonumber \\
	= & \frac{1}{2} \int_{0}^{T_1} \left( \frac{\partial_{t_1} X_{\bs{t}}^1}{g'_{\bs{t},1} (W_{t_1}^1)} - g'_{\bs{t},1} (W_{t_1}^1) \sum_{j=2}^{n} \frac{2}{X_{\bs{t}}^1 - X_{\bs{t}}^j} \right)^2 \ud t_1 \tag{due to~\eqref{eqn::halfwatermelonSLE0_minimizer_aux2}}\\
	\ge & \underbrace{\frac{1}{2} \int_{0}^{T_1} \left( \frac{\partial_{t_1} X_{\bs{t}}^1}{g'_{\bs{t},1} (W_{t_1}^1)} \right)^2 \ud t_1}_{:=J_1} \underbrace{- \int_{0}^{T_1} \partial_{t_1} X_{\bs{t}}^1 \sum_{j=2}^{n} \frac{2}{X_{\bs{t}}^1 - X_{\bs{t}}^j} \ud t_1}_{:=J_2}. \label{eqn::finite_energy_simple_multichord_aux1}
\end{align}
For $J_1$, we have
\begin{equation} \label{eqn::finite_energy_simple_multichord_aux2}
	J_1\ge \frac{1}{2} \int_{0}^{T_1} \left( \partial_{t_1} X_{\bs{t}}^1 \right)^2 \ud t_1\geq 0.
\end{equation}
For $J_2$, we denote $f(\bs{t})|_{0}^{T_1}:=f(\bs{t})|_{t_1=T_1}-f(\bs{t})|_{t_1=0}$ for functional $f$ on $\bs{t}$, then we have
\begin{align} 
	J_2 = & - \sum_{j=2}^{n} \int_{0}^{T_1} \partial_{t_1} (X_{\bs{t}}^1 - X_{\bs{t}}^j) \frac{2}{X_{\bs{t}}^1 - X_{\bs{t}}^j} \ud t_1 - \sum_{j=2}^{n} \int_{0}^{T_1} \partial_{t_1} (X_{\bs{t}}^j) \frac{2}{X_{\bs{t}}^1 - X_{\bs{t}}^j} \ud t_1 \notag \\
	\ge & - \sum_{j=2}^{n} \int_{0}^{T_1} \partial_{t_1} (X_{\bs{t}}^1 - X_{\bs{t}}^j) \frac{2}{X_{\bs{t}}^1 - X_{\bs{t}}^j} \ud t_1 \tag{due to~\eqref{eqn::halfwatermelonSLE0_minimizer_aux2}} \\
	= & - 2 \sum_{j=2}^{n} \log (X_{\bs{t}}^j - X_{\bs{t}}^1)|_{0}^{T_1}. \label{eqn::finite_energy_simple_multichord_aux3}
\end{align}
Combining~(\ref{eqn::finite_energy_simple_multichord_aux1},\ref{eqn::finite_energy_simple_multichord_aux2},\ref{eqn::finite_energy_simple_multichord_aux3}), we have
\begin{equation*}
	- 2 \sum_{j=2}^{n} \log (X_{\bs{t}}^j - X_{\bs{t}}^1)|_{0}^{T_1} \le \ratefusion{n}(\bs{\eta}_{[\bs{0},(T_1,\bs{s})]}) <\infty.
\end{equation*}
This implies the first control in~\eqref{eqn::finite_energy_simple_multichord_TODO}. The first control in~\eqref{eqn::finite_energy_simple_multichord_TODO} implies that $\eta_{[0,T_1]}^1$ has a positive distance from $\cup_{j=2}^n \eta_{[0,t_j]}^j$, which gives the second control in~\eqref{eqn::finite_energy_simple_multichord_TODO}. Moreover, combing~(\ref{eqn::finite_energy_simple_multichord_aux1},\ref{eqn::finite_energy_simple_multichord_aux2},\ref{eqn::finite_energy_simple_multichord_aux3}) again, the process $X_{(\cdot,T_2,\ldots,T_n)}^1$ has finite Dirichlet energy up to time $T_1$. Since $X_{(\cdot,T_2,\ldots,T_n)}^1$ is the driving function of $g_{(0,T_2,\ldots,T_n)}\left(\eta_{[0,T_1]}^1\right)$ with weight function $g'_{\bs{t},1}(W_{t_1}^1)^2$, combining~\eqref{eqn::finite_energy_simple_multichord_TODO} and Lemma~\ref{lem::finite_energy_simplechord}, the process $g_{(0,T_2,\ldots,T_n)}\left(\eta_{[0,T_1]}^1\right)$ is a simple curve in $\HH$, which implies that $\eta_{[0,T_1]}^1$ is a simple curve in $\HH\setminus \left( \cup_{j=2}^n \eta_{[0,T_j]}^j \right)$. Similarly, the process $\eta_{[0,T_j]}^j$ is a simple curve in $\HH\setminus \left( \cup_{\ell\neq j} \eta_{[0,T_{\ell}]}^{\ell} \right)$ for $2\le j\le n$, and we obtain the desired result.
\end{proof}

\subsection{Large deviation in finite time: proof of Proposition~\ref{prop::finite_watermelon_LDP}}\label{sec::LDP_finite_time_halfwatermelon}
In this section, we prove Proposition~\ref{prop::finite_watermelon_LDP}. The proof relies on the multi-time martingale for half-watermelon SLE presented in Lemma~\ref{lem::halfwatermelon_mart} and the large deviation principle for SLE curves in finite time presented in Lemma~\ref{lem::finite_single_LDP}.

\begin{lemma}[{\cite[Lemma~2.5]{HuangWuYangMultipleSLEsDysonBM}}]
\label{lem::halfwatermelon_mart}
Fix $\kappa\in (0,4], n\ge 2$ and $\bs{x}=(x_1, \ldots, x_n)\in\LX_n^{\HH}$. For each $j\in\{1, \ldots, n\}$, let $\eta^j$ be chordal $\SLE_{\kappa}$ in $(\HH; x_j,\infty)$ and let $\PP_n^{(\kappa)}$ be the probability measure on $\bs{\eta}=(\eta^1, \ldots, \eta^n)$ under which the curves are independent. We parameterize $\bs{\eta}$ by $n$-time parameter $\bs{t}$ and let $\bs{X}_{\bs{t}}$ denote the driving function as in~\eqref{eqn::multitime_driving}. Define the process 
\begin{equation}\label{eqn::halfwatermelon_multitime_mart}
	M_{\bs{t}}(\LZ_{\shuffle_n}^{(\kappa)}):=\one_{\LE_{\emptyset}(\bs{\eta}_{[\bs{0},\bs{t}]})} \prod_{j=1}^n g'_{\bs{t},j}(W^j_{t_j})^{\mathfrak{b}}\times \exp\left(\frac{\mathfrak{c}}{2}\blm_{\bs{t}}\right)\times \LZ_{\shuffle_n}^{(\kappa)}(\bs{X}_{\bs{t}}),
\end{equation}
where $\LE_{\emptyset}(\bs{\eta}_{[\bs{0},\bs{t}]})=\{\eta^i_{[0,t_i]}\cap\eta^j_{[0,t_j]}=\emptyset, \forall i\neq j\}$ is the event that different curves are disjoint, and
\begin{equation}\label{eqn::parameters_b_c}
\mathfrak{b}=\frac{6-\kappa}{2\kappa}, \qquad \mathfrak{c}=\frac{(6-\kappa)(3\kappa-8)}{2\kappa}, 
\end{equation}
and $\blm_{\bs{t}}$ is defined in~\eqref{eqn::mt_def}, and $\LZ_{\shuffle_n}^{(\kappa)}$ is the partition function defined in~\eqref{eqn::halfwatermelon_pf_H}. 
Then $M_{\bs{t}}(\LZ_{\shuffle_n}^{(\kappa)})$ is an $n$-time parameter local martingale with respect to $\PP_n^{(\kappa)}$. Moreover, the law of $\PP_n^{(\kappa)}$ weighted by $M_{\bs{t}}(\LZ_{\shuffle_n}^{(\kappa)})$ is the same as half-$n$-watermelon $\SLE_{\kappa}$ in $(\HH; \bs{x}, \infty)$ when restricted to the event $\LE_{\emptyset}(\bs{\eta}_{[\bs{0},\bs{t}]})$. 
\end{lemma}

If we set $t_1=t\ge 0$ and $t_2=\cdots=t_n=0$, the martingale~\eqref{eqn::halfwatermelon_multitime_mart} becomes 
\begin{equation*}\label{eqn::chordalSLEkappa2_mart}
M_t(\LZ_{\shuffle_n}^{(\kappa)})=\prod_{j=2}^n g_t'(x_j)^{\mathfrak{b}}\times \LZ_{\shuffle_n}^{(\kappa)}(W_t, g_t(x_2), \ldots, g_t(x_n)), 
\end{equation*}
which coincides with \eqref{eqn::SLEkapparho_martingale} with $\rho^L_{\ell}=\cdots=\rho^L_1=0$ and $\rho^R_1=\cdots=\rho^R_r=2$.

\begin{lemma} \label{lem::finite_single_LDP}
Fix $2$-polygon $(\Omega;x,y)$ and assume $D\subset \Omega$ is a simply connected subdomain which agrees with $\Omega$ in a neighborhood of $x$ and has a positive distance from $y$. Suppose $\eta\sim \PPtwo(\Omega; x, y)$  is chordal $\SLE_{\kappa}$ in $(\Omega;x,y)$. Then the family of laws of $\eta_{[0,T_D]}$ under $\PPtwo(\Omega; x, y)$ satisfies large deviation principle in the space $(\chamber_{T_D}(\Omega; x, y), \dist_{\chamber})$ as $\kappa\to 0+$ with good rate function $\rate(\Omega; x, y; \cdot)$ given by~\eqref{eqn::LDP_rate_function_standard}:
\begin{equation*}
	\rate(\Omega; x, y; \eta_{[0, T_D]})	:= \; \frac{1}{2} \int_{0}^{T_D} \left(\dot{W}_t\right)^2 \ud t,
\end{equation*}
if $W$ is absolutely continuous up to $T_D$; and setting $\rate(\Omega; x, y; \eta_{[0, T_D]}):= +\infty$ otherwise.
\end{lemma}
\begin{proof}
By Lemma~\ref{lem::LDP_SLEkappa_standard}, the family $\{\PPtwo(\Omega; x, y)\}_{\kappa\in(0,4]}$ of laws of chordal $\SLE_\kappa$ satisfies large deviation principle in the space $(\chamber(\Omega; x, y), \dist_{\chamber})$ as $\kappa\to 0+$ with good rate function $\rate(\Omega; x, y; \cdot)$. Since the projection
\begin{equation*}
\begin{cases}
	 (\chamber(\Omega; x, y), \dist_{\chamber}) \to (\chamber_{T_D}(\Omega; x, y), \dist_{\chamber}), \\
	 \eta \mapsto \eta_{[0, T_D]},
\end{cases}
\end{equation*}
is a continuous map, the contraction principle implies the desired result.
\end{proof}

\begin{proof}[Proof of Proposition~\ref{prop::finite_watermelon_LDP}]
It suffices to show the conclusion for $(\Omega;\bs{x},y)=(\HH;\bs{x},\infty)$ and we eliminate them from the notation. We simply write 
\begin{align*}
&\chamber_{\bs{t}}=\chamber_{\bs{t}}(\HH;\bs{x},\infty), \quad \ratefusion{n}(\HH;\bs{x},\infty;\bs{\eta}_{[\bs{0}, \bs{t}]})=\ratefusion{n}(\bs{\eta}_{[\bs{0}, \bs{t}]}). 
\end{align*}
For each $j\in \{1,\ldots, n\}$, let $\eta^{j}$ be chordal $\SLE_{\kappa}$ in $(\HH; x_j, \infty)$ and let $\PP_{n}^{(\kappa)}$ be the probability measure on $\bs{\eta}=(\eta^{1}, \ldots, \eta^{n})$ under which the curves are independent. Let $\E_{n}^{(\kappa)}$ be the expectation under the probability measure $\PP_{n}^{(\kappa)}$.
Note that $\chamber_{\bs{T}_D}$ is an open topological subspace of $\prod_{j=1}^{n} \chamber_{T_D^j}(\HH;x_j,\infty)$. 
We have the following two observations. 
\begin{itemize}
	\item A direct consequence of Lemma~\ref{lem::finite_single_LDP} implies that the law of $\bs{\eta}_{[\bs{0},\bs{T}_D]}$ under $\PP_{n}^{(\kappa)}$ satisfies large deviation principle in the space \[\left( \prod_{j=1}^{n} \chamber_{T_D^j}(\HH;x_j,\infty), \dist_{\chamber} \right)\] as $\kappa\to 0+$ with good rate function 
	\begin{equation}\label{eqn::Varadhan_application_aux1}
		\rate_{\oplus n}(\bs{\eta}_{[\bs{0},\bs{T}_D]}):=\frac{1}{2}\sum_{j=1}^n \int_0^{T_D^j}\left(\dot{W}^j_{t_j}\right)^2\ud t_j. 
	\end{equation}
	\item For $\bs{\eta}_{[\bs{0}, \bs{T}]} \in {\chamber}_{\bs{T}}$, we define
\begin{align}
	\Phi(\bs{\eta}_{[\bs{0}, \bs{T}]};\kappa)=&\kappa\log \frac{M_{\bs{T}}(\LZ_{\shuffle_n}^{(\kappa)})}{M_{\bs{0}}(\LZ_{\shuffle_n}^{(\kappa)})}\label{eqn::Def_Phi}\nonumber	\\
	=& \underbrace{- \LU_{\shuffle_n} (\bs{X}_{\bs{T}}) + \LU_{\shuffle_n}(\bs{x}) }_{\Phi_0(\bs{\eta}_{[\bs{0}, \bs{T}]}):=}
	\underbrace{ - \blm_{\bs{T}} }_{\Phi_1(\bs{\eta}_{[\bs{0}, \bs{T}]}):=} \underbrace{ \frac{(6-\kappa)(8-3\kappa)}{4} }_{f_1(\kappa):=} + \underbrace{ (3-\kappa/2) }_{f_2(\kappa):=} \underbrace{ \sum_{j=1}^n \log g_{\bs{T},j}'\left(W_{T_j}^{j}\right) }_{\Phi_2(\bs{\eta}_{[\bs{0}, \bs{T}]}):=}. \notag
\end{align}
By Lemma~\ref{lem::halfwatermelon_mart}, for any subset $A\subset \chamber_{\bs{T}_D}$, we have
\begin{equation*}
	\PPfusion{n} [A]= \E_{n}^{(\kappa)} \left[ \exp\left( \frac{1}{\kappa} \Phi(\bs{\eta}_{[\bs{0}, \bs{T}_D]};\kappa)\right) \one\{\bs{\eta}_{[\bs{0}, \bs{T}_D]}\in A\}\right]. 
\end{equation*}
\end{itemize}
To apply the generalization of Varadhan's lemma (see Lemma~\ref{lem::Varadhan} and Remark~\ref{rmk::Varadhan}), we have the following observations.
\begin{itemize}
\item Assume $\left\{ \bs{\eta}_{[\bs{0}, \bs{T}_D]}^{(n)} \right\}$ is a sequence in $(\chamber_{\bs{T}_D},\dist_{\chamber})$ converging to some $\bs{\eta}_{[\bs{0}, \bs{T}_D]}^{(\infty)}\in \overline{\chamber}_{\bs{T}_D}$ and $\bs{X}_{[\bs{0},\bs{T}_D]}^{(n)}$ is the driving function of $\bs{\eta}_{[\bs{0}, \bs{T}_D]}^{(n)}$ for $n\in \Z_+ \cup \{\infty\}$. 
Then by Carathéodory convergence theorem (see e.g.~\cite[Theorem~1.8]{Pommerenke}), the process $\bs{\eta}_{[\bs{0}, \bs{T}_D]}^{(n)}$ converges to $\bs{\eta}_{[\bs{0}, \bs{T}_D]}^{(\infty)}$ in the Carathéodory sense. By Loewner-Kufarev theorem (see e.g.~\cite[Theorem~8.5]{Berestycki2011LecturesOS}), the process $\bs{X}_{[\bs{0},\bs{T}_D]}^{(n)}$ converges to $\bs{X}_{[\bs{0},\bs{T}_D]}^{(\infty)}$ in the sense of uniform convergence. Hence, tha maps $\Phi_0(\bs{\eta}_{[\bs{0}, \bs{T}_D]})$ and $\Phi_j(\bs{\eta}_{[\bs{0}, \bs{T}_D]})$ are continuous for $1\le j\le 2$ and the condition~(a) in Lemma~\ref{lem::Varadhan} holds.
\item  We have $g_{\bs{T},j}'\left(\xi_{T_j}^{j}\right)\le 1$ and $\blm_{\bs{T}}\ge 0$. Thus $(f_1(\kappa),\Phi_1(\bs{\eta}_{[\bs{0}, \bs{T}_D]}))$ and $(f_2(\kappa),\Phi_2(\bs{\eta}_{[\bs{0}, \bs{T}_D]}))$ satisfy condition~(c) in Lemma~\ref{lem::Varadhan}.
\item Since $D$ has a positive distance from $\infty$, there exists $R>0$ such that $D\subset B(0,R)$. By~\cite[Lemma~4.5]{KemppainenSchrammLoewnerevolution}, we have $|X_{\bs{T}_D}^j|\le 3R$. Thus 
\[
	\Phi_0(\bs{\eta}_{[\bs{0}, \bs{T}_D]})\le \frac{n(n-1)}{2} \log(6R)+ \LU_{\shuffle_n}(\bs{x}),
\]
and condition~(d') in Lemma~\ref{lem::Varadhan} holds.
\item The function $\Phi(\bs{\eta}_{[\bs{0}, \bs{T}_D]};\kappa)$ is continuous in $\kappa$ and we have
\begin{align*} \label{eqn::Def_Phi0}
\begin{split}
	\Phi(\bs{\eta}_{[\bs{0}, \bs{T}_D]};0)=\lim_{\kappa\to 0} \Phi(\bs{\eta}_{[\bs{0}, \bs{T}_D]};\kappa)	= - \LU_{\shuffle_n} (\bs{X}_{\bs{T}_D}) + \LU_{\shuffle_n}(\bs{x}) 
	- 12\blm_{\bs{T}_D}  +  3\sum_{j=1}^n \log g_{\bs{T}_D,j}'\left(\xi_{T_D^j}^{j}\right). 	
\end{split}
\end{align*}
\end{itemize}
Combining the observations above and applying Lemma~\ref{lem::Varadhan}, we obtain that the law of $\bs{\eta}_{[\bs{0},\bs{T}_D]}$ under $\PPfusion{n}$ satisfies large deviation principle in the space $(\chamber_{\bs{T}_D},\dist_{\chamber})$ as $\kappa\to 0+$ with rate function
\begin{equation*}\label{eqn::Varadhan_application_aux2}
	\rate_{\oplus n}(\bs{\eta}_{[\bs{0}, \bs{T}_D]})- \Phi(\bs{\eta}_{[\bs{0}, \bs{T}_D]};0)\stackrel{\eqref{eqn::multi_time_energy_chordal_def}}{=}\ratefusion{n}(\bs{\eta}_{[\bs{0}, \bs{T}_D]})
\end{equation*} 
as desired. 
\medbreak
It remains to show that $\ratefusion{n}$ is a good rate function in the space $(\chamber_{\bs{T}_D}, \dist_{\chamber})$, i.e. for any $C>0$, the level set $\{ \bs{\eta}_{[\bs{0},\bs{T}_D]}: \ratefusion{n}(\bs{\eta}_{[\bs{0},\bs{T}_D]})\le C \}$ is compact. To this end, we compare the function $\ratefusion{n}$ and the good rate function $\rate_{\oplus n}$ defined in~\eqref{eqn::Varadhan_application_aux1}. Since $D$ has a positive distance from $\infty$, there exists $R>0$ such that $D\subset B(0,R)$. By~\cite[Lemma~4.5]{KemppainenSchrammLoewnerevolution}, we have 
\begin{equation*}
	|X_{\bs{t}}^j|\le 3R, \quad \text{for } \bs{t}\in [\bs{0},\bs{T}_D] \text{ and } 1\le j\le n.
\end{equation*}
Plugging into~(\ref{eqn::finite_energy_simple_multichord_aux1},\ref{eqn::finite_energy_simple_multichord_aux3}) and letting $t_2=\cdots=t_n=0$, we have
\begin{align*}
	\ratefusion{n}(\bs{\eta}_{[\bs{0},\bs{T}_D]}) \ge & \rate(\eta_{[0, T_D^1]}^1) + 2\sum_{j=2}^{n} \log(x_j-x_1) - 2n \log(6R),
\end{align*}
which implies
\begin{equation} \label{eqn::finite_watermelon_LDP_aux1}
	\rate_{\oplus n}(\bs{\eta}_{[\bs{0},\bs{T}_D]}) \le n \ratefusion{n}(\bs{\eta}_{[\bs{0},\bs{T}_D]}) + 2n^2 \log(6R) - 4 \sum_{i<j} \log(x_j-x_i),
\end{equation}
by symmetry. Assume $\left\{ \bs{\eta}_{[\bs{0},\bs{T}_D]}^{(k)} \right\}_{k\ge 1}\subset \chamber_{\bs{T}_D}$ is a sequence such that $\ratefusion{n}\left(\bs{\eta}_{[\bs{0},\bs{T}_D]}^{(k)}\right)\le C$ for $k\ge 1$. Since $\rate_{\oplus_n}$ is a good rate function in the space $\left( \prod_{j=1}^{n} \chamber_{T_D^j}(\HH;x_j,\infty), \dist_{\chamber} \right)$, combining with~\eqref{eqn::finite_watermelon_LDP_aux1}, we can pass to a subsequence, still denoted by $\left\{ \bs{\eta}_{[\bs{0},\bs{T}_D]}^{(k)} \right\}_{k\ge 1}$, which converges to some element $\bs{\eta}_{[\bs{0},\bs{T}_D]}^{(\infty)} \in  \overline{\chamber}_{\bs{T}_D}$. Recalling~\eqref{eqn::multi_time_energy_chordal_def}, we have
\begin{align} \label{eqn::finite_watermelon_LDP_aux2}
	\ratefusion{n} (\bs{\eta}_{[\bs{0}, \bs{T}_D]}) \; = \;  \rate_{\oplus n}(\bs{\eta}_{[\bs{0}, \bs{T}_D]}) + 12 \blm_{\bs{T}_D} - 3\sum_{j=1}^{n} \log g'_{\bs{T}_D,j}\left(W^j_{T_D^j}\right) + \LU_{\shuffle_n}(\bs{X}_{\bs{T}_D})  - \LU_{\shuffle_n}(\bs{x}).
\end{align}
On the one hand, the good rate function $\rate_{\oplus_n}$ is lower semicontinuous with respect to $\bs{\eta}_{[\bs{0},\bs{T}_D]}$. On the other hand, the other terms in the RHS of~\eqref{eqn::finite_watermelon_LDP_aux2}, except the first term, are continuous with respect to $\bs{\eta}$ due to Carathéodory convergence theorem (see e.g.~\cite[Theorem~1.8]{Pommerenke}) and Loewner–Kufarev theorem (see e.g.~\cite[Theorem~8.5]{Berestycki2011LecturesOS}). By~\eqref{eqn::finite_watermelon_LDP_aux2}, we have
\begin{equation} \label{eqn::finite_watermelon_LDP_aux3}
	\ratefusion{n} (\bs{\eta}_{[\bs{0},\bs{T}_D]}^{(\infty)}) \le \liminf_{k\to \infty} \ratefusion{n} (\bs{\eta}_{[\bs{0},\bs{T}_D]}^{(k)}) \le C. 
\end{equation}
Eq.~\eqref{eqn::finite_watermelon_LDP_aux3} implies that $\bs{\eta}_{[\bs{0},\bs{T}_D]}^{(\infty)}$ has finite $n$-time energy and Lemma~\ref{lem::finite_energy_simple_multichord} implies that $\bs{\eta}_{[\bs{0},\bs{T}_D]}^{(\infty)}\in \chamber_{\bs{T}_D}$. Thus the level set of $\ratefusion{n}$ is compact and $\ratefusion{n}$ is good.
\end{proof}

\subsection{Large deviation in infinite time: proof of Theorem~\ref{thm::halfwatermelon_LDP}}\label{subsec::LDP_infinite_time_halfwatermelon}

We will prove Theorem~\ref{thm::halfwatermelon_LDP} using Proposition~\ref{prop::finite_watermelon_LDP}, return estimate for chordal $\SLE$ in Proposition~\ref{prop::SLE_return_chordal} and a contraction principle in Lemma~\ref{lem::generalized_contraction}.

\begin{lemma}[{Generalized contraction principle,~\cite[Appendix~C]{AbuzaidPeltolaLargeDeviationCapacityParameterization}}] \label{lem::generalized_contraction}
	Let $X$ and $Y$ be Hausdorff topological spaces and $f:X\to Y$ a measurable map. Suppose that the family $\left\{ \PP^{(\kappa)} \right\}_{\kappa>0}$ of probability measures on $X$ satisfies large deviation principle with rate function $I:X\to [0,\infty]$. Consider the pushforward probability measures $\QQ^{(\kappa)}:=\PP^{(\kappa)} \circ f^{-1}$ of $\PP^{(\kappa)}$ by $f$. Define $J:Y\to [0,\infty]$ by 
	\[
	J(y):=\inf_{x\in f^{-1}(y)} I(x), \qquad y\in Y.
	\]
	\begin{enumerate}
		\item[(1)] Suppose that for each $M\in [0,\infty)$, there exists a closed set $E=E(M)\subset X$ only depending on $M$ such that $f$ restricted to $E$ is continuous and 
		\begin{equation*}
			\limsup_{\kappa\to 0} \kappa \log \PP^{(\kappa)} [X\setminus E] \le -M.
		\end{equation*}
		Then the measures $\QQ^{(\kappa)}$ satisfy the large deviation principle with rate function $J$.
		\item[(2)] If furthermore $I$ is a good rate function and $f$ is continuous on $I^{-1}[0,\infty)$, then $J$ is also a good rate function.
	\end{enumerate}
\end{lemma}

\paragraph*{Projective system.}
Let $(\Lambda,\le)$ be a partially ordered, right-filtering set (for any $i,j\in \Lambda$, there exists $k\in \Lambda$ such that $i\le k$ and $j\le k$).
\begin{itemize}
	\item A projective system $(X_j,p_{ij})_{i\le j\in \Lambda}$ consists of Hausdorff topological spaces $(X_j)_{j\in \Lambda}$ and continuous maps $p_{ij}:X_j\to X_i$ such that $p_{ik}=p_{ij}\circ p_{jk}$ whenever $i\le j\le k$.
	\item The projective limit of this system, denote by $\varprojlim X_j$, is the subset of the topological product space $\prod_{j\in \Lambda} X_j$, consisting all of the elements $(x_j)_{j\in \Lambda}$ for which $x_i=p_{ij}(x_j)$ whenever $i\le j$, equipped with the topology induced by $\prod_{j\in \Lambda} X_j$.
\end{itemize}

\begin{lemma}[{Dawson-G\"artner Theorem,~\cite[Theorem~4.6.1]{DemboZeitouniLargeDeviations}}] \label{lem::Dawson-Gartner}
Let $(X_j)_{j\in \Lambda}$ be a projective system and $\varprojlim X_j$ its projective limit with projections $p_j:\varprojlim X_j\to X_j$. Let $\left\{  \PP^{(\kappa)} \right\}_{\kappa>0}$ be a family of probability measures on $\varprojlim X_j$. Suppose that for every $j\in \Lambda$, the family $\left\{ \PP^{(\kappa)} \circ p_j^{-1} \right\}_{\kappa>0}$ of pushforward probability measures satisfies large deviation principle with good rate function $I_j:X_j\to [0,\infty]$. Then, the family $\left\{ \PP^{(\kappa)} \right\} _{\kappa>0}$ satisfies large deviation principle with good rate function $I:\varprojlim X_j\to [0,\infty]$, where
\begin{equation*}
	I(x):=\sup_{j\in \Lambda} I_j(p_j(x)), \qquad x\in \varprojlim X_j.
\end{equation*}
\end{lemma}

We introduce the projective system that we will use. Without loss of generality, we assume $(\Omega;\bs{x},y)=(\HH;\bs{x},\infty)$. 
We simply denote
\[
	\chamber_{\bs{T}(R)}:=\chamber_{\bs{T}_{\HH\cap B(0,R)}}.
\]
The collection $(\chamber_{\bs{T}(R)})_{R\in [0,\infty)}$ of $n$-time parameter simple curves, equipped with the metric $\dist_{\chamber}$ in~\eqref{eqn::metric}, forms a projective system with restriction maps
\begin{equation*}
	p_{R_1,R_2}:\chamber_{\bs{T}(R_2)} \to \chamber_{\bs{T}(R_1)}, \quad \bs{\eta}_{[\bs{0},\bs{T}_{\HH\cap B(0,R_2)}]} \mapsto \bs{\eta}_{[\bs{0},\bs{T}_{\HH\cap B(0,R_1)}]}, \quad \text{for } R_1\le R_2,
\end{equation*}
which are continuous and satisfy $p_{R_1,R_3}=p_{R_1,R_2} \circ p_{R_2,R_3}$ whenever $R_1\le R_2\le R_3$. Its projective limit,
\begin{align*}
	\overleftarrow{\chamber}:= & \varprojlim \chamber_{\bs{T}(R)}= \left\{ \left( \bs{\eta}_{[\bs{0},\bs{T}_{\HH\cap B(0,R)}]} \right)_{R\in [0,\infty)}: \bs{\eta}_{[\bs{0},\bs{T}_{\HH\cap B(0,R_1)}]}=p_{R_1,R_2}\left( \bs{\eta}_{[\bs{0},\bs{T}_{\HH\cap B(0,R_2)}]} \right) \text{ for all }R_1\le R_2 \right\} \\
	\subset & \prod_{R\in [0,\infty)} \chamber_{\bs{T}(R)},
\end{align*}
is equipped with the topology induced by the product topological space $\prod_{R\in [0,\infty)} \chamber_{\bs{T}(R)}$, the projections are given by
\begin{equation} \label{eqn::projection}
	P_R:\overleftarrow{\chamber} \to \chamber_{\bs{T}(R)}, \quad \left( \bs{\eta}_{[\bs{0},\bs{T}_{\HH\cap B(0,R')}]} \right)_{R'\in [0,\infty)} \mapsto \bs{\eta}_{[\bs{0},\bs{T}_{\HH\cap B(0,R)}]}.
\end{equation}

\begin{lemma} \label{lem::LDP_projection}
Fix $n\ge 2$ and $(n+1)$-polygon $(\HH; \bs{x}, \infty)$. 
The family $\{ \PPfusion{n}(\HH; \bs{x}, \infty)\}_{\kappa\in (0,4]}$ of laws of half-$n$-watermelon $\SLE_{\kappa}$ satisfies large deviation principle in $\overleftarrow{\chamber}$ with the topology induced by the product topological space $\prod_{R\in [0,\infty)} \chamber_{\bs{T}(R)}$ as $\kappa\to 0+$ with good rate function  $\ratefusion{n}(\HH; \bs{x}, \infty; \cdot)$ in~\eqref{eqn::multi_time_energy_chordal}. 
\end{lemma}
\begin{proof}
This is a combination of Proposition~\ref{prop::finite_watermelon_LDP} and Lemma~\ref{lem::Dawson-Gartner}.
\end{proof}

Recall we equip $\chamberfusion{n}(\HH; \bs{x}, \infty)$ with metric $\dist_{\chamber}$ defined in~\eqref{eqn::chamberfusion_metric}. The projections~\eqref{eqn::projection} induce a continuous bijection
\begin{equation*}
	\iota: \chamberfusion{n}(\HH; \bs{x}, \infty) \to \overleftarrow{\chamber}, \quad \bs{\eta} \mapsto \left( \bs{\eta}_{[\bs{0},\bs{T}_{\HH\cap B(0,R)}]} \right)_{R\in [0,\infty)}.
\end{equation*}
However, the map $\iota$ is not a homeomorphism from $\chamberfusion{n}(\HH; \bs{x}, \infty)$ to $\dist_{\chamber}$ (see~\cite[Remark~4.2]{AbuzaidPeltolaLargeDeviationCapacityParameterization} for the radial case, and the chordal case is similar). To apply generalized contraction principle, we consider a smaller space:
\begin{equation*}
	F_{\overrightarrow{N}}:=\bigcap_{p=1}^{\infty} \bigcap_{j=1}^{n} \left\{ \bs{\eta}\in \chamberfusion{n}(\HH; \bs{x}, \infty): \eta_{[T_{\HH\cap B(0,N_p)}^j,\infty)}^{j}  \cap B(0,p) =\emptyset \right\},
\end{equation*}
for any sequence of integers $\overrightarrow{N}=(N_p)_{p\in \N}$ with $N_p>p$.

\begin{lemma} \label{lem::homeomorphisim_FN}
For any sequence of integers $\overrightarrow{N}=(N_p)_{p\in \N}$ with $N_p>p$, the continuous bijection
\begin{equation*}
	\iota|_{F_{\overrightarrow{N}}}: \chamberfusion{n}(\HH; \bs{x}, \infty) \to \overleftarrow{\chamber}, \qquad \bs{\eta} \mapsto \left( \bs{\eta}_{[\bs{0},\bs{T}_{\HH\cap B(0,R)}]} \right)_{R\in [0,\infty)}
\end{equation*}
is a homeomorphism and the set $\iota(F_{\overrightarrow{N}})$ is closed.
\end{lemma}
\begin{proof}
First, we observe that $F_{\overrightarrow{N}}$ is a closed subset of $\chamberfusion{n}(\HH; \bs{x}, \infty)$. Indeed, for each $p\in \N$ and $1\le j\le n$, the set
\[
\left\{ \bs{\eta}\in \chamberfusion{n}(\HH; \bs{x}, \infty): \eta_{[T_{\HH\cap B(0,N_p)}^j,\infty)}^{j}  \cap B(0,p) \ne\emptyset \right\}
\] 
is open. It suffices to show that $\iota|_{F_{\overrightarrow{N}}}$ is a homeomorphism. Since $\iota|_{F_{\overrightarrow{N}}}$ is a continuous bijection, it suffices to show that $\left(\iota|_{F_{\overrightarrow{N}}}\right)^{-1}$ is continuous. Assuming
\[
	\left( \bs{\eta}_{[\bs{0},\bs{T}_{\HH\cap B(0,R)}]}^{(\ell)} \right)_{R\in [0,\infty)} \to \left( \bs{\eta}_{[\bs{0},\bs{T}_{\HH\cap B(0,R)}]} \right)_{R\in [0,\infty)} \text{  in  } \overleftarrow{\chamber} \text{  as  } \ell \to \infty,
\]
we need to show that $\bs{\eta}^{(\ell)}\to \bs{\eta} \text{ in } \chamberfusion{n}(\HH; \bs{x}, \infty) \text{ as } \ell\to\infty$. 
On the event $F_{\overrightarrow{N}}$, we have for all $p\in \N$,
\begin{equation}\label{eqn::homeomorphism_aux1}
	\dist_{\chamber}(\bs{\eta}^{(\ell)},\bs{\eta}) \le \dist_{\chamber}(\bs{\eta}_{[\bs{0},\bs{T}_{\HH\cap B(0,N_p)}]}^{(\ell)},\bs{\eta}_{[\bs{0},\bs{T}_{\HH\cap B(0,N_p)}]})+ 2 \sup_{x,y\in \HH\setminus B(0,p)} |\varphi_{\U}(x)-\varphi_{\U}(y)|,
\end{equation}
where $\varphi_{\U}: \HH\to \U$ is a conformal map. Letting $\ell\to \infty$ and $p\to \infty$, we have $\bs{\eta}^{(\ell)}\to \bs{\eta}$ as desired.
\end{proof}

\begin{proof}[{Proof of Theorem~\ref{thm::halfwatermelon_LDP}}]
Without loss of generality, we assume $(\Omega;\bs{x},y)=(\HH;\bs{x},\infty)$ and eliminate them from the notation. First, we derive the large deviation principle. Fix $M\in [0,\infty)$. By \eqref{eqn::SLE_return_estimate} in Proposition~\ref{prop::SLE_return_chordal} (see also \cite[Remark~3.10]{AbuzaidPeltolaLargeDeviationCapacityParameterization}), there exists a sequence of integer $\overrightarrow{N}=(N_p)_{p\in \N}$ with $N_p>p$ (depending on $p$ and $M$) such that for $\kappa\in (0,1)$ and $1\le j\le n$,
\begin{equation*} \label{eqn::LDP_aux1}
	\PPfusion{n} \left[ \eta_{[T_{\HH\cap B(0,N_p)}^j,\infty)}^j \cap \partial B(0,p) \neq \emptyset \right] \le \exp(-(M+p \log 2+ \log n)/\kappa)\le \frac{\exp(-M/\kappa)}{2^p n}.
\end{equation*}
Thus
\begin{equation} \label{eqn::LDP_aux2}
	\PPfusion{n} \left[ \overleftarrow{\chamber} \setminus \iota\left(F_{\overrightarrow{N}}\right) \right] \le \sum_{p=1}^{\infty} \sum_{j=1}^{n} \PPfusion{n} \left[ \eta_{[T_{\HH\cap B(0,N_p)}^j,\infty)}^j \cap \partial B(0,p) \neq \emptyset \right] \le \exp(-M/\kappa).
\end{equation}
From Lemma~\ref{lem::homeomorphisim_FN}, the map $\iota^{-1}$ is a continuous on the closed set $\iota\left(F_{\overrightarrow{N}}\right)$. Combining with~\eqref{eqn::LDP_aux2}, Lemma~\ref{lem::LDP_projection} and applying Lemma~\ref{lem::generalized_contraction}, we obtain the large deviation principle:
\begin{align*}\label{eqn::LDP_aux3}
\begin{split}
	&\liminf_{\kappa\to 0} \kappa \log \PPfusion{n}[O] \ge -\inf_{\bs{\eta}\in O} \ratefusion{n}(\bs{\eta}), \quad \text{for open } O\subset \chamberfusion{n}(\HH;\bs{x},\infty), \\
	&\limsup_{\kappa\to 0} \kappa \log \PPfusion{n}[F] \le -\inf_{\bs{\eta}\in F} \ratefusion{n}(\bs{\eta}), \quad \text{for closed } F\subset \chamberfusion{n}(\HH;\bs{x},\infty).
\end{split}
\end{align*}
\medbreak
Next, we show that $\ratefusion{n}$ is a good rate function on $\chamberfusion{n}(\HH;\bs{x},\infty)$. To apply Lemma~\ref{lem::generalized_contraction}, it suffices to show that $\iota^{-1}$ is continuous on $\left(\ratefusion{n}\right)^{-1}[0,\infty)\subset \overleftarrow{\chamber}$. For $p\ge 1$ and $1\le j\le n$, we have
\begin{align}
	-\inf \left\{ \ratefusion{n}(\bs{\eta}): \eta_{[T_{\HH\cap B(0,N_p)}^j,\infty)}^j \cap \partial B(0,p) \neq \emptyset \right\} \le & \limsup_{\kappa\to 0} \kappa \log \PPfusion{n} \left[ \eta_{[T_{\HH\cap B(0,N_p)}^j,\infty)}^j \cap \partial B(0,p) \neq \emptyset \right]\nonumber \\
	\le & -M \tag{due to~\eqref{eqn::LDP_aux1}},
\end{align}
which implies
\begin{equation*} \label{eqn::return_energy}
	\inf \left\{ \ratefusion{n}(\bs{\eta}): \eta_{[T_{\HH\cap B(0,N_p)}^j,\infty)}^j \cap \partial B(0,p) \neq \emptyset \right\} \ge M.
\end{equation*}
Combining with Lemma~\ref{lem::homeomorphisim_FN}, the map $\iota^{-1}$ is continuous on $(\ratefusion{n})^{-1}[0,M) \subset \iota(F_{\overrightarrow{N}})$. As $M\ge 0$ is arbitrary, we see that $\iota^{-1}$ is continuous on $(\ratefusion{n})^{-1}[0,\infty)$. Since $\ratefusion{n}$ is a good rate function on $\overleftarrow{\chamber}$ by Lemma~\ref{lem::LDP_projection}, applying Lemma~\ref{lem::generalized_contraction} (2), we conclude that $\ratefusion{n}$ is a good rate function on $\chamberfusion{n}(\HH;\bs{x},\infty)$.
\end{proof}

\subsection{Large deviation for chordal SLE with force points: proof of Proposition~\ref{prop::LDP_SLEkapparho}}\label{subsec::proof_rho_LDP}

In this section, we prove Proposition~\ref{prop::LDP_SLEkapparho}. The proof is similar to that of Theorem~\ref{thm::halfwatermelon_LDP} in Sections~\ref{sec::LDP_finite_time_halfwatermelon}-\ref{subsec::LDP_infinite_time_halfwatermelon}: we first prove the large deviation principle for finite-time curves (see Lemma~\ref{lem::rhoLR_finite_LDP_chordal}), and then extend it to infinite-time curves using the return estimate in Proposition~\ref{prop::SLE_return_chordal}.

\medbreak 

The following lemma gives an equivalent definition of the energy ${\rate}^{(\bs{\rho})}(\Omega;\bs{x}^L, a, \bs{x}^R, b;\eta_{[0,T]})$ defined in~\eqref{eqn::rho_energy_chordal}. It will be used in the proof of Proposition~\ref{prop::LDP_SLEkapparho}.

\begin{lemma}
	Assume the same notation as in Definition~\ref{def::rho_energy_chordal}.  For $\eta\in \overline{\chamber}(\Omega; a,b)$, we have 
	\begin{align}\label{eqn::rho_energy_chordal2}
		&{\rate}^{(\bs{\rho})}(\Omega;\bs{x}^L, a, \bs{x}^R, b;\eta_{[0,T]}) \nonumber\\
		=& \frac{1}{2}\int_0^T \left(\dot{W}_t\right)^2 \ud t - \left(\sum_{j=1}^\ell \frac{\rho_j^L(\rho_j^L + 4)}{4}\log |g'_T(x_j^L)| +  \sum_{j=1}^r \frac{\rho_j^R(\rho_j^R + 4)}{4} \log |g'_T(x_j^R)|\right. \nonumber \\
			&+ \sum_{j=1}^\ell \rho_j^L \log \left|\frac{W_T - g_T(x_j^L)}{x_j^L}\right| + \sum_{j=1}^r \rho_j^R \log \left|\frac{g_T(x_j^R)-W_T}{x_j^R}\right| + \sum_{\substack{1\leq i \leq \ell \\ 1\leq j \leq r}} \frac{\rho_i^L \rho_j^R}{2} \log \left|\frac{g_T(x_j^R) - g_T(x_i^L)}{x_j^R - x_i^L} \right| \nonumber \\
			&+ \left. \sum_{1\leq i <j \leq \ell} \frac{\rho_i^L \rho_j^L}{2} \log \left|\frac{g_T(x_j^L) - g_T(x_i^L)}{x_j^L - x_i^L} \right|+ \sum_{1\leq i <j \leq r} \frac{\rho_i^R \rho_j^R}{2} \log \left|\frac{g_T(x_j^R) - g_T(x_i^R)}{x_j^R - x_i^R} \right| \right).
	\end{align}
\end{lemma}
\begin{proof}	Expanding~\eqref{eqn::rho_energy_chordal}, we obtain
		\begin{align}
			&\rate^{(\bs{\rho})}(\HH;\bs{x}^L, 0, \bs{x}^R, \infty;\eta_{[0,T]}) \nonumber \\
			=& \frac{1}{2}\int_0^T \left(\dot{W}_t\right)^2 \ud t  +  \int_0^T \sum_{j=1}^\ell \left(\frac{(\rho_j^L)^2}{2} \frac{1}{(W_t - g_t(x_j^L))^2} -  \frac{\rho_j^L \dot{W_t}}{W_t - g_t(x_j^L)}\right)\ud t \nonumber \\
			&+ \int_0^T \sum_{j=1}^r \left(\frac{(\rho_j^R)^2}{2} \frac{1}{(W_t - g_t(x_j^R))^2} -  \frac{\rho_j^R \dot{W_t}}{W_t-g_t(x_j^R)}\right) \ud t+\int_0^T \sum_{\substack{1\leq i \leq \ell \\ 1\leq j \leq r}} \frac{\rho_i^L \rho_j^R}{(W_t - g_t(x_i^L))(W_t-g_t(x_j^R))} \ud t \nonumber\\
			&+\int_0^T\sum_{1\leq i<j \leq \ell} \frac{\rho_i^L \rho_j^L}{(W_t - g_t(x_i^L))(W_t - g_t(x_j^L))} \ud t + \int_0^T\sum_{1\leq i<j \leq r} \frac{\rho_i^R \rho_j^R}{(W_t - g_t(x_i^R))(W_t - g_t(x_j^R))} \ud t. \label{eqn::rho_energy_expanded_aux_chordal}
		\end{align}
		Standard calculations of Loewner equation gives that, for any $x,\tilde{x}\in \{\bs{x}^L, \bs{x}^R\}$,
		\begin{align*}
			&\partial_t \log |W_t - g_t(x)| = \frac{\dot{W}_t}{W_t - g_t(x)} + \frac{2}{(W_t - g_t(x))^2},\quad \partial_t \log g_t'(x) = -\frac{2}{(W_t - g_t(x))^2},\\
			& \partial_t \log |g_t(x) - g_t(\tilde{x})| = -\frac{2}{(W_t - g_t(\tilde{x}))(W_t - g_t(x))}.
		\end{align*}
		From these, we have
		\begin{align*}
			\frac{2}{(W_t - g_t(x))^2} &= -\partial_t \log g_t'(x), \\
			\frac{\dot{W}_t}{W_t - g_t(x)} &= \partial_t \log |W_t - g_t(x)| + \partial_t \log g_t'(x),\\
			\frac{2}{(W_t - g_t(x))(W_t - g_t(\tilde{x}))} &= -\partial_t \log |g_t(x) - g_t(\tilde{x})|.
		\end{align*}
		Plugging these into \eqref{eqn::rho_energy_expanded_aux_chordal}, we obtain the desired equality.
	\end{proof}

	\begin{lemma}\label{lem::rhoLR_finite_LDP_chordal} 
		Fix $\ell, r\geq 0$, and $(\ell+r+2)$-polygon $(\Omega; \bs{x}^L, a, \bs{x}^R, b)$. Fix $\bs{\rho} = (\bs{\rho}^L; \bs{\rho}^R) \in \mathbb{R}_{\geq 0}^{\ell+r}$. Assume $D\subset \Omega$ is a simply connected subdomain which agrees with $\Omega$ in a neighborhood of the point $a$ and has a positive distance from the point $b$. Then the family of laws of $\eta_{[0,T_D]}$ under $\PPkapparho(\Omega; \bs{x}^L, a, \bs{x}^R, b)$ satisfies large deviation principle in the space $(\chamber_{T_D}(\Omega; \bs{x}^L, a, \bs{x}^R, b), \dist_{\chamber})$ as $\kappa\to 0+$ with good rate function $\rate^{(\bs{\rho})}(\Omega; \bs{x}^L, a, \bs{x}^R, b;\cdot)$ defined in \eqref{eqn::rho_energy_chordal2}.
	\end{lemma}
	
	\begin{proof}
		It suffices to show the conclusion for $(\Omega; \bs{x}^L, a, \bs{x}^R, b) = (\HH; \bs{x}^L, 0, \bs{x}^R, \infty)$ and we eliminate them from the notation.
		Recall that the law of chordal $\SLE_{\kappa}(\bs{\rho}^L; \bs{\rho}^R)$ in $\HH$ from $0$ to $\infty$ with force points $(\bs{x}^L; \bs{x}^R)$ is the same as the law of chordal $\SLE_{\kappa}$ in $(\HH; 0, \infty)$ tilted by $M_t$ defined in~\eqref{eqn::SLEkapparho_martingale}, up to the first time $x^L_1$ or $x^R_1$ is swallowed. For $\eta_{[0,t]}\in \overline{\chamber}_t(\HH;0,\infty)$, we define 
		\begin{align*}
			\Phi(\eta_{[0,t]};\kappa):=&\kappa \log\frac{M_t}{M_0} \\
			=&
			\sum_{j=1}^\ell \rho_j^L \log \left|\frac{g_t(x_j^L)-W_t}{x_j^L}\right|+\sum_{j=1}^r \rho_j^R \log \left|\frac{g_t(x_j^R)-W_t}{x_j^R}\right| + \sum_{\substack{1\leq i \leq \ell \\ 1\leq j \leq r}} \frac{\rho_i^L \rho_j^R}{2} \log \left|\frac{g_t(x_j^R) - g_t(x_i^L)}{x_j^R-x_i^L}\right|\\
			&+ \sum_{1\leq i<j\leq \ell} \frac{\rho_i^L \rho_j^L}{2} \log \left|\frac{g_t(x_j^L) - g_t(x_i^L)}{x_j^L-x_i^L}\right| + \sum_{1\leq i<j\leq r} \frac{\rho_i^R \rho_j^R}{2} \log \left|\frac{g_t(x_j^R) - g_t(x_i^R)}{x_j^R-x_i^R}\right|\\
			&+ \sum_{i=1}^\ell \frac{\rho_i^L(\rho_i^L+4-\kappa)}{4} \log g_t'(x_i^L) + \sum_{j=1}^r \frac{\rho_j^R(\rho_j^R+4-\kappa)}{4} \log g_t'(x_j^R). 
		\end{align*}
		By Lemma~\ref{lem::SLEkapparho_martingale}, for any subset $A\subset \chamber_{T_D}(\HH;0,\infty)$, we have
		\[
		\PPkapparho[A] = \E^{(\kappa)}\left[ \exp\left( \frac{1}{\kappa} \Phi(\eta_{[0,T_D]};\kappa) \right) \one\{\eta_{[0,T_D]}\in A\} \right],
		\]
		where $\E^{(\kappa)}$ is the expectation with respect to the law of chordal $\SLE_{\kappa}$ in $(\HH;0,\infty)$. 
		Applying the generalization of Varadhan's lemma (see Lemma~\ref{lem::Varadhan} and Remark~\ref{rmk::Varadhan}), we conclude that the family of laws of $\eta_{[0,T_D]}$ under $\PPkapparho$ satisfies large deviation principle in the space $(\chamber_{T_D}(\HH; \bs{x}^L, 0, \bs{x}^R, \infty), \dist_{\chamber})$ with good rate function
		\begin{equation}
			\rate(\eta_{[0,T_D]}) - \Phi(\eta_{[0,T_D]};0) \stackrel{\eqref{eqn::rho_energy_chordal2}}{=}\rate^{(\bs{\rho})}(\eta_{[0,T_D]})
		\end{equation}
		with the same argument as in the proof of Proposition~\ref{prop::finite_watermelon_LDP}.		
	\end{proof}
	
	\begin{proof}[Proof of Proposition~\ref{prop::LDP_SLEkapparho}]
		Without loss of generality, we assume $(\Omega; \bs{x}^L, a, \bs{x}^R, b) = (\HH; \bs{x}^L, 0, \bs{x}^R, \infty)$ and we eliminate them from the notation. By Lemma~\ref{lem::rhoLR_finite_LDP_chordal} and Lemma~\ref{lem::Dawson-Gartner}, we obtain large deviation principle of $\PPkapparho$ on the projective limit $\overleftarrow{\chamber}$ with the topology induced by the product topological space $\prod_{R\in [0,\infty)} \chamber_{{T}(R)}$ as $\kappa\to 0+$ with good rate function $\rate^{(\bs{\rho})}(\Omega;\bs{x}^L, a, \bs{x}^R, b;\eta)$ in~\eqref{eqn::rho_energy_chordal_full}. Combining this with Lemma~\ref{prop::SLE_return_chordal}, we complete the proof with the same argument as in the proof of Theorem~\ref{thm::halfwatermelon_LDP}.
	\end{proof}

\subsection{Common-time parameter and proof of Proposition~\ref{prop::LDP_DysonBM_chordal}}
\label{subsec::common_time_chordal}
\paragraph*{Common-time parameter (chordal).}
Fix $n\ge 1$ and $\bs{x}=(x_1, \ldots, x_n)\in\LX_n^{\HH}$. Consider an $n$-tuple $\bs{\eta}_{[\bs{0},\bs{t}]}=(\eta^1_{[0,t_1]}, \ldots, \eta^n_{[0,t_n]})\in \chamber_{\bs{t}}(\HH;\bs{x},\infty)$ parameterized by $\bs{t}=(t_1, \ldots, t_n)\in[0,\infty)^n$.
Recall that $g_{\bs{t}}$ is the conformal map from $\HH\setminus \cup_{j=1}^n \eta_{[0,t_j]}^j$ onto $\HH$ normalized at $\infty$, i.e.,
\[
g_{\bs{t}}(z) = z +\frac{\aleph_{\bs{t}}}{z}+o\left(|z|^{-1}\right), \quad \text{ as } z\to \infty.
\]
A standard calculation gives
\[
\ud \aleph_{\bs{t}} = 2\sum_{j=1}^n g_{\bs{t},j}'(W_{t_j})^2 \ud t_j.
\]
We say that the $n$-tuple $\bs{\eta}_{[\bs{0},\bs{t}]}$ is parameterized by common-time parameter if
\begin{equation*}
	\partial_{t_j} \aleph_{(t,\ldots,t)} = 2, \quad \text{ for } 1\le j\le n.
\end{equation*}
It is proved in~\cite[Lemma~3.1]{HuangWuYangMultipleSLEsDysonBM} that common-time parameter exists for continuous disjoint simple curves $\bs{\eta}_{[\bs{0},\bs{t}]}$. The relation between common-time parameter and multi-time parameter is given by 
\begin{equation}\label{eqn::common_time_relation_chordal}
	\ud t_j = g_{\bs{t},j}'(W_{t_j}^j)^{-2} \ud t, \quad \text{ for }1\le j\le n,
\end{equation}
Moreover, 
	\begin{equation*}
		t\leq t_j(t) \leq n t, \quad \text{ for } 1\le j\le n.
	\end{equation*}
Under common-time parameter, we denote the $n$-tuple by $\bs{\eta}_{[0,t]} = \bs{\eta}_{[(0,\ldots,0),(t,\ldots,t)]}$. We define the following
normalized conformal transformations:
\begin{itemize}
	\item $g_t^j$ is the conformal map from $\mathbb{H} \backslash \eta_{[0, t]}^j$ onto $\mathbb{H}$ with $\lim _{z \rightarrow \infty}|g_t^j(z)-z|=0,1 \leq j \leq n$.
	\item $g_{t, j}$ is the conformal map from $\mathbb{H} \backslash g_{t}^j\left(\cup_{i \neq j} \eta_{[0, t]}^i\right)$ onto $\mathbb{H}$ with $\lim _{z \rightarrow \infty}\left|g_{t, j}(z)-z\right|=0,1 \leq j \leq n$.
	\item $g_t$ is the conformal map from $\mathbb{H} \backslash \cup_{j=1}^n \eta_{[0, t]}^j$ onto $\mathbb{H}$ with $\lim _{z \rightarrow \infty}\left|g_t(z)-z\right|=0$. Note that $g_t$ satisfies the following expansion at $\infty$:
\[
g_t(z)=z+\frac{2 n t}{z}+o\left(|z|^{-1}\right), \quad \text { as } z \rightarrow \infty. 
\]
\end{itemize}
Denote by $W^j$ the driving function of $\eta^j$ and denote the driving function of the $n$-tuple $\boldsymbol{\eta}_{[0,t]}$, started from $X_0=\left(x_1, \ldots, x_n\right)\in\LX_n^{\HH}$, by

\[
\boldsymbol{X}_t=\left(X_t^1, \ldots, X_t^n\right), \quad \text { with } X_t^j=g_{t, j}\left(W_t^j\right), \quad \text { for } 1 \leq j \leq n .
\]
\medbreak

\begin{lemma}[{\cite[Lemma~3.4]{HuangWuYangMultipleSLEsDysonBM}}] \label{lem::DysonBM_driving_function}
	Under common-time parameter, the driving function $\boldsymbol{X}_t$ of half-$n$-watermelon $\SLE_{\kappa}$ satisfies the SDE of Dyson Brownian motion with parameter $\beta = 8/\kappa$ in~\eqref{eqn::DysonBM}.
\end{lemma}
\begin{proof}[Proof of Proposition~\ref{prop::LDP_DysonBM_chordal}]
	From  \eqref{eqn::halfwatermelonSLE0_minimizer_aux2} and \eqref{eqn::multi_time_energy_chordal_def2}, we have 
	\begin{align*}
		\ud \ratefusion{n}(\HH;\bs{x},\infty;\bs{\eta}_{\bs{[0,t]}}) =& \frac{1}{2}\sum_{j=1}^{n} \bigg( \dot{W}_{t_j}^{j} - 2 \sum_{i\neq j} \frac{ g'_{\bs{t},j}(W_{t_j}^{j})}{X_{\bs{t}}^{j} - X_{\bs{t}}^{i}} 
	- 3 \, \frac{g''_{\bs{t},j}(W_{t_j}^{j})}{g'_{\bs{t},j}(W_{t_j}^{j})}\bigg)^2 \ud t_j \\
	=& \frac{1}{2}\sum_{j=1}^n \left(g_{\bs{t},j}'(W_{t_j}^j)^{-2} \partial_{t_j}X^j_{\bs{t}}  - \sum_{i\ne j}\frac{2}{X_{\bs{t}}^j - X_{\bs{t}}^i}\right)^2 g_{\bs{t},j}'(W_{t_j}^j)^2 \ud t_j.
	\end{align*}
	Using the relation between common-time parameter and multi-time parameter~\eqref{eqn::common_time_relation_chordal},
	we obtain from Theorem~\ref{thm::halfwatermelon_LDP} that under common-time parameter, for any finite time $T$, the family of laws of $\bs{\eta}_{[{0},{T}]}$ under $\PPfusion{n}(\HH;\bs{x},\infty)$  satisfies large deviation principle in the space $(\chamber_{T}(\HH;\bs{x},\infty), \dist_\chamber)$ as $\kappa\to 0+$ with good rate function
	\begin{align*}
		\ratefusion{n}(\HH;\bs{x},\infty;\bs{\eta}_{[0,T]}) =& \frac{1}{2}\int_0^T \sum_{j=1}^n \left( \dot{X}_t^j - \sum_{i\ne j} \frac{4}{X_t^j - X_t^i} \right)^2 \ud t.
	\end{align*}
	Here $\chamber_{T}$ is short for $\chamber_{(T,\ldots, T)}$ under common-time parameter. Finally, by Loewner-Kufarev theorem (see e.g.~\cite[Theorem~8.5]{Berestycki2011LecturesOS}), the map $\bs{\eta}_{[0,T]}\mapsto\boldsymbol{X}_{[0,T]}$ from $(\chamber_{T}(\HH;\bs{x},\infty), \dist_\chamber)$ to $\mathrm{C}([0,T],\LX_n^{\HH})$ is continuous. Thus, by the contraction principle and Lemma~\ref{lem::DysonBM_driving_function}, the family of laws of Dyson Brownian motion $\{\bs{X}^{\kappa}\}_{\kappa\in (0,4]}$ satisfies large deviation principle in the space $\mathrm{C}([0,T], \LX_n^{\HH})$ as $\kappa\to 0+$ with good rate function \eqref{eqn::rate_DysonBM} as desired.
\end{proof}

\subsection{Boundary perturbation: proof of Proposition~\ref{prop::bp_chordal}}
\label{subsec::bp_chordal_proof}
In this section, we prove Proposition~\ref{prop::bp_chordal}. We first recall the boundary perturbation of half-$n$-watermelon $\SLE_{\kappa}$ in Lemma~\ref{lem::halfwatermelon_bp}, and then combine it with Theorem~\ref{thm::halfwatermelon_LDP} to prove Proposition~\ref{prop::bp_chordal}.
The boundary perturbation of chordal $\SLE_\kappa$ is originally established in~\cite{LawlerSchrammWernerConformalRestriction,KozdronLawlerMultipleSLEs}, and later extended to multiple $\SLE_{\kappa}$ in~\cite{PeltolaWuGlobalMultipleSLEs, HuangPeltolaWuMultiradialSLEResamplingBP}. 
\begin{lemma}[{\cite[Proposition~3.4]{HuangPeltolaWuMultiradialSLEResamplingBP}}]
\label{lem::halfwatermelon_bp}
Fix $\kappa \in(0,4]$, $n \geq 1$, and nice $(n+1)$-polygon $(\Omega; \bs{x}, y)= (\Omega; x_1, \ldots, x_n, y)$. Let $U\subset\Omega$ be a simply connected subdomain which coincides with $\Omega$ in a neighborhood of $\{x_1, \ldots, x_n, y\}$. Then, the half-$n$-watermelon $\SLE_\kappa$ probability measure in the smaller polygon  $(U ; \boldsymbol{x} , y)$  is absolutely continuous with respect to that in  $(\Omega ; \boldsymbol{x}, y)$, with Radon-Nikodym derivative
\[
\frac{\ud \PPfusion{n}(U ; \boldsymbol{x}, y)}{\ud \PPfusion{n}(\Omega ; \boldsymbol{x}, y)}(\boldsymbol{\eta}) = \frac{\LZfusion{n}(\Omega ; \boldsymbol{x}, y)}{\LZfusion{n}(U ; \boldsymbol{x}, y)} \mathbb{1}\{\boldsymbol{\eta} \cap (\Omega\setminus U)=\emptyset\} \exp \left(\frac{\mathfrak{c}}{2} \blm(\Omega ; \boldsymbol{\eta}, \Omega \setminus U)\right),
\]
where 
$\LZfusion{n}$ is partition function in~\eqref{eqn::halfwatermelon_pf_def}, and the constant
$\mathfrak{c}$ is the central charge defined in~\eqref{eqn::parameters_b_c}, and $\blm(\Omega ; \boldsymbol{\eta}, \Omega\setminus U)$ is the Brownian loop measure in $\Omega$ of those loops that intersect both $\boldsymbol{\eta}$ and $\Omega\setminus U$ (see~\eqref{eqn::blm_def}).	
\end{lemma}
\begin{proof}[Proof of Proposition~\ref{prop::bp_chordal}]
Comparing~\eqref{eqn::LUfusion_def} and~\eqref{eqn::halfwatermelon_pf_def}, we have 
\begin{equation}\label{eqn::bp_chordal_aux0}
\kappa\log\LZfusion{n}(\Omega; \bs{x}, y)=-\LUfusion{n}(\Omega; \bs{x}, y)-\frac{\kappa}{2}\sum_{\ell=1}^n\log\Poisson(\Omega; x_{\ell}, y). 
\end{equation}
Note that $\chamberfusion{n}(U;\bs{x},y)$ is an open subset of $\chamberfusion{n}(\Omega;\bs{x},y)$. For $\bs{\eta}\in\chamberfusion{n}(\Omega;\bs{x},y)$, we define 
\begin{align*}
\Phi(\bs{\eta};\kappa) :=& \kappa\log\frac{\ud \PPfusion{n}(U ; \boldsymbol{x}, y)}{\ud \PPfusion{n}(\Omega ; \boldsymbol{x}, y)}(\bs{\eta}) \notag \\ 
=&\underbrace{-\LUfusion{n}(\Omega ; \boldsymbol{x}, y)+\LUfusion{n}(U ; \boldsymbol{x}, y)}_{\Phi_0(\boldsymbol{\eta}):=} \underbrace{-m(\Omega ; \boldsymbol{\eta}, \Omega \backslash U)}_{\Phi_1(\boldsymbol{\eta}):=} \underbrace{\frac{(6-\kappa)(8-3 \kappa)}{4}}_{f_1(\kappa):=} \notag \\ 
&\underbrace{-\frac{\kappa}{2}}_{f_2(\kappa):=} \underbrace{\sum_{1 \leq \ell \leq n}\left(\log \mathrm{P}\left(\Omega ; x_{\ell}, y\right)-\log \mathrm{P}\left(U ; x_{\ell}, y\right)\right)}_{\Phi_2(\boldsymbol{\eta}):=}, 
\end{align*}
where the second equation is due to~\eqref{eqn::bp_chordal_aux0}. 
	For any subset $A \subset \chamberfusion{n}(U ; \boldsymbol{x}, y)$, 
\begin{equation}\label{eqn::bp_chordal_proof_aux1}
\PPfusion{n}(U ; \boldsymbol{x}, y)[A]=\Efusion{n}(\Omega ; \boldsymbol{x}, y)\left[\exp \left(\frac{1}{\kappa} \Phi(\boldsymbol{\eta} ; \kappa)\right) \mathbb{1}{\{\boldsymbol{\eta} \in A\}}\right].
\end{equation}
To apply Lemma~\ref{lem::Varadhan} and Remark~\ref{rmk::Varadhan}, note that $\Phi_0(\boldsymbol{\eta}), \Phi_2(\boldsymbol{\eta})$ are constants and $\Phi_1(\boldsymbol{\eta})$ is continuous. Thus conditions (a) and (d') of Lemma~\ref{lem::Varadhan} holds; $\left(f_1(\kappa), \Phi_1(\boldsymbol{\eta})\right)$ satisfies condition (c) of Lemma~\ref{lem::Varadhan}; $\left(f_2(\kappa), \Phi_2(\boldsymbol{\eta})\right)$ satisfies condition (b) of Lemma~\ref{lem::Varadhan}. Combining with Theorem~\ref{thm::halfwatermelon_LDP}, the RHS of \eqref{eqn::bp_chordal_proof_aux1} satisfies large deviation with rate function 
\[
\ratefusion{n}(\Omega ; \boldsymbol{x}, y ; \boldsymbol{\eta})-\lim _{\kappa \rightarrow 0+} \Phi(\boldsymbol{\eta} ; \kappa)=\ratefusion{n}(\Omega ; \boldsymbol{x}, y ; \boldsymbol{\eta})+\LUfusion{n}(\Omega ; \boldsymbol{x}, y)-\LUfusion{n}(U ; \boldsymbol{x}, y)+12 m(\Omega ; \boldsymbol{\eta}, \Omega \backslash U) .
\]
Thus we obtain
\[
\ratefusion{n}(U ; \boldsymbol{x}, y ; \boldsymbol{\eta})=\ratefusion{n}(\Omega ; \boldsymbol{x}, y ; \boldsymbol{\eta})+\LUfusion{n}(\Omega ; \boldsymbol{x}, y)-\LUfusion{n}(U ; \boldsymbol{x}, y)+12 m(\Omega ; \boldsymbol{\eta}, \Omega \backslash U),
\]
which gives \eqref{eqn::bp_choral} as desired. 
\end{proof}

\section{Return estimate for radial SLE with spiral}
\label{sec::return_radial}
The goal of this section is to derive a return estimate for radial SLE with force points and spiral. 
Fix $\kappa>0$ and $\mu\in\R$ and $n\ge 1$.
Suppose
\begin{equation} \label{eqn::forcepoints_weights_radial}
	\bs{\rho}=(\rho_2,\ldots,\rho_n)\in \R_{\ge 0}^{n-1}, \qquad \bs{\theta}=(\theta_1,\ldots,\theta_n)\in \LX_n^{\U}.	
\end{equation}
The radial $\SLE_{\kappa}^{\mu}(\bs{\rho})$ in $(\U;\ee^{\ii\bs{\theta}};0)$ is defined as the radial Loewner chain $(K_t)_{t\ge 0}$ driven by a continuous function $\xi:[0,\infty)\to \R$ satisfying the SDE system
\begin{equation}\label{eqn::radialSLEkappa_rho_sde}
\begin{cases}
\displaystyle 	\ud \xi_t = \sqrt{\kappa} \ud B_t + \sum_{j=2}^{n} \frac{\rho_j}{2} \cot \left( \frac{\xi_t-V_t^j}{2} \right) \ud t + \mu \ud t, \qquad \xi_0=\theta_1, \\
\displaystyle 	\ud V_t^j = \cot \left( \frac{V_t^j-\xi_t}{2} \right) \ud t, \qquad V_0^{j}=\theta_j, \qquad 2\le j\le n, 
\end{cases}
\end{equation}
where $B_t$ is a standard one-dimensional Brownian motion. 
We denote the law of $(K_t)_{t\ge 0}$ by $\PPradial{1}^{(\kappa;\bs{\rho};\mu)}$. 
In this article, we only consider the case when $\kappa\in (0,4]$ and all weights are non-negative. In such case, the process is well-defined for all time and is almost surely generated by a continuous curve $\gamma$, see e.g.~\cite{MillerSheffieldIG4, HuangPeltolaWuMultiradialSLEResamplingBP}. Moreover, in this case, the curve does not touch the boundary except at the beginning, and the evolution $V_t^{j}$ coincides with $\covmap_t(\theta_j)$ for $2\le j\le n$, where $\covmap_t$ is the covering map of the radial Loewner chain (see Section~\ref{subsec::pre_radialLoewner}).
The main result of this section is the following return estimate.

\begin{proposition}\label{prop::radialSLE_return}
Assume the same notation as in~\eqref{eqn::forcepoints_weights_radial}. 
Suppose $\gamma$ is radial $\SLE_{\kappa}^{\mu}(\bs{\rho})$ in $(\U;\ee^{\ii\bs{\theta}};0)$. For $s>0$, we denote
\begin{equation}\label{eqn::radial_return_notations}
	\U_s=\ee^{-s} \U,\qquad \tau_s=\inf \{ t\ge 0: |\gamma_t|=\ee^{-s} \},\qquad \LF_s=\sigma(\{ \gamma_t: 0\le t\le \tau_s \}).
\end{equation}
Then, for any $s>0$ and $M\in [0,\infty)$, there is a constant $S=S(s,M)>s$ such that
\begin{equation}\label{eqn::radialSLE_return}
	\PPradial{1}^{(\kappa;\bs{\rho};\mu)} \left[ \gamma_{[\tau_S,\infty)} \cap \partial \U_s \neq \emptyset \right] \le \exp(-M/\kappa), \qquad \text{for } \kappa\in(0,1/5).
\end{equation}
\end{proposition}

Proposition~\ref{prop::radialSLE_return} will be a key ingredient in the proof of large deviation in Theorem~\ref{thm::radialSLE_LDP}. Return estimates for radial SLE (without force point) were originally proved in~\cite[Theorem~1.3]{LawlerContinuityofradialSLE} and~\cite[Theorem~1.2]{field2015escape},  and later adapted to the LDP framework in~\cite[Proposition~4.6]{AbuzaidPeltolaLargeDeviationCapacityParameterization}. Return estimates for radial SLE with a single force point were proved in~\cite[Proposition~18]{krusell2024rholoewnerenergylargedeviations}. We prove the general case in Proposition~\ref{prop::radialSLE_return} in this section.
To this end, we first recall from~\cite{HuangPeltolaWuMultiradialSLEResamplingBP} a monotone coupling in the radial setting in Lemma~\ref{lem::radial_mono_coupling} (in Section~\ref{subsec::mono_radial}) to reduce the general case of radial SLE with force points to the case of one force point. This monotone coupling is analogous to the one in the chordal setting in Lemma~\ref{lem::monotone_coupling}. 
Then, in Section~\ref{subsec::proof_return_estimate}, we split the return estimates into the estimates on ``good event" and the estimates on ``bad event". 
To control the probability of ``bad event", we use the monotone coupling and the estimates for radial Bessel process in Lemma~\ref{lem::radialBessel_estimates}, see Section~\ref{subsec::proof_return_estimate}. 
To control the probability of ``good event", we carefully analyze the Radon-Nikodym derivative between radial $\SLE_{\kappa}^{\mu}(\rho)$ process and standard radial $\SLE_{\kappa}$ process in Section~\ref{subsec::goodevent}. 
The main difference between the chordal setting and the radial setting is that the control on ``good event" is more complicated in the radial setting, while the control on the ``bad event" is more complicated in the chordal setting. 

\medbreak
As a consequence of Proposition~\ref{prop::radialSLE_return}, we obtain large deviation principle for radial SLE with force points. Fix $n\ge 1$ and $n$-polygon $(\Omega;\bs{x})$ with $z\in \Omega$. 
Fix $\kappa>0$ and $\mu\in\R$ and $\bs{\rho}\in \R_{\ge 0}^{n-1}$. The $\SLE_{\kappa}^{\mu}(\bs{\rho})$ process in $(\Omega;\bs{x};z)$ is defined via conformal invariance: it is the preimage of the radial $\SLE_{\kappa}^{\mu}(\bs{\rho})$ in $(\U;\ee^{\ii\bs{\theta}};0)$ under any conformal map $\varphi:\Omega\to \U$ satisfying $\varphi(\bs{x})=\ee^{\ii\bs{\theta}}$ and $\varphi(z)=0$. 
When $\kappa\in (0,4]$ and $\bs{\rho}=(\rho_2, \ldots, \rho_n)\in\R_{\ge 0}^{n-1}$, the curve does not touch the boundary except at the beginning, thus it is a continuous curve in $\chamber(\Omega; x_1; z)$.
The law of the process is denoted by
$\PPradial{1}^{(\kappa;\bs{\rho};\mu)} (\Omega;\bs{x};z)$. 
We also write $\PPradial{1}^{(\kappa;\bs{\rho};\mu)}=\PPradial{1}^{(\kappa;\bs{\rho};\mu)}(\U;\ee^{\ii\bs{\theta}};0)$ for short.
For standard radial SLE, we just call it radial $\SLE_{\kappa}$ in $(\Omega; x; z)$.

\begin{proposition}\label{prop::radialSLE_LDP}
Fix $n\ge 1$ and $n$-polygon $(\Omega;\bs{x})$ with $z\in \Omega$. Fix $\mu\in\R$ and $\bs{\rho}\in \R_{\ge 0}^{n-1}$. The family $\{ \PPradial{1}^{(\kappa;\bs{\rho};\mu)} (\Omega;\bs{x};z) \}_{\kappa\in (0,4]}$ of laws of radial $\SLE_{\kappa}^{\mu}(\bs{\rho})$ satisfies large deviation principle in the space $(\chamber(\Omega;x_1;z),\dist_\chamber)$ as $\kappa\to 0+$ with good rate function $\rateradial{1}^{(\bs{\rho};\mu)}(\Omega; \bs{x}; z, \cdot)$ which will be defined in Definition~\ref{def::rho_energy_radial}.
\end{proposition}

The rest of this section is organized as follows.  
In Section~\ref{subsec::pre_radialLoewner}, we review the definition and basic properties of radial SLE with force points. 
We introduce the radial Loewner energy $\rateradial{1}^{(\bs{\rho};\mu)}$ in Definition~\ref{def::rho_energy_radial}. 
We also recall the return estimate and the large deviation for standard radial SLE in Section~\ref{subsec::pre_radialLoewner}.
We prove Proposition~\ref{prop::radialSLE_return} in Sections~\ref{subsec::mono_radial}-\ref{subsec::goodevent}.
The large deviation in Proposition~\ref{prop::radialSLE_LDP} will be proved in Section~\ref{subsec::radial_rho_LDP}.

\subsection{Radial Loewner chain and its energy}
\label{subsec::pre_radialLoewner}
\paragraph*{Radial Loewner chain.}For a simply connected domain $\Omega\subsetneq\C$ and $z\in\Omega$, the conformal radius of $\Omega$ seen from $z$ is defined by $\CR(\Omega; z):=1/\phi'(z)$ where $\phi:\Omega\to \U$ is the conformal map such that $\phi(z)=0$ and $\phi'(z)>0$. In particular, $\CR(\U; 0)=1$. 

A $\U$-hull is a relatively closed subset $K$ of $\U$ such that $\U\setminus K$ is simply connected and containing $0$. 
Let $\mathfrak{g}_K$ be
the conformal map from $\U\setminus K$ onto $\U$ satisfying the normalization $\mathfrak{g}_K(0)=0$ and $\mathfrak{g}'_K(0)=\CR(\U\setminus K; 0)^{-1}>0$. We refer to $\mathfrak{g}_K$ as the mapping-out function of $K$. 

Fix $\theta\in \R$ and $T \in (0,\infty]$. Let $\xi: [0,T) \to \R$ be a continuous function with $\xi_0 = \theta$.
A radial Loewner chain driven by $\xi$ is a family of $\U$-hulls $(K_t)_{t\in [0,T)}$ such that the mapping-out function $\mathfrak{g}_t:=\mathfrak{g}_{K_t}$ satisfies the Loewner equation
\begin{align}\label{eqn::single_radial_Loewner_equation}
	\partial_t \mathfrak{g}_t(z)=\mathfrak{g}_t(z) \frac{\ee^{\ii \xi_t}+\mathfrak{g}_t(z)}{\ee^{\ii \xi_t}-\mathfrak{g}_t(z)}, \qquad \mathfrak{g}_0(z)=z, \qquad z\in \overline{\U}.
\end{align}
It is convenient to use the covering map
$\covmap_t \colon \R \to \R$ 
of $\mathfrak{g}_t$ defined via $\mathfrak{g}_t(\ee^{\ii \theta}) = \exp( \ii \covmap_t(\theta) )$. The covering map satisfies the Loewner equation
\begin{align*}\label{eqn::single_radial_Loewner_equation_cov}
	\partial_t \covmap_t(\theta) = \cot \bigg(\frac{\covmap_t(\theta) - \xi_t}{2}\bigg) , \qquad \covmap_0(\theta)=\theta.
\end{align*}
For each $z\in \overline{\U}$, the flow $t\mapsto \mathfrak{g}_t(z)$ is well-defined up to the swallowing time
\[
	\sigma_z := \sup\left\{ t \in [0, T) : \inf_{s \in [0, t]} |\mathfrak{g}_s(z) - \exp(\ii\xi_s)| > 0 \right\}.
\]
The hulls are determined by  $K_t = \{z\in \HH : \sigma_z \le t\}$,
and $(K_t)_{t\in [0,T)}$ is parameterized by radial capacity, i.e., $\CR(\U\setminus K_t;0)=\ee^{-t}$ for all $t\in [0,T)$. Moreover, $(K_t)_{t\in [0,T)}$ satisfies the local growth property: for $0\le s<t<T$, the diameter of $\mathfrak{g}_s(K_t\setminus K_s)$ tends to $0$ as $t\downarrow s$ uniformly over $s\in [0,T)$. 

Conversely, any increasing family of $\U$-hulls $(K_t)_{t\in [0,T)}$ satisfying the local growth property can be represented as a radial Loewner chain driven by some continuous function $\xi$. A simple curve $\gamma$ in $(\U; \ee^{\ii\theta}; 0)$ satisfies the local growth property, and the driving function can be obtained easily by $\xi_t = \arg\mathfrak{g}_t(\gamma_t)$. 

\begin{lemma}\label{lem::taum_Koebe}
Fix $\theta\in \R$ and suppose $\gamma\in \chamber(\U; \ee^{\ii\theta}; 0)$. We define the hitting time $\tau_s$ as in~\eqref{eqn::radial_return_notations}.
Then
\begin{equation} \label{eqn::taum_Koebe}
	s-\log 4\le \tau_{s}\le s.
\end{equation}
\end{lemma}
\begin{proof}
On the one hand, applying Koebe's quarter theorem with $\mathfrak{g}^{-1}_{\tau_s}$, we have $\frac{1}{4}\ee^{-\tau_s}\U\subset \U\setminus \gamma_{[0,\tau_s]}$, which implies the lower bound of~\eqref{eqn::taum_Koebe}. On the other hand, since $\U_s\subset \U \setminus \gamma_{[0,\tau_s]}$, we have $\mathfrak{g}'_{\tau_s}(0)\le \ee^{s}$, which implies the upper bound of~\eqref{eqn::taum_Koebe}.	
\end{proof}

\paragraph*{Radial SLE with force points.}
We use the same notation as in~\eqref{eqn::forcepoints_weights_radial}. Recall that the radial $\SLE_{\kappa}^{\mu}(\bs{\rho})$ in $(\U;\ee^{\ii\bs{\theta}};0)$ is defined as the radial Loewner chain $(K_t)_{t\ge 0}$ driven by a continuous function $\xi:[0,\infty)\to \R$ satisfying the SDE system~\eqref{eqn::radialSLEkappa_rho_sde}. Its law is absolutely continuous with respect to radial SLE and the Radon-Nikodym derivative is given by the following lemma.

\begin{lemma}[{\cite[Lemma~2.3]{HuangPeltolaWuMultiradialSLEResamplingBP}}]\label{lem::SLEkapparho_martingale_radial}
Fix $\kappa\in(0,4]$ and assume the same notation as in~\eqref{eqn::forcepoints_weights_radial}. The law of radial $\SLE_{\kappa}^{\mu}(\bs{\rho})$ in $(\U;\ee^{\ii\bs{\theta}};0)$ is the same as the law of radial $\SLE_{\kappa}$ in $(\U;\ee^{\ii\theta_1};0)$ tilted by the following martingale, up to the first time $\ee^{\ii\theta_2}$ or $\ee^{\ii\theta_n}$ is swallowed: 
\begin{align} \label{eqn::SLEkapparho_martingale_radial}
\begin{split}
	M_t=& \mathfrak{g}'_t(0)^{\frac{\bar{\rho}(\bar{\rho}+4)-4\mu^2}{8\kappa}} \times \prod_{j=2}^{n} \covmap'_t(\theta_j)^{\frac{\rho_j(\rho_j+4-\kappa)}{4\kappa}} \times \prod_{j=2}^{n}  \sin \left( \frac{\covmap_{t}(\theta_j)-\xi_t}{2} \right)^{\frac{\rho_j}{\kappa}} \\
	& \times \prod_{2\le j<\ell\le n} \sin\left( \frac{\covmap_t(\theta_\ell)-\covmap_t(\theta_j)}{2} \right)^{\frac{\rho_\ell \rho_j}{2\kappa}} \times \exp\left( \frac{\mu}{\kappa} \xi_t + \frac{\mu}{\kappa} \sum_{j=2}^{n}\frac{\rho_j}{2} \covmap_t(\theta_j) \right),
\end{split}
\end{align}
where $\bar{\rho}:=\rho_2+\cdots+\rho_n$.
\end{lemma}

\begin{definition}[Radial Loewner energy] \label{def::rho_energy_radial}
Fix $n\ge 1$ and $n$-polygon $(\Omega;\bs{x})$ with $z\in \Omega$. Fix $\mu\in\R$ and $\bs{\rho}\in \R_{\ge 0}^{n-1}$. Let $\varphi:\Omega\to\U$ be a conformal map satisfying $\varphi(\bs{x})=\ee^{\ii\bs{\theta}}$ and $\varphi(z)=0$. For $\gamma\in \overline{\chamber}(\Omega;x_1;z)$, we parameterize $\gamma$ by radial capacity of $\varphi(\gamma)$. We denote by $\xi_t$ the driving function of $\varphi(\gamma)$. We define the \emph{truncated radial $(\bs{\rho};\mu)$-Loewner energy} associated with the force points $(x_2,\ldots,x_n)$ to be
\begin{equation} \label{eqn::rho_energy_radial_noAC}
	\rateradial{1}^{(\bs{\rho};\mu)}(\Omega;\bs{x};z;\gamma_{[0,T]}) \; := \; \frac{1}{2} \int_{0}^{T} \left(\dot{\xi}_t-\mu- \sum_{j=2}^{n} \frac{\rho_j}{2} \cot \left( \frac{\xi_t-\covmap_t(\theta_j)}{2} \right) \right)^2 \ud t, 
\end{equation}
if $\xi$ is absolutely continuous; and setting $\rateradial{1}^{(\bs{\rho};\mu)}(\Omega;\bs{x};z;\cdot)=+\infty$ otherwise. The quantity $\rateradial{1}^{(\bs{\rho};\mu)} (\Omega; \bs{x}; z; \gamma_{[0,T]})$ is increasing in $T$. We define the \emph{radial $(\bs{\rho};\mu)$-Loewner energy} of $\gamma\in \overline{\chamber}(\Omega;x_1;z)$ to be its limit:
\begin{equation} \label{eqn::rho_energy_radial}
	\rateradial{1}^{(\bs{\rho};\mu)}(\Omega;\bs{x};z;\gamma) \; := \; \lim_{T\to +\infty} \rateradial{1}^{(\bs{\rho};\mu)}(\Omega;\bs{x};z;\gamma_{[0,T]}).
\end{equation}
\end{definition}

When all weights are zero and $\mu=0$, the radial $(\bs{\rho};\mu)$-Loewner energy reduces to the standard radial Loewner energy defined in~\cite{AbuzaidPeltolaLargeDeviationCapacityParameterization}:
\begin{equation}\label{eqn::energy_radial_single}
	\rateradial{1}(\Omega;x;z;\gamma_{[0,T]}):=\frac{1}{2} \int_{0}^{T} \left(\dot{\xi}_t \right)^2 \ud t,\qquad \rateradial{1}(\Omega;x;z;\gamma):=\frac{1}{2} \int_{0}^{+\infty} \left(\dot{\xi}_t \right)^2 \ud t,
\end{equation}
if $\xi$ is absolutely continuous; and setting $\rateradial{1}(\Omega;x;z;\cdot)=+\infty$ otherwise.  We recall the return estimate for standard radial SLE and its large deviation below. They will be the foundation for our proof of Proposition~\ref{prop::radialSLE_return} and Theorem~\ref{thm::radialSLE_LDP}.

\begin{lemma}[{\cite[Proposition~4.8]{AbuzaidPeltolaLargeDeviationCapacityParameterization}}] \label{lem::intersection_estimate}
Suppose $\gamma\sim\PPradial{1}^{(\kappa)}$ is radial $\SLE_{\kappa}$ in $(\U; \ee^{\ii\theta_1}; 0)$. 
Fix $t_0\in (0,\infty)$. There exists a universal constant $C_{\eqref{eqn::intersection_estimate}}\in (0,\infty)$ such that for any fixed simple curve $\gamma_{\mathrm{fix}}$ from $\ee^{\ii\theta_1}$ to $0$ in $\U$ and any crosscut $\eta$ of $\U$ (a simple curve in $\U$ whose prime ends lying on $\partial\U$) disjoint from $\gamma_{\mathrm{fix}}$, we have
\begin{equation} \label{eqn::intersection_estimate}
	\PPradial{1}^{(\kappa)} [\gamma_{[0,t_0]}\cap \eta\neq \emptyset] \le \exp\left( C_{\eqref{eqn::intersection_estimate}}/\kappa \right) \LE_{\U}(\eta,\gamma_{\mathrm{fix}})^{8/\kappa-1}, \qquad \text{ for } \kappa\in (0,4],
\end{equation}
where $\LE_{\U}(\eta,\gamma_{\mathrm{fix}})$ is the Brownian excursion measure defined in~\eqref{eqn::Def_Brownian excursion measure}. 
\end{lemma}

\begin{lemma}[{\cite[Proposition~3.2]{AbuzaidPeltolaLargeDeviationCapacityParameterization}}]\label{lem::finite_radialSLE_LDP}
Fix $1$-polygon $(\Omega;x;z)$ and $T>0$. Suppose $\gamma\sim \PPradial{1}^{(\kappa)}(\Omega; x; z)$ is radial $\SLE_{\kappa}$ in $(\Omega;x;z)$. Then the family of laws of $\gamma_{[0,T]}$ under $\PPradial{1}^{(\kappa)}(\Omega; x; z)$ satisfies large deviation principle in the space $(\chamber_T(\Omega;x;z), \dist_{\chamber})$ as $\kappa\to 0+$ with good rate function $\rateradial{1}(\Omega;x;z;\cdot)$ given by~\eqref{eqn::energy_radial_single}. 
\end{lemma}

\subsection{A monotone coupling}
\label{subsec::mono_radial}

\begin{lemma}[{\cite[Lemma~A.6]{HuangPeltolaWuMultiradialSLEResamplingBP}}]\label{lem::radial_mono_coupling}
Fix $\kappa\in (0,4]$ and assume the same notation as in~\eqref{eqn::forcepoints_weights_radial}. 
Suppose $(\xi_t, V_t^2, \ldots, V_t^n)$ is the solution to the SDE system~\eqref{eqn::radialSLEkappa_rho_sde}.  
We define $\bs{\Delta}_t:=(\Delta_t^{2},\ldots,\Delta_t^{n})$ where
\[
\Delta_{t}^{j}:=V_t^{j}-\xi_t,  
\qquad  2\le j\le n.
\]
Suppose $(X_t)_{t\geq 0}$ is the solution to the SDE~\eqref{eqn::radialBessel_SDE} with $X_0=\Delta_0^{n}$ and
	\begin{equation*}
		\alpha=1+(\rho_2+\cdots+\rho_n)/2=1+\bar{\rho}/2. 
	\end{equation*}
Then there exists a coupling between $(X_t)_{t\geq 0}$ and $(\bs{\Delta}_t)_{t\geq 0}$ such that $X_0=\Delta_0^{n}$ and $X_t\le \Delta_t^{n}$ for all $t>0$ almost surely. 
\end{lemma}

The following lemma shows that $V_t^p-V_t^2$ decays exponentially fast. 
Using the exponential decay, we are able to reduce the multiple force points case to the single force point case. 

\begin{lemma}[{\cite[Lemma~A.8]{HuangPeltolaWuMultiradialSLEResamplingBP}}]\label{lem::radial_expdecay}
Assume the same notation as in Lemma~\ref{lem::radial_mono_coupling}. 
The difference $\Delta_{t}^{n}-\Delta_{t}^{2}$ decays exponentially fast to $0$ as $t\to\infty$: 
\begin{equation}\label{eqn::radial_expdecay}
	0\le \Delta_t^{n}-\Delta_t^{2}\le \left(\Delta_0^{n}-\Delta_0^{2}\right)\exp(-C_{\eqref{eqn::radial_expdecay}}t), 
\end{equation}
where \[C_{\eqref{eqn::radial_expdecay}}=\inf\left\{\frac{\sin(u/2)}{u}: 0\le u\le \Delta_0^{n}-\Delta_0^{2}\right\}>0.\]
\end{lemma}

\subsection{Proof of Proposition~\ref{prop::radialSLE_return}}
\label{subsec::proof_return_estimate}

\begin{proof}[Proof of Proposition~\ref{prop::radialSLE_return}]
Recall that the driving function for radial $\SLE_{\kappa}^{\mu}(\bs{\rho})$ satisfies the SDE system~\eqref{eqn::radialSLEkappa_rho_sde}. For $2\leq j \leq n$, we denote 
\[
	\bs{\rho}:=(\rho_2,\ldots,\rho_n),\qquad \bar{\rho}=\sum_{j=2}^{n} \rho_j, \qquad
	\Delta_{t}^{j}:=V_t^{j}-\xi_t.
\]
We use notations in~\eqref{eqn::radial_return_notations} and define ``bad events" to be
\begin{equation*}
	E_j:=\left\{ \min_{0 \le t \le \tau_{j}} \Delta_t^{n}\le j^{-1/2} \right\} \bigcup \left\{ \max_{0 \le t \le \tau_{j}} \Delta_t^{2} \ge 2\pi-j^{-1/2} \right\}, \qquad \text{ for } j\in \N. 
\end{equation*}
Then we have 
\begin{align}\label{eqn::radialSLE_return_aux1}
	&\PPradial{1}^{(\kappa;\bs{\rho};\mu)} \left[ \gamma_{[\tau_{S},\infty)} \cap \partial \U_s \neq \emptyset \right] \notag\\
	\le &\sum_{m=0}^{\infty} \PPradial{1}^{(\kappa;\bs{\rho};\mu)} [E_{S+m+1}]+\sum_{m=0}^{\infty} \PPradial{1}^{(\kappa;\bs{\rho};\mu)} \left[ \{\gamma_{[\tau_{S+m},\tau_{S+m+1}]} \cap \partial \U_s \neq \emptyset \}\cap E_{S+m+1}^c\right]. 
\end{align}
We will prove in Lemma~\ref{lem::radial_badevents} that when $\kappa<(2\bar{\rho}+4)/5$, we have
\begin{align}\label{eqn::radial_badevents}
	\sum_{m=0}^{\infty} \PPradial{1}^{(\kappa;\bs{\rho};\mu)} [E_{S+m+1}] \le  2\exp\left(C_{\eqref{eqn::radialBessel_estimates}} /\kappa\right) \frac{S^{5/2-(\bar{\rho}+2)/\kappa}}{(\bar{\rho}+2)/\kappa-5/2}.
\end{align}
We will prove in Lemma~\ref{lem::radial_goodevents} that there exist constants $S_{\eqref{eqn::radial_goodevents}}\in \N$ and $C_{\eqref{eqn::radial_goodevents}}\in (0,\infty)$ depending on $s,\mu,\bs{\rho}$ such that when $S>S_{\eqref{eqn::radial_goodevents}}$ and $\kappa<1$, we have
\begin{align}\label{eqn::radial_goodevents}
\sum_{m=0}^{\infty} \PPradial{1}^{(\kappa;\bs{\rho};\mu)} \left[ \{\gamma_{[\tau_{S+m},\tau_{S+m+1}]} \cap \partial \U_s \neq \emptyset \}\cap E_{S+m+1}^c\right]
\le \frac{\exp\left( \frac{1}{\kappa} (C_{\eqref{eqn::radial_goodevents}}-S) \right)}{1-\ee^{-1}}.
\end{align}

When $\kappa\le 1/5$, we have $(\bar{\rho}+2)/\kappa-5/2> 1/\kappa > 1$. Plugging~\eqref{eqn::radial_badevents} and~\eqref{eqn::radial_goodevents} into~\eqref{eqn::radialSLE_return_aux1}, we obtain
\begin{align*}
	\PPradial{1}^{(\kappa;\bs{\rho};\mu)} \left[ \gamma_{[\tau_{S},\infty)} \cap \partial \U_s \neq \emptyset \right]
	\le& 2\exp\left(C_{\eqref{eqn::radialBessel_estimates}} /\kappa\right) \frac{S^{5/2-(\bar{\rho}+2)/\kappa}}{(\bar{\rho}+2)/\kappa-5/2} + \frac{\exp\left( \frac{1}{\kappa} (C_{\eqref{eqn::radial_goodevents}}-S)  \right)}{1-\ee^{-1}}\\
	\le & \exp \left( \frac{1}{\kappa} \left( C_{\eqref{eqn::radialBessel_estimates}}+C_{\eqref{eqn::radial_goodevents}}+\kappa \log 4 -\log S \right)\right). 
\end{align*}
Taking 
\[ 
	S:=S_{\eqref{eqn::radial_goodevents}} \vee \exp \left( C_{\eqref{eqn::radialBessel_estimates}}+C_{\eqref{eqn::radial_goodevents}}+ \log 4 +M \right),
\]
we obtain~\eqref{eqn::radialSLE_return} as desired.
\end{proof}

\begin{lemma}\label{lem::radial_badevents}
Assume the same notation as in the proof of Proposition~\ref{prop::radialSLE_return}. The estimate~\eqref{eqn::radial_badevents} on the bad event holds for $\kappa<(2\bar{\rho}+4)/5$. 
\end{lemma}

\begin{proof}
We denote by $\PP$ the measure of the coupling in Lemma~\ref{lem::radial_mono_coupling} with $\alpha=1+\bar{\rho}/2$. From the coupling, we have  $X_0=\Delta_0^{n}$ and $X_t\le \Delta_t^{n}$ for all $t$. Consequently, when $\kappa<4+2\bar{\rho}$, we have
\begin{align*}
	\PPradial{1}^{(\kappa;\bs{\rho};\mu)} \left[ \min_{0 \le t \le \tau_{m}} \Delta_t^{n}\le m^{-1/2} \right] \le & \PP \left[ \min_{0 \le t \le \tau_{m}} X_t\le m^{-1/2} \right].
\end{align*}
Combining with~\eqref{eqn::radialBessel_estimates} and~\eqref{eqn::taum_Koebe}, we have
\begin{align}
	\PPradial{1}^{(\kappa;\bs{\rho};\mu)} \left[ \min_{0 \le t \le \tau_{m}} \Delta_t^{n}\le m^{-1/2} \right] \le & \exp\left(C_{\eqref{eqn::radialBessel_estimates}} /\kappa\right) m^{3/2-(\bar{\rho}+2)/\kappa}, \qquad\text{ for } m\in \N.  \label{eqn::radialSLE_return_aux2}
\end{align}
By symmetry, we have 
\begin{equation} \label{eqn::radialSLE_return_aux3}
	\PPradial{1}^{(\kappa;\bs{\rho};\mu)} \left[ \max_{0 \le t \le \tau_{m}} \Delta_t^{2}\ge 2\pi-m^{-1/2}  \right] \le \exp\left(C_{\eqref{eqn::radialBessel_estimates}} /\kappa\right) m^{3/2-(\bar{\rho}+2)/\kappa}, \qquad\text{ for } m\in \N.
\end{equation}
When $\kappa<(2\bar{\rho}+4)/5$, combining~\eqref{eqn::radialSLE_return_aux2} and~\eqref{eqn::radialSLE_return_aux3}, we have
\begin{align*}
	\sum_{m=0}^{\infty} \PPradial{1}^{(\kappa;\bs{\rho};\mu)} [E_{S+m+1}]\le & 2\exp\left(C_{\eqref{eqn::radialBessel_estimates}} /\kappa\right) \sum_{m=0}^{\infty}  (S+m+1)^{3/2-(\bar{\rho}+2)/\kappa} \notag\\
	\le & 2\exp\left(C_{\eqref{eqn::radialBessel_estimates}} /\kappa\right) \int_S^\infty x^{3/2-(\bar{\rho}+2)/\kappa} \, \ud x \\
	=& 2\exp\left(C_{\eqref{eqn::radialBessel_estimates}} /\kappa\right) \frac{S^{5/2-(\bar{\rho}+2)/\kappa}}{(\bar{\rho}+2)/\kappa-5/2},
\end{align*}
which gives~\eqref{eqn::radial_badevents} as desired.
\end{proof}

\subsection{Estimate on the good event}
\label{subsec::goodevent}
To prove the estimate~\eqref{eqn::radial_goodevents} on the good event in the proof of Proposition~\ref{prop::radialSLE_return}, two essential inputs are the return estimate in Lemma~\ref{lem::intersection_estimate} and Lemma~\ref{lem::bound_finitetime_RN} where we control the Radon-Nikodym derivative between radial $\SLE_{\kappa}^{\mu}(\bs{\rho})$ and radial $\SLE_{\kappa}$.

\begin{lemma} \label{lem::bound_finitetime_RN}
Assume the same notation as in the proof of Proposition~\ref{prop::radialSLE_return}. 
Suppose $M_t$ is the martingale defined in~\eqref{eqn::SLEkapparho_martingale_radial}. 
There exist constants $C_{\eqref{eqn::bound_finitetime_RN}},C'_{\eqref{eqn::bound_finitetime_RN}}\in (0,\infty)$ depending on $\mu,\bs{\rho}$ and there exists a constant $m_{\eqref{eqn::bound_finitetime_RN}}$ depending on $ \Delta_0^{2},\ldots, \Delta_0^{n}, \mu,\bs{\rho}$ such that for all $\kappa \in (0,4]$ and $m>m_{\eqref{eqn::bound_finitetime_RN}}$, we have 
\begin{equation} \label{eqn::bound_finitetime_RN}
	\frac{M_{\tau_{m+1}} }{M_{\tau_{m}}}\le \exp \left( \frac{1}{\kappa} \left( C'_{\eqref{eqn::bound_finitetime_RN}}+C_{\eqref{eqn::bound_finitetime_RN}} \sqrt{m} \right)\right), \text{ a.s. on } E_{m+1}^c.
\end{equation}
\end{lemma}

\begin{proof}
From~\eqref{eqn::SLEkapparho_martingale_radial}, we have
\begin{align} \label{eqn::bound_finitetime_RN_aux1}
\begin{split}
	\frac{M_{\tau_{m+1}}}{M_{\tau_{m}}}=&\underbrace{\left( \frac{\mathfrak{g}'_{\tau_{m+1}}(0)}{\mathfrak{g}'_{\tau_{m}}(0)} \right)^{\frac{\bar{\rho}(4+\bar{\rho})-4\mu^2}{8\kappa}}}_{R_1:=}  \underbrace{\prod_{j=2}^{n} \left( \frac{h'_{\tau_{m+1}} (\theta_j)}{h'_{\tau_{m}} (\theta_j)} \right)^{\frac{\rho_j (\rho_j+4-\kappa)}{4\kappa}} }_{R_2:=}  \underbrace{\prod_{2\le i<j\le n} \left( \frac{\sin \left( \left( \Delta_{\tau_{m+1}}^{j}-\Delta_{\tau_{m+1}}^{i} \right)/2 \right)}{\sin \left( \left( \Delta_{\tau_{m}}^{j}-\Delta_{\tau_{m}}^{i} \right)/2 \right)} \right)^{\frac{\rho_i \rho_j}{2\kappa}} }_{R_3:=}\\
	& \times \underbrace{\prod_{j=2}^{n} \left(\frac{\sin \left( \Delta_{\tau_{m+1}}^{j}/2 \right)}{\sin \left( \Delta_{\tau_{m}}^{j}/2 \right)}\right)^{\frac{\rho_j}{\kappa}} }_{R_4:=} \times \underbrace{\exp \left(\frac{\mu}{\kappa} \left(\xi_{\tau_{m+1}}-\xi_{\tau_{m}}+\sum_{j=2}^{n} \frac{\rho_j}{2} \left( \covmap_{\tau_{m+1}}(\theta_j)-\covmap_{\tau_{m}}(\theta_j) \right)\right) \right) }_{R_5:=}.
\end{split}
\end{align}
Let us evaluate the five terms in RHS of~\eqref{eqn::bound_finitetime_RN_aux1} one by one.
\begin{itemize}
	\item For $R_1$, since $\mathfrak{g}'_t(0)=\ee^t$, combining with~\eqref{eqn::radialSLE_return_aux2}, we have
	\begin{equation} \label{eqn::bound_R1}
		R_1\le \exp\left( (1+\log 4) \left| \frac{\bar{\rho}(4+\bar{\rho})-4\mu^2}{8\kappa} \right| \right).
	\end{equation}
	\item For $R_2$, since $\covmap'_{t}(\theta_j)$ is decreasing in $t$, we have $R_2\le 1$. 
	\item For $R_3$, let \begin{equation}\label{eqn::bound_R3}
	m_{\eqref{eqn::bound_R3}}:=\log 4+\frac{1}{c_{\eqref{eqn::radial_expdecay}}}\log\left( \left(  \Delta_0^{n}-\Delta_0^{2} \right)/\pi \right). 
	\end{equation}
	When $m>m_{\eqref{eqn::bound_R3}}$, Lemma~\ref{lem::radial_expdecay} and~\eqref{eqn::radialSLE_return_aux2} gives
	\[ \Delta_{\tau_{m+1}}^{j}-\Delta_{\tau_{m+1}}^{i} \le \Delta_{\tau_{m}}^{j}-\Delta_{\tau_{m}}^{i} \le \pi, \qquad \text{ for } i<j, \]
	which implies $R_3\le 1$. 
\item For $R_5$, since
	\[ \covmap_t(\theta_2)-2\pi \le \xi_t \le \covmap_t(\theta_2), \qquad \text{ and } \partial_t \covmap_t(\theta_j)=\cot(\Delta_t^{j}/2), \] we have 
	\begin{align} \label{eqn::bound_R5}
		R_5 \le &\exp(2\pi |\mu|/\kappa) \exp \left( \frac{(1+\rho_2/2)|\mu|}{\kappa} \left( \covmap_{\tau_{m+1}}(\theta_2)-\covmap_{\tau_{m}}(\theta_2) \right) +\sum_{j=3}^{n} \frac{\rho_j|\mu|/2}{\kappa} \left( \covmap_{\tau_{m+1}}(\theta_j)-\covmap_{\tau_{m}}(\theta_j) \right) \right), \notag\\
		\le & \exp(2\pi |\mu|/\kappa) \exp\left( \frac{(1+\bar{\rho}/2)|\mu|}{\kappa} (1+\log 4) \sup_{\stackrel{t\in [\tau_{m},\tau_{m+1}]}{2\le j\le n}} \frac{1}{\sin(\Delta_t^{j}/2)} \right).
	\end{align}
	\item For $R_4$ and $R_5$, let $m_{\eqref{eqn::bound_R4R5}}$ be a constant depending on $\Delta_0^{n}-\Delta_0^{2}$ such that
	\[ (m+1)^{-1/2}- (2m)^{-1/2}\ge \left(\Delta_0^{n}- \Delta_0^{2}\right) \exp(-c_{\eqref{eqn::radial_expdecay}}(m-\log 4))\qquad \text{ for all } m\ge m_{\eqref{eqn::bound_R4R5}}. \]
	On $E_{m+1}^c$, we have
	\[ (m+1)^{-1/2}\le \Delta_t^{n} \text{ and } \Delta_t^{2} \le 2\pi-(m+1)^{-1/2} \text{ for } t\in [0,\tau_{m+1}]. \] 
	Combining with Lemma~\ref{lem::radial_expdecay}, for $t\in [\tau_m,\tau_{m+1}]$ and $2\le j\le n$, we have
	\begin{align} \label{eqn::bound_R4R5_aux1}
		\begin{split}
			\Delta_t^{j}\ge &(m+1)^{-1/2}-\left(\Delta_0^{n}-\Delta_0^{2}\right)\exp(-c_{\eqref{eqn::radial_expdecay}}t), \\
			\Delta_t^{j}	\le &  2\pi-(m+1)^{-1/2}+\left(\Delta_0^{n}-\Delta_0^{2}\right)\exp(-c_{\eqref{eqn::radial_expdecay}}t).
		\end{split}
	\end{align}
	When $m>m_{\eqref{eqn::bound_R4R5}}$, Plugging~\eqref{eqn::radial_expdecay} into~\eqref{eqn::bound_R4R5_aux1}, we have 
	\begin{equation} \label{eqn::bound_R4R5_aux2}
		(2m)^{-1/2} \le \Delta_t^{j} \le  2\pi-(2m)^{-1/2}, \qquad \text{ for } 2\le j\le n \text{ and } t\in [\tau_m,\tau_{m+1}].
	\end{equation}
	Combining~\eqref{eqn::bound_R5} and~\eqref{eqn::bound_R4R5_aux2}, we have a.s. on $E_{m+1}^c$,
	\begin{align} \label{eqn::bound_R4R5}
		R_4 R_5
		\le & \sin((2m)^{-1/2}/2)^{-\bar{\rho}/\kappa} \exp(2\pi |\mu|/\kappa) \exp\left( \frac{(1+\bar{\rho}/2)|\mu|}{\kappa} (1+\log 4)  \frac{1}{\sin((2m)^{-1/2}/2)} \right) \notag\\
		\le & \left( \pi \sqrt{2m} \right)^{\bar{\rho}/\kappa} \exp(2\pi |\mu|/\kappa) \exp\left( \frac{(1+\bar{\rho}/2)|\mu|}{\kappa} (1+\log 4) \pi \sqrt{2m} \right).
	\end{align}
\end{itemize}
Combining~\eqref{eqn::bound_R1} and $R_2\le 1, R_3\le 1$ and~\eqref{eqn::bound_R4R5}, when $m>(m_{\eqref{eqn::bound_R3}} \vee m_{\eqref{eqn::bound_R4R5}})$, we have a.s. on $E_{m+1}^c$,
\begin{align*}
	\frac{M_{\tau_{m+1}}}{M_{\tau_{m}} }
	\le& \left( \pi \sqrt{2m} \right)^{\bar{\rho}/\kappa} \exp(2\pi |\mu|/\kappa) \exp\left( (1+\log 4) \left| \frac{\bar{\rho}(4+\bar{\rho})-4\mu^2}{8\kappa} \right| \right) \\
	&\times \exp \left( \frac{(1+\bar{\rho}/2)|\mu|}{\kappa} (1+\log 4) \pi \sqrt{2m} \right).
\end{align*}
Letting 
\begin{align*}
	& m_{\eqref{eqn::bound_finitetime_RN}}:=m_{\eqref{eqn::bound_R3}} \vee m_{\eqref{eqn::bound_R4R5}}, \\
	& C'_{\eqref{eqn::bound_finitetime_RN}}:=2\pi |\mu|+\frac{1+\log 4}{8} \left| \bar{\rho}(4+\bar{\rho})-4\mu^2 \right|, \qquad  C_{\eqref{eqn::bound_finitetime_RN}}:=\bar{\rho} \log(\sqrt{2}\pi) + (1+\bar{\rho}/2) |\mu|(1+\log 4) \sqrt{2} \pi,
\end{align*}  
we obtain~\eqref{eqn::bound_finitetime_RN} as desired.
\end{proof}

\begin{lemma}\label{lem::radial_goodevents}
	Assume the same notation as in the proof of Proposition~\ref{prop::radialSLE_return}. The estimate~\eqref{eqn::radial_goodevents} on the good event holds for $\kappa<1$. 
\end{lemma}

\begin{proof}
Note that $\partial \U_s \cap (\U\setminus \gamma_{[0,\tau_S])}$ is a union of crosscuts of the form
\[
	\eta_j=\{ \exp(-s+\ii \theta): \vartheta_j^{\min}<\theta<\vartheta_j^{\max} \}.
\]
The event $\gamma_{[\tau_{S+m},\tau_{S+m+1}]} \cap \partial \U_s \neq \emptyset$ implies that there exists $j$ such that $\gamma_{[\tau_{S+m},\tau_{S+m+1}]} \cap \eta_j \neq \emptyset$, thus
\begin{align}\label{eqn::radialSLE_return_aux4}
\begin{split}
&\sum_{m=0}^{\infty} \PPradial{1}^{(\kappa;\bs{\rho};\mu)} \left[ \{\gamma_{[\tau_{S+m},\tau_{S+m+1}]} \cap \partial \U_s \neq \emptyset \}\cap E_{S+m+1}^c\right]\\
\le& \sup_{\{\eta_j\}} \sum_{m=0}^{\infty} \sum_j \PPradial{1}^{(\kappa;\bs{\rho};\mu)} \left[ \left\{\gamma_{[\tau_{S+m},\tau_{S+m+1}]}\cap \eta_j \neq \emptyset\right\}\cap E_{S+m+1}^c    \right] \\
	\le& \sup_{\{\eta_j\}} \sum_{m=0}^{\infty} \sum_j \PPradial{1}^{(\kappa;\bs{\rho};\mu)} \left[ \left\{\mathfrak{g}_{\tau_{S+m}} \left( \gamma_{[\tau_{S+m},\tau_{S+m+1}]} \right) \cap \mathfrak{g}_{\tau_{S+m}} \left( \eta_j \right) \neq \emptyset\right\}\cap E_{S+m+1}^c  \right].
\end{split}
\end{align}

We denote by $\PPradial{1}^{(\kappa)}$ the law of radial $\SLE_{\kappa}$ in $(\U;\ee^{\ii \xi_{\tau_{S+m}}};0)$. The Radon-Nikodym derivative between radial $\SLE_{\kappa}^{\mu}(\bs{\rho})$ and $\PPradial{1}^{(\kappa)}$ is given by $M_t$ in~\eqref{eqn::SLEkapparho_martingale_radial} in Lemma~\ref{lem::SLEkapparho_martingale_radial}. For $\kappa\le 4$, we have
\begin{align}
	&\PPradial{1}^{(\kappa;\bs{\rho};\mu)} \left[ \left\{\mathfrak{g}_{\tau_{S+m}} \left( \gamma_{[\tau_{S+m},\tau_{S+m+1}]} \right) \cap \mathfrak{g}_{\tau_{S+m}} \left( \eta_j \right) \neq \emptyset\right\}\cap E_{S+m+1}^c   \cond \LF_{S+m} \right] \notag\\
	=&\Eradial{1}^{(\kappa)} \left[\frac{M_{\tau_{S+m+1}}}{M_{\tau_{S+m}}}\one\left\{\mathfrak{g}_{\tau_{S+m}} \left( \gamma_{[\tau_{S+m},\tau_{S+m+1}]} \right) \cap \mathfrak{g}_{\tau_{S+m}} \left( \eta_j \right) \neq \emptyset\right\} \one\{E_{S+m+1}^c\}   \cond \LF_{S+m}\right]\notag\\
	\le & \exp \left( \frac{1}{\kappa} \left( C'_{\eqref{eqn::bound_finitetime_RN}}+C_{\eqref{eqn::bound_finitetime_RN}} \sqrt{S+m} \right)\right)\PPradial{1}^{(\kappa)}\left[ \left\{\mathfrak{g}_{\tau_{S+m}} \left( \gamma_{[\tau_{S+m},\tau_{S+m+1}]} \right) \cap \mathfrak{g}_{\tau_{S+m}} \left( \eta_j \right) \neq \emptyset\right\}\cap E_{S+m+1}^c   \cond \LF_{S+m} \right] \notag \\
	\le &\exp \left( \frac{1}{\kappa} \left( C'_{\eqref{eqn::bound_finitetime_RN}}+C_{\eqref{eqn::bound_finitetime_RN}} \sqrt{S+m} \right)\right) \PPradial{1}^{(\kappa)} \left[ \gamma_{[0,1+\log 4]} \cap \mathfrak{g}_{\tau_{S+m}} \left( \eta_j \right) \neq \emptyset \right],  
	\label{eqn::intersection_bound_aux1}
\end{align}
where the first inequality is due to~\eqref{eqn::bound_finitetime_RN}. 
Let $\hat{\gamma}_{\mathrm{fix}}\subset \overline{\U}_{S+m}$ be the straight line segment from $\gamma_{\tau_{S+m}}$ to $0$ and let $\gamma_{\mathrm{fix}}:=\mathfrak{g}_{\tau_{S+m}}(\hat{\gamma}_{\mathrm{fix}})$. Applying~\eqref{eqn::intersection_estimate} in Lemma~\ref{lem::intersection_estimate}, we have
\begin{align}\label{eqn::intersection_bound_aux2}
	\begin{split}
		\PPradial{1}^{(\kappa)} \left[ \gamma_{[0,1+\log 4]} \cap \mathfrak{g}_{\tau_{S+m}} \left( \eta_j \right) \neq \emptyset \right] \le & \exp\left( C_{\eqref{eqn::intersection_estimate}}/\kappa \right) \left( \LE_{\U} (\mathfrak{g}_{\tau_{S+m}} ( \eta_j ),\gamma_{\mathrm{fix}}) \right)^{8/\kappa-1} \\
		= & \exp\left( C_{\eqref{eqn::intersection_estimate}}/\kappa \right) \left( \LE_{\U\setminus \gamma_{[0,\tau_{S+m}]}} (\eta_j ,\hat{\gamma}_{\mathrm{fix}}) \right)^{8/\kappa-1}  \\
		\le & \exp\left( C_{\eqref{eqn::intersection_estimate}}/\kappa \right) \left( \LE_{\U\setminus \gamma_{[0,\tau_{S+m}]}} (\eta_j,\partial \U_{S+m}) \right)^{8/\kappa-1},
	\end{split}
\end{align}
where the second line is due to the conformal invariance of $\LE$ and the third line is due to the monotonicity of $\LE$. We also note that
\begin{align} \label{eqn::intersection_bound_aux3}
	\begin{split}
		\sum_{j} \left( \LE_{\U\setminus \gamma_{[0,\tau_{S+m}]}} (\eta_j,\partial \U_{S+m}) \right)^{8/\kappa-1} \le & \left( \sum_j \LE_{\U\setminus \gamma_{[0,\tau_{S+m}]}} (\eta_j,\partial \U_{S+m}) \right)^{8/\kappa-1} \\
		\le & \left( 2 \LE_{\U\setminus \gamma_{[0,\tau_{S+m}]}} (\partial \U_{s},\partial \U_{S+m}) \right)^{8/\kappa-1} \\
		\le & \left( C_{\eqref{eqn::intersection_bound_aux3}} \exp(-(S+m-s)/2) \right)^{8/\kappa-1} ,
	\end{split}
\end{align}
where the first inequality is due to $L^p$ norm monotonicity, the second inequality is due to~\cite[Lemma~3.3]{field2015escape}, the third inequality is due to~\cite[Eq.~(2.5) and scaling]{field2015escape}, and $C_{\eqref{eqn::intersection_bound_aux3}}$ is a universal constant.
Combing~\eqref{eqn::intersection_bound_aux1}-\eqref{eqn::intersection_bound_aux3} we have
\begin{align*}
	&\sum_{j} \PPradial{1}^{(\kappa;\bs{\rho};\mu)} \left[ \left\{\mathfrak{g}_{\tau_{S+m}} \left( \gamma_{[\tau_{S+m},\tau_{S+m+1}]} \right) \cap \mathfrak{g}_{\tau_{S+m}} \left( \eta_j \right) \neq \emptyset\right\}\cap E_{S+m+1}^c \right] \notag\\
	\le & \exp\left( \frac{1}{\kappa} \left( C_{\eqref{eqn::intersection_estimate}}+C'_{\eqref{eqn::bound_finitetime_RN}}+(8-\kappa) \log C_{\eqref{eqn::intersection_bound_aux3}} \right) \right)  \exp\left( -\frac{8-\kappa}{2\kappa}(S+m-s)+\frac{C_{\eqref{eqn::bound_finitetime_RN}}}{\kappa} \sqrt{S+m} \right).
\end{align*}
Letting
\[ 
S_{\eqref{eqn::radial_goodevents}}=(4C_{\eqref{eqn::bound_finitetime_RN}}^2) \vee (3s) \vee m_{\eqref{eqn::bound_finitetime_RN}}, \qquad C_{\eqref{eqn::radial_goodevents}}=C_{\eqref{eqn::intersection_estimate}}+C'_{\eqref{eqn::bound_finitetime_RN}}+8\log C_{\eqref{eqn::intersection_bound_aux3}},	\]
we have
\begin{align} \label{eqn::intersection_bound}
\begin{split}
	&\sum_{j} \PPradial{1}^{(\kappa;\bs{\rho};\mu)} \left[ \left\{\mathfrak{g}_{\tau_{S+m}} \left( \gamma_{[\tau_{S+m},\tau_{S+m+1}]} \right) \cap \mathfrak{g}_{\tau_{S+m}} \left( \eta_j \right) \neq \emptyset\right\}\cap E_{S+m+1}^c    \right] \le \exp \left( \frac{1}{\kappa} \left(C_{\eqref{eqn::radial_goodevents}}-(S+m) \right) \right).
\end{split}
\end{align}

Plugging~\eqref{eqn::intersection_bound} into~\eqref{eqn::radialSLE_return_aux4}, when $\kappa<1$ and $S>S_{\eqref{eqn::radial_goodevents}}$, we have
\begin{align*} 
\begin{split}
\sum_{m=0}^{\infty} \PPradial{1}^{(\kappa;\bs{\rho};\mu)} \left[ \{\gamma_{[\tau_{S+m},\tau_{S+m+1}]} \cap \partial \U_s \neq \emptyset \}\cap E_{S+m+1}^c\right]
\le &\sum_{m=0}^{\infty} \exp \left( \frac{1}{\kappa} \left(C_{\eqref{eqn::radial_goodevents}}-(S+m) \right) \right) \\
=& \frac{\exp\left( \frac{1}{\kappa} (C_{\eqref{eqn::radial_goodevents}}-S)  \right)}{1-\exp(-1/\kappa)} \\
\le& \frac{\exp\left( \frac{1}{\kappa} (C_{\eqref{eqn::radial_goodevents}}-S)  \right)}{1-\ee^{-1}},
\end{split}
\end{align*}
which gives~\eqref{eqn::radial_goodevents} as desired.
\end{proof}
\section{Large deviation for multi-radial SLE with spiral}
\label{sec::LDP_radialSLE}
The goal of this section is to prove large deviation in Theorem~\ref{thm::radialSLE_LDP}. Analogously to the chordal case, we first show the large deviation for radial SLE in finite time in Proposition~\ref{prop::finitetime_LDP_radial} and then prove Theorem~\ref{thm::radialSLE_LDP} using Proposition~\ref{prop::finitetime_LDP_radial} and return estimate in Proposition~\ref{prop::radialSLE_return}. We first introduce notations for finite-time curve space.

\paragraph*{Curve spaces.} Fix $n\geq 1$ and $n$-polygon $(\Omega; \bs{x})$ with $z\in\Omega$, we denote by $\chamber_{\bs{T}}=\chamber_{\bs{T}}(\Omega; \bs{x}; z)$ the space of simple multichords $\bs{\gamma}_{[\bs{0},\bs{T}]}=(\gamma^1_{[0,T_1]}, \ldots, \gamma^n_{[0,T_n]})$ in $\Omega$ starting from $\bs{x}$, that is, 
\[
	\gamma_{(0,T_j]}^j \subset \Omega \text{ for } 1\le j\le n, \quad \text{and furthermore }\quad \gamma_{(0,T_j]}^j \cap \gamma_{(0,T_k]}^k = \emptyset \quad \text{ for all } j\neq k.
\]
Here $\bs{\gamma}_{[\bs{0}, \bs{T}]}$ is parameterized by $n$-time parameter of $\varphi_{\U}(\bs{\gamma})=(\varphi_{\U}(\gamma^1), \ldots, \varphi_{\U}(\gamma^n))$ (see~\eqref{eqn::multitime_def_radial}), where $\varphi_{\U}:\Omega\to \U$ is a conformal map with $\varphi_{\U}(z)=0$.
We equip a metric analogous to \eqref{eqn::dist_chamber} on the truncated curve space $\chamber_{\bs{T}}$:
\begin{equation} \label{eqn::para_curve_metric}
	\dist_{\chamber}(\bs{\gamma}_{[\bs{0},\bs{T}]}, \tilde{\bs{\gamma}}_{[\bs{0},\bs{T}]}):=\sup_{\bs{t}\in [\bs{0},\bs{\infty}]} \sup_{1\le j\le n} \left| \varphi_{\U}\left(\gamma^j_{t_j\wedge T_j}\right)-\varphi_{\U}\left(\tilde{\gamma}^j_{t_j\wedge T_j}\right) \right|, \quad  \bs{\gamma}_{[\bs{0},\bs{T}]}, \tilde{\bs{\gamma}}_{[\bs{0},\bs{T}]}\in \chamber_{\bs{T}}.
\end{equation}
In this way, we obtain an incomplete metric space $(\chamber_{\bs{T}}(\Omega;\bs{x};z), \dist_{\chamber})$.
This space differs from its chordal counterpart $(\chamber_{T}(\Omega; \bs{x}, y), \dist_{\chamber})$, since the reference point $z$ is in the interior of $\Omega$ instead of on the boundary, and thus the multi-time parameterization is different. 

\begin{proposition}
	\label{prop::finitetime_LDP_radial}
Fix $\mu\in\R$ and $n\ge 1$ and $n$-polygon $(\Omega; \bs{x})$ with $z\in\Omega$. Suppose $\bs{\gamma}\sim \PPradial{n}^{(\kappa; \mu)}(\Omega; \bs{x}; z)$ is $n$-radial $\SLE_{\kappa}^{\mu}$ in $(\Omega; \bs{x}; z)$. Then the family of laws of $\bs{\gamma}_{[\bs{0},\bs{T}]}$ under $\PPradial{n}^{(\kappa; \mu)}(\Omega; \bs{x}; z)$ satisfies large deviation principle in the space $(\chamber_{\bs{T}}(\Omega;\bs{x};z),\dist_{\chamber})$ as $\kappa\to 0+$ with good rate function $\rateradial{n}^{(\mu)}(\Omega; \bs{x}; z; \cdot)$ which will be defined in Definition~\ref{def::multitime_energy_radial}.
\end{proposition}

The rest of this section is organized as follows. In Section~\ref{subsec::multitime_radial}, we introduce the radial multi-time parameter, and recall the definition and basic properties of multi-time radial Loewner energy from~\cite{ChenHuangPeltolaWumultitimeenergy}. In Section~\ref{subsec:::LDP_finite_time_multiradial}, we establish the large deviation of $n$-radial $\SLE$ with spiral in finite time (Proposition~\ref{prop::finitetime_LDP_radial}). In Section~\ref{subsec:::LDP_infinite_time_multiradial}, we extend this result to the infinite-time setting, using Proposition~\ref{prop::finitetime_LDP_radial} and the return estimate in Proposition~\ref{prop::radialSLE_return}. 
In Section~\ref{subsec::radial_rho_LDP}, we adapt the procedures from  Sections~\ref{subsec:::LDP_finite_time_multiradial}-\ref{subsec:::LDP_infinite_time_multiradial} to prove the LDP for radial SLE with force points (Proposition~\ref{prop::radialSLE_LDP}).

The next two sections present direct consequences of large deviation of multi-radial SLE. In Section~\ref{subsec::common_time_radial}, we prove the LDP for Dyson circular ensemble (Proposition~\ref{prop::LDP_DysonBM_radial}), utilizing the relation between $n$-radial $\SLE$ with spiral and Dyson circular ensemble (see Lemma~\ref{lem::DysonBM_driving_function_radial}). In Section~\ref{subsec::bp_radial_proof}, we prove the boundary perturbation property of multi-time radial Loewner energy (Proposition~\ref{prop::bp_radial}) using the boundary perturbation property of multi-radial SLE proved in~\cite[Proposition~1.4]{HuangPeltolaWuMultiradialSLEResamplingBP} (see Lemma~\ref{lem::bp_multiradialkappa}) and large deviation principle in Theorem~\ref{thm::radialSLE_LDP}.

\subsection{Multi-time radial Loewner energy}
\label{subsec::multitime_radial}
\paragraph*{Multi-time parameter (radial).} Fix $n\ge 1$ and $\bs{\theta}=(\theta_1,\ldots,\theta_n)\in \LX_n^{\U}$. 
Consider an $n$-tuple $\bs{\gamma}_{[\bs{0},\bs{t}]}=(\gamma^1_{[0,t_1]}, \ldots, \gamma^n_{[0,t_n]})\in \overline{\chamber}_{\bs{t}}(\U;\ee^{\ii \bs{\theta}};0)$\footnote{That is, the curves are allowed to touch themselves or each other, but not to cross.} parameterized by $\bs{t}=(t_1, \ldots, t_n)\in[0,\infty)^n$.
We define the following normalized conformal transformations:
\begin{itemize}
	\item $\mathfrak{g}_{t_j}^j$ is the conformal map from the component of $\U\setminus\gamma^j_{[0,t_j]}$ that contains $0$ onto $\U$ with $\mathfrak{g}_{t_j}^j(0)=0$ and $(\mathfrak{g}_{t_j}^j)'(0)>0$, where $1\le j\le n$.
	\item $\mathfrak{g}_{\bs{t}}$ is the conformal map from the component of $\U\setminus\cup_{j=1}^n \gamma^j_{[0,t_j]}$ that contains $0$ onto $\U$ with $\mathfrak{g}_{\bs{t}}(0)=0$ and $\mathfrak{g}_{\bs{t}}'(0)>0$.
	\item $\mathfrak{g}_{\bs{t}, j}$ is the conformal map from the component of $\U \setminus \mathfrak{g}_{t_j}^j\left(\cup_{i\neq j}\gamma^i_{[0,t_i]}\right)$ that contains $0$ onto $\U$ with $\mathfrak{g}_{\bs{t}, j}(0)=0$ and $\mathfrak{g}'_{\bs{t}, j}(0)> 0$, where $1\le j\le n$. 
\end{itemize}
Using these notations, we have $\mathfrak{g}_{\bs{t}}=\mathfrak{g}_{\bs{t}, j}\circ \mathfrak{g}_{t_j}^j$ for $1\le j\le n$. Let $\covmap_{t_j}^{j},\covmap_{\bs{t}},\covmap_{\bs{t},j}$ be the covering maps of $\mathfrak{g}_{t_j}^{j},\mathfrak{g}_{\bs{t}},\mathfrak{g}_{\bs{t},j}$, respectively. 
We say that the $n$-tuple $\bs{\gamma}_{\bs{t}}$ of curves has $n$-time parameter if $\mathfrak{g}^{j}$ is parameterized by capacity:
\begin{equation}\label{eqn::multitime_def_radial}
	(\mathfrak{g}_{t_j}^j)'(0)= \ee^{t_j},\quad \text{ for }1\leq j \leq n.
\end{equation}
Denote by $\xi^j$ the driving function of each $\gamma^j$, and we define the driving function of the $n$-tuple $\bs{\gamma}_{\bs{t}}$, started at $\bs{\mslitdriv}_0=(\theta_1, \ldots, \theta_n)\in\LX_n^{\U}$, by
\begin{equation}\label{eqn::multitime_driving_radial}
	\bs{\mslitdriv}_{\bs{t}}=(\bs{\mslitdriv}^1_{\bs{t}}, \ldots, \bs{\mslitdriv}^n_{\bs{t}}), \qquad \text{with }\bs{\mslitdriv}^j_{\bs{t}}=\covmap_{\bs{t}, j}(\xi^j_{t_j}), \qquad \text{for }1\le j\le n.
\end{equation}
A standard calculation shows that 
\begin{equation}\label{eqn::expansion_mslitdriv}
	\ud \mslitdriv_{\bs{t}}^j = \left( \covmap_{\bs{t},j}'\left(\xi_{t_j}^j\right) \dot{\xi}_{t_j}^j - 3\covmap_{\bs{t},j}''\left(\xi_{t_j}^j\right)\right) \ud t_j+ \sum_{\ell\ne j} \cot \left( \frac{\mslitdriv_{\bs{t}}^j - \mslitdriv_{\bs{t}}^{\ell}}{2} \right) \covmap_{\bs{t},\ell}'\left(\xi_{t_\ell}^\ell\right)^2 \ud t_\ell,\quad \text{ for } 1\leq j \leq n.
\end{equation}
The mapping-out function $\mathfrak{g}_{\bs{t}}$ solves the Loewner equation
\begin{equation*}
	\ud \mathfrak{g}_{\bs{t}}(z)=\sum_{j=1}^n \partial_{t_j}\mathfrak{g}_{\bs{t}}(z)\ud t_j=\mathfrak{g}_{\bs{t}}(z)\sum_{j=1}^n(\covmap'_{\bs{t},j}(\xi_{t_j}^j))^2\left(\frac{\ee^{\ii\mslitdriv_{\bs{t}}^{j}}+\mathfrak{g}_{\bs{t}}(z)}{\ee^{\ii\mslitdriv_{\bs{t}}^{j}}-\mathfrak{g}_{\bs{t}}(z)} \right)\ud t_j. 
\end{equation*}

Fix $n\geq 1$ and $n$-polygon $(\U; \ee^{\ii \bs{\theta}};0)$ with $\bs{\theta}=(\theta_1, \ldots, \theta_n)\in\LX_n^{\U}$. 
Consider an $n$-tuple $\bs{\gamma}_{[\bs{0},\bs{t}]}=(\gamma^1_{[0,t_1]}, \ldots, \gamma^n_{[0,t_n]})$ of disjoint simple curves in $\U$ parameterized by multi-time parameter. Define 
$\blm_{\bs{t}}$ to be the unique potential solving the exact differential equation (see~\cite[Lemma~2.6]{HuangPeltolaWuMultiradialSLEResamplingBP})
\begin{align}\label{eqn::mt_def_radial_mart}
	\ud  \blm_{\bs{t}} 
	= \sum_{j=1}^{n} \partial_{t_j} \blm_{\bs{t}} \ud t_j
	= \sum_{j=1}^{n} \Big( \! -\frac{1}{3} \LS \covmap_{\bs{t},j}(\xi_{t_j}^{j})+\frac{1}{6} \big(1- (\covmap'_{\bs{t},j}(\xi_{t_j}^{j}))^2 \big) \Big) \ud t_j, 
	\qquad \blm_{\bs{0}} = 0 ,
\end{align}
with $\LS\covmap=\frac{\covmap'''}{\covmap'}-\frac{3}{2}\big(\frac{\covmap''}{\covmap'}\big)^2$ denoting the Schwarzian derivative of a function $\covmap$. It is proved in~\cite[Lemma~2.7]{HuangPeltolaWuMultiradialSLEResamplingBP} that the solution to~\eqref{eqn::mt_def_radial_mart} 
can be described in terms of Brownian loop measure as
\begin{align*}\label{eqn::mt_blm_radial}
	\blm_{\bs{t}} = \blm \left(\U;\gamma_{[0,t_1]}^{1},\ldots,\gamma_{[0,t_n]}^{n} \right).
\end{align*}
Consequently, the measure $\blm_{\bs{t}}$ is finite as long as $\gamma_{[0,t_1]}^{1},\ldots,\gamma_{[0,t_n]}^{n}$ are disjoint.

\paragraph*{Semi-classical limits of BPZ equations.}
We define 
\begin{equation}\label{eqn::multi_time_energy_radial_LV}
	\LVradial{n}^{(\mu)}(\bs{\theta})=-2 \sum_{1\le i<j\le n} \log \sin\left( \frac{\theta_j-\theta_i}{2} \right) - \mu \sum_{j=1}^{n} \theta_j, \quad \text{for } \bs{\theta} = (\theta_1, \ldots, \theta_n)\in \LX_n^{\U}.
\end{equation}
The function $\LVradial{n}^{(\mu)}:\LX_n^{\U}\to \R$ is the semi-classical limit of the partition function $\LZradial{n}^{(\kappa;\mu)}$ in~\eqref{eqn::nradial_pf_U}:
\begin{equation*}
	\LVradial{n}^{(\mu)}(\bs{\theta}) = -\lim_{\kappa\to 0} \kappa \log \LZradial{n}^{(\kappa;\mu)}(\bs{\theta});
\end{equation*}
and it satisfies the semi-classical limits of the BPZ equations:
\begin{equation}\label{eqn::radialBPZ_semiclassical_mu}
	(\partial_j\LV)^2-\sum_{\ell\neq j}\left(2\cot\left(\frac{\theta_{\ell}-\theta_j}{2}\right)\partial_{\ell}\LV+\frac{3}{\sin^2\left(\frac{\theta_{\ell}-\theta_j}{2}\right)}\right)=\mu^2+1-n^2, \qquad\text{for }1\le j\le n. 
\end{equation}

\begin{definition}[Multi-time radial Loewner energy]
\label{def::multitime_energy_radial}
Fix $\mu\in\R$ and $n\ge 2$ and $n$-polygon $(\Omega; \bs{x}) = (\Omega; x_1, \ldots, x_n)$ with $z\in \Omega$.
Let $\varphi: \Omega\to \U$ be a conformal map with $\varphi(z)=0$. 
For $\bs{\gamma}_{[\bs{0},\bs{t}]}\in \overline{\chamber}_{\bs{t}}(\Omega; \bs{x};z)$, we denote by $\xi_{t_j}^j$ the driving function of $\varphi(\gamma^j_{[0,t_j]})$ and denote by $\bs{\mslitdriv}_{\bs{t}}$ the driving function of the $n$-tuple $\varphi(\bs{\gamma}_{[\bs{0},\bs{t}]})$.
We define the \emph{truncated $n$-time radial Loewner energy} $\rateradial{n}^{(\mu)} (\Omega; \bs{x}; z; \bs{\gamma}_{[\bs{0}, \bs{t}]})$ to be 
\begin{equation} \label{eqn::multi_time_energy_radial_def}
\begin{split}
	\rateradial{n}^{(\mu)} (\Omega; \bs{x}; z; \bs{\gamma}_{[\bs{0}, \bs{t}]}) \; =& \; \sum_{j=1}^{n} \rateradial{1}(\Omega; x_j; z; \gamma_{[0, t_j]}^j) +12 \blm(\Omega; \gamma^1_{[0,t_1]}, \ldots, \gamma^n_{[0,t_n]})-\frac{3}{2}\sum_{j=1}^n t_j\\
	&-\frac{1}{2}(n^2-4-\mu^2)\log \mathfrak{g}'_{\bs{t}}(0)
	-3\sum_{j=1}^n \log\covmap'_{\bs{t}, j}(\xi^j_{t_j})+\LVradial{n}^{(\mu)}\left(\bs{\mslitdriv}_{\bs{t}}\right)-\LVradial{n}^{(\mu)}(\bs{\theta}).
\end{split}
\end{equation}
Here, the quantity $\rateradial{1}$ is the radial Loewner energy defined in~\eqref{eqn::energy_radial_single}, and $\blm$ is the Brownian loop measure defined in~\eqref{eqn::mt_def_radial_mart}, and $\LVradial{n}^{(\mu)}(\bs{\theta})$ is defined in \eqref{eqn::multi_time_energy_radial_LV}. Using~\eqref{eqn::radialBPZ_semiclassical_mu}, it is proved in~\cite{ChenHuangPeltolaWumultitimeenergy} that $\rateradial{n}^{(\mu)} (\Omega; \bs{x}; z; \bs{\gamma}_{[\bs{0}, \bs{t}]})$ is the unique solution to the following system of exact differential equations:
\begin{align} \label{eqn::multi_time_energy_radial_def2}
\begin{split}
	\ud \rateradial{n}^{(\mu)}(\Omega; \bs{x}; z; \bs{\gamma}_{[\bs{0},\bs{t}]}) 
	= \; & \sum_{j=1}^{n} \partial_{t_j} \rateradial{n}^{(\mu)}(\Omega; \bs{x}; z; \bs{\gamma}_{[\bs{0},\bs{t}]}) \ud t_j
	\\
	= \; &  \frac{1}{2}\sum_{j=1}^{n} \bigg( \dot{\xi}_{t_j}^{j} - \covmap'_{\bs{t},j}(\xi_{t_j}^{j}) \bigg( \mu+\sum_{i\neq j} \cot \bigg(\frac{\mslitdriv_{\bs{t}}^{j} - \mslitdriv_{\bs{t}}^{i}}{2}\bigg) \bigg) 
	- 3 \, \frac{\covmap''_{\bs{t},j}(\xi_{t_j}^{j})}{\covmap'_{\bs{t},j}(\xi_{t_j}^{j})}\bigg)^2 \ud t_j ,
	\\
	\rateradial{n}^{(\mu)}(\Omega; \bs{x}; z; \bs{\gamma}_{[\bs{0},\bs{0}]}) 
	= \; & 0.
\end{split}
\end{align}
In particular, the quantity $\rateradial{n}^{(\mu)} (\Omega; \bs{x}; z; \bs{\gamma}_{[\bs{0}, \bs{t}]})$ is increasing in $\bs{t}$, and we define \emph{$n$-time radial Loewner energy} of $\bs{\gamma} \in \overline{\chamber}_{n\mathrm{\textnormal{-}rad}}(\Omega; \bs{x}; z)$ to be its limit: 
	\begin{equation}\label{eqn::multi_time_energy_radial}
		\rateradial{n}^{(\mu)}(\Omega; \bs{x}; z; \bs{\gamma}):=\lim_{\bs{t}\to +\bs{\infty}} \rateradial{n}^{(\mu)}(\Omega; \bs{x}; z; \bs{\gamma}_{[\bs{0}, \bs{t}]}). 
	\end{equation}
\end{definition}

The following basic property of the energy will be used in the proof of Proposition~\ref{prop::finitetime_LDP_radial}.

\begin{lemma}\label{lem::finite_energy_simple}
If $\rateradial{n}^{(\mu)}(\Omega;\bs{x};z;\bs{\gamma}_{[\bs{0},\bs{T}]})<\infty$, then the component curves $\gamma_{[0,T_j]}^j$ are simple and mutually disjoint for $1\le j\le n$.
\end{lemma}

The proof of Lemma~\ref{lem::finite_energy_simple} relies on~\cite[Theorem~3.9]{AbuzaidHealeyPeltolaLargeDeviationDysonBM}: 
\begin{lemma}[{\cite[Theorem~3.9]{AbuzaidHealeyPeltolaLargeDeviationDysonBM}}]
	\label{lem::finite_Dirichlet_energy_simple}
Suppose $T<\infty$ and $\xi_{[0,T]}$ has finite Dirichlet energy~\eqref{eqn::energy_radial_single}. Suppose that $\lambda:[0,T]\to (0,\infty)$ is a weight function that $\inf_{t\in [0,T]} |\lambda_t|>0$. Consider the radial Loewner chain with driving function $\xi$ and weight function $\lambda$:
	\begin{equation} \label{eqn::radial_chain_weight}
		\partial_t \mathfrak{g}_t(z)=\lambda_t \mathfrak{g}_t(z) \frac{\exp(\ii \xi_t)-\mathfrak{g}_t(z)}{\exp(\ii \xi_t)+\mathfrak{g}_t(z)}, \qquad  0\le t\le T,
	\end{equation}
	then the hull $\gamma_{[0,T]}$ generated by the Loewner chain~\eqref{eqn::radial_chain_weight} is a simple curve in $\U$.
\end{lemma}

\begin{proof}[Proof of Lemma~\ref{lem::finite_energy_simple}]
It suffices to show the conclusion for $(\Omega;\bs{x};z)=(\U;\ee^{\ii \bs{\theta}};0)$ and we eliminate them from the notation. We show that $\gamma_{[0,T_1]}^{1}$ is a simple curve in $\U\setminus \left( \bigcup_{j=2}^{n} \gamma^j_{[0,T_j]} \right)$. 
To this end, we will prove that 
\begin{equation} \label{eqn::finite_energy_simple_aux4}
	\inf_{\stackrel{\bs{t}\in [\bs{0},\bs{T}]}{2\le j\le n}} \left| \mslitdriv_{\bs{t}}^{j}-\mslitdriv^1_{\bs{t}} \right|>0, \qquad \text{ and } \qquad \inf_{\bs{t}\in [\bs{0},\bs{T}]} \covmap'_{\bs{t},1}\left(\xi_{t_1}^{1} \right)^2>0.
\end{equation}
	
Fix $(t_2,\ldots,t_n)$ with $0\le t_j\le T_j$ for $2\le j\le n$ and denote $\bs{s}=(t_2, \ldots, t_n)$. We view $\bs{t}=(t_1,\bs{s})$ as a function in $t_1\in [0,T_1]$. By \eqref{eqn::multi_time_energy_radial_def2}, we have
\begin{align}
	\rateradial{n}^{(\mu)}(\bs{\gamma}_{[\bs{0}, (T_1, \bs{s})]}) \ge & \frac{1}{2} \int_0^{T_1}\left( \dot{\xi}_{t_1}^{1} - \covmap'_{\bs{t},1}\left(\xi_{t_1}^{1}\right) \left( \mu+\sum_{j=2}^n \cot\left(\frac{\mslitdriv^1_{\bs{t}}-\mslitdriv_{\bs{t}}^{j}}{2}\right)\right)-3\frac{\covmap''_{\bs{t},1}\left(\xi_{t_1}^{1}\right)}{\covmap'_{\bs{t},1}\left(\xi_{t_1}^{1}\right)}\right)^2 \ud t_1 \notag \\
	= & \frac{1}{2} \int_0^{T_1}\left( \frac{ \partial_{t_1} \mslitdriv^1_{\bs{t}} }{ \covmap'_{\bs{t},1}\left(\xi_{t_1}^{1}\right) } - \covmap'_{\bs{t},1}\left(\xi_{t_1}^{1}\right) \left( \mu+\sum_{j=2}^n \cot\left(\frac{\mslitdriv^1_{\bs{t}}-\mslitdriv_{\bs{t}}^{j}}{2}\right)\right)\right)^2 \ud t_1 \tag{due to~\eqref{eqn::expansion_mslitdriv}} \\
	\ge & \underbrace{\frac{1}{2} \int_0^{T_1} \left( \frac{ \partial_{t_1} \mslitdriv^1_{\bs{t}} }{ \covmap'_{\bs{t},1}\left(\xi_{t_1}^{1}\right) } \right)^2 \ud t_1}_{:=J_1} \underbrace{- \int_0^{T_1} \partial_{t_1} \mslitdriv^1_{\bs{t}} \left( \mu+\sum_{j=2}^n \cot\left(\frac{\mslitdriv^1_{\bs{t}}-\mslitdriv_{\bs{t}}^{j}}{2}\right)\right) \ud t_1}_{:=J_2}.   \label{eqn::finite_energy_simple_aux1}
\end{align}
For $J_1$, we have
\begin{equation} \label{eqn::finite_energy_simple_aux2}
	J_1\ge \frac{1}{2} \int_0^{T_1} \left(  \partial_{t_1} \mslitdriv^1_{\bs{t}} \right)^2 \ud t_1 \ge 0.
\end{equation}
For $J_2$, we denote $f(\bs{t})|_{0}^{T_1}:=f(\bs{t})|_{t_1=T_1}-f(\bs{t})|_{t_1=0}$ for functional $f$ on $\bs{t}$, then we have
\begin{align}
	J_2= & -\mu \mslitdriv^1_{\bs{t}}|_{0}^{T_1} - \sum_{j=2}^{n}  \int_{0}^{T_1} \cot\left(\frac{\mslitdriv^1_{\bs{t}}-\mslitdriv_{\bs{t}}^{j}}{2}\right) \partial_{t_1} \left( \mslitdriv^1_{\bs{t}}-\mslitdriv_{\bs{t}}^{j} \right) \ud t_1 - \sum_{j=2}^{n} \int_{0}^{T_1} \cot\left(\frac{\mslitdriv^1_{\bs{t}}-\mslitdriv_{\bs{t}}^{j}}{2}\right) \partial_{t_1} \mslitdriv_{\bs{t}}^{j} \ud t_1 \notag \\
	\ge & -\mu \mslitdriv^1_{\bs{t}}|_{0}^{T_1} - \sum_{j=2}^{n}  \int_{0}^{T_1} \cot\left(\frac{\mslitdriv^1_{\bs{t}}-\mslitdriv_{\bs{t}}^{j}}{2}\right) \partial_{t_1} \left( \mslitdriv^1_{\bs{t}}-\mslitdriv_{\bs{t}}^{j} \right) \ud t_1 \tag{due to~\eqref{eqn::expansion_mslitdriv}} \\
	= & -\mu \mslitdriv^1_{\bs{t}}|_{0}^{T_1} -2 \sum_{j=2}^n \log \sin \left(\frac{\mslitdriv_{\bs{t}}^{j}-\mslitdriv^1_{\bs{t}}}{2}\right) |_{0}^{T_1}. \label{eqn::finite_energy_simple_aux3}
\end{align}
Combining~(\ref{eqn::finite_energy_simple_aux1},\ref{eqn::finite_energy_simple_aux2},\ref{eqn::finite_energy_simple_aux3}), we have 
	\[-\mu \mslitdriv^1_{\bs{t}}|_{0}^{T_1} -2 \sum_{j=2}^n \log \sin \left(\frac{\mslitdriv_{\bs{t}}^{j}-\mslitdriv^1_{\bs{t}}}{2}\right) |_{0}^{T_1}\le\rateradial{n}^{(\mu)}(\bs{\gamma}_{[\bs{0},(T_1,\bs{s})]})<\infty.\]
This implies the first control in~\eqref{eqn::finite_energy_simple_aux4}.
The first control in~\eqref{eqn::finite_energy_simple_aux4} implies that $\gamma^1_{[0,T_1]}$ has a positive distance from $\bigcup_{j=2}^{n} \gamma^j_{[0,T_j]}$, which gives the second control in~\eqref{eqn::finite_energy_simple_aux4}.
	Moreover, combining~(\ref{eqn::finite_energy_simple_aux1},\ref{eqn::finite_energy_simple_aux2},\ref{eqn::finite_energy_simple_aux3}) again, the process $\mslitdriv_{(\cdot,T_2,\ldots,T_n)}^{1}$ has finite Dirichlet energy up to time $T_1$. Since $\mslitdriv_{(\cdot,T_2,\ldots,T_n)}^{1}$ is the driving function of $\mathfrak{g}_{(0,T_2,\ldots,T_n)}\left( \gamma_{[0,T_1]}^{1} \right)$ with weight function $\covmap'_{\bs{t},1}\left(\xi_{t_1}^{1} \right)^2$, combining~\eqref{eqn::finite_energy_simple_aux4} and Lemma~\ref{lem::finite_Dirichlet_energy_simple}, the process $\mathfrak{g}_{(0,T_2,\ldots,T_n)}\left( \gamma_{[0,T_1]}^{1} \right)$ is a simple curve in $\U$, which implies that $\gamma_{[0,T_1]}^{1}$ is a simple curve in $\U\setminus \left( \bigcup_{j=2}^{n} \gamma^j_{[0,T_j]} \right)$. Similarly, the process $\gamma_{[0,T_j]}^j$ is a simple curve in $\U\setminus \left( \cup_{\ell\neq j} \gamma_{[0,T_{\ell}]}^{\ell} \right)$ for $2\le j\le n$, and we obtain the desired result.
\end{proof}

\subsection{Large deviation in finite time: proof of Proposition~\ref{prop::finitetime_LDP_radial}}
\label{subsec:::LDP_finite_time_multiradial}
In this section, we prove Proposition~\ref{prop::finitetime_LDP_radial}. The proof relies on the multi-time martingale for $n$-radial SLE with spiral presented in Lemma~\ref{lem::nradial_mart} and the large deviation principle for SLE curves in finite time presented in Lemma~\ref{lem::finite_radialSLE_LDP}.

\begin{lemma}[{\cite[Lemma~2.11]{HuangPeltolaWuMultiradialSLEResamplingBP}}]
\label{lem::nradial_mart}
Fix $\kappa\in (0,4], \mu\in\R$ and $n\ge 2$ and $\bs{\theta}=(\theta_1, \ldots, \theta_n)\in \LX_n^{\U}$. For each $j\in \{1,\ldots, n\}$, let $\gamma^{j}$ be radial $\SLE_{\kappa}$ in $(\U; \ee^{\ii\theta_j}; 0)$ and let $\mathsf{P}_n^{(\kappa)}$ be the probability measure on $\bs{\gamma}=(\gamma^{1}, \ldots, \gamma^{n})$ under which the curves are independent. 
We parameterize $\bs{\gamma}$ by $n$-time parameter $\bs{t}$ and let $\bs{\mslitdriv}_{\bs{t}}$ denote the driving function as in~\eqref{eqn::multitime_driving_radial}. Define the process 
\begin{equation}\label{eqn::nradial_multitime_mart}
	M_{\bs{t}}(\LZradial{n}^{(\kappa; \mu)})
	:= \one_{\LE_{\emptyset}(\bs{\gamma}_{[\bs{0},\bs{t}]})} \, \exp\bigg(\frac{\mathfrak{c}}{2}\blm_{\bs{t}} - \tilde{\mathfrak{b}}\sum_{j=1}^{n}t_j\bigg) 
	\; |\mathfrak{g}_{\bs{t}}'(0)|^{\tilde{\mathfrak{b}}+\frac{n^2-1-\mu^2}{2\kappa}} 
	\Big(\prod_{j=1}^{n} \covmap_{\bs{t},j}'(\xi_{t_j}^{j}) \Big)^{\mathfrak{b}}
	\LZradial{n}^{(\kappa; \mu)}(\bs{\mslitdriv}_{\bs{t}}),
\end{equation}
where $\LE_{\emptyset}(\bs{\gamma}_{[\bs{0},\bs{t}]}) = \{ \gamma_{[0,t_j]}^{j} \cap \gamma_{[0,t_i]}^{i}=\emptyset, \, \forall i\neq j\}$ is the event that different curves are disjoint, the parameters $\mathfrak{b}, \mathfrak{c}$ are defined in~\eqref{eqn::parameters_b_c}, and
\begin{equation*}\label{eqn::parameter_btilde}
\tilde{\mathfrak{b}}=\frac{(6-\kappa)(\kappa-2)}{8\kappa},
\end{equation*}
and $\blm_{\bs{t}}$ is defined in~\eqref{eqn::mt_def_radial_mart}, and $\LZradial{n}^{(\kappa; \mu)}$ is the partition function defined in~\eqref{eqn::nradial_pf_U}. Then $M_{\bs{t}}(\LZradial{n}^{(\kappa; \mu)})$ is an $n$-time parameter local martingale with respect to $\mathsf{P}_n^{(\kappa)}$. Moreover, the law of $\mathsf{P}_n^{(\kappa)}$ weighted by $M_{\bs{t}}(\LZradial{n}^{(\kappa; \mu)})$ is the same as $n$-radial $\SLE_{\kappa}^{\mu}$ in $(\U; \ee^{\ii\bs{\theta}};0)$ when restricted to the event $\LE_{\emptyset}(\bs{\gamma}_{[\bs{0},\bs{t}]})$. 
\end{lemma}

If we set $t_1=t\ge 0$ and $t_2=\cdots=t_n=0$, the martingale~\eqref{eqn::nradial_multitime_mart} becomes
\begin{equation*}\label{eqn::multiradialSLE_mart_marginal}
M_t(\LZradial{n}^{(\kappa; \mu)})=\mathfrak{g}_t'(0)^{\frac{n^2-1-\mu^2}{2\kappa}}\times\prod_{j=2}^n \covmap_t(\theta_j)^{\mathfrak{b}}\times\LZradial{n}^{(\kappa; \mu)}(\xi_t, \covmap_t(\theta_2), \ldots, \covmap_t(\theta_n)),
\end{equation*}
which coincides with~\eqref{eqn::SLEkapparho_martingale_radial} with $\rho_2=\cdots=\rho_n=2$.

\begin{proof}[Proof of Proposition~\ref{prop::finitetime_LDP_radial}]
It suffices to show the conclusion for $(\Omega;\bs{x};z)=(\U;\ee^{\ii \bs{\theta}};0)$ and we eliminate them from the notation. 
We simply write 
\[ \chamber_{\bs{t}}=\chamber_{\bs{t}}(\U;\ee^{\ii\bs{\theta}};0), \quad \rateradial{n}^{(\mu)}(\U;\ee^{\ii\bs{\theta}};0;\bs{\gamma}_{[\bs{0},\bs{t}]})=\rateradial{n}^{(\mu)}(\bs{\gamma}_{[\bs{0},\bs{t}]}). \]
For each $j\in \{1,\ldots, n\}$, let $\gamma^{j}$ be radial $\SLE_{\kappa}$ in $(\U; \ee^{\ii\theta_j}; 0)$ and let $\mathsf{P}_{n}^{(\kappa)}$ be the probability measure on $\bs{\gamma}=(\gamma^{1}, \ldots, \gamma^{n})$ under which the curves are independent. Let $\mathsf{E}_{n}^{(\kappa)}$ be the expectation under the probability measure $\mathsf{P}_{n}^{(\kappa)}$.
Note that $\chamber_{\bs{T}}$ is an open topological subspace of $\prod_{j=1}^{n} \chamber_{T_j}(\U;e^{\ii\theta_j};0)$. 
We have the following two observations. 
\begin{itemize}
	\item A direct consequence of Lemma~\ref{lem::finite_radialSLE_LDP} implies that the law of $\bs{\gamma}_{[\bs{0},\bs{T}]}$ under $\mathsf{P}_{n}^{(\kappa)}$ satisfies large deviation principle in the space \[\left( \prod_{j=1}^{n} \chamber_{T_j}(\U;\ee^{\ii\theta_j};0),\dist_{\chamber} \right)\] as $\kappa\to 0+$ with good rate function 
	\begin{equation}\label{eqn::LDP_finitetime_aux1}
 		\LJ_{\oplus n}(\bs{\gamma}_{[\bs{0},\bs{T}]}):=\frac{1}{2}\sum_{j=1}^n \int_0^{T_j}\left(\dot{\xi}^j_{t_j}\right)^2\ud t_j. 
	\end{equation}
 	\item For $\bs{\gamma}_{[\bs{0},\bs{T}]}\in {\chamber}_{\bs{T}}$, we define
\begin{align*} 
	\Phi(\bs{\gamma}_{[\bs{0},\bs{T}]}; \kappa)=&\kappa\log \frac{	M_{\bs{T}}(\LZradial{n}^{(\kappa;\mu)})}{	M_{\bs{0}}(\LZradial{n}^{(\kappa;\mu)})}\\
 	=&\underbrace{\frac{1}{2}(n^2-4-\mu^2) \log \mathfrak{g}_{\bs{T}}'(0) - \LVradial{n}^{(\mu)} \left(\bs{\mslitdriv}_{\bs{T}}\right) + \LVradial{n}^{(\mu)} (\bs{\theta}) + \frac{3}{2} \sum_{j=1}^n T_j}_{\Phi_0^{\mu}(\bs{\gamma}_{[\bs{0}, \bs{T}]}):=}\\
 	&+ \underbrace{ \kappa (\kappa-8) }_{f_1(\kappa):=} \underbrace{ \left( \sum_{j=1}^n T_j - \log \mathfrak{g}_{\bs{T}}'(0) \right) }_{\Phi_1(\bs{\gamma}_{[\bs{0}, \bs{T}]}):=}  \underbrace{ - \blm_{\bs{T}} }_{\Phi_2(\bs{\gamma}_{[\bs{0}, \bs{T}]}):=} \underbrace{ \frac{(6-\kappa)(8-3\kappa)}{4} }_{f_2(\kappa):=} \\
 	&+ \underbrace{ (3-\kappa/2) }_{f_3(\kappa):=} \underbrace{ \sum_{j=1}^n \log \covmap_{\bs{T},j}'\left(\xi_{T_j}^{j}\right) }_{\Phi_3(\bs{\gamma}_{[\bs{0}, \bs{T}]}):=} .
\end{align*}
By Lemma~\ref{lem::nradial_mart}, for any subset $A\subset \chamber_{\bs{T}}$, we have 
\[
\PPradial{n}^{(\kappa;\mu)}[A] = \mathsf{E}_{n}^{(\kappa)}\left[ \exp\left(\frac{1}{\kappa}\Phi(\bs{\gamma}_{[\bs{0},\bs{T}]}; \kappa)\right) \mathbb{1}\{\bs{\gamma}_{[\bs{0},\bs{T}]}\in A\}\right]. 
\]
\end{itemize}
To apply the generalization of Varadhan's Lemma (see Lemma~\ref{lem::Varadhan} and Remark~\ref{rmk::Varadhan}), we have the following observations.
\begin{itemize}
 	\item Assume $\left\{ \bs{\gamma}_{[\bs{0}, \bs{T}]}^{(n)} \right\}$ is a sequence in $(\chamber_{\bs{T}},\dist_{\chamber})$ converging to some $\bs{\gamma}_{[\bs{0}, \bs{T}]}^{(\infty)}\in \overline{\chamber}_{\bs{T}}$ and $\bs{\mslitdriv}_{[\bs{0},\bs{T}]}^{(n)}$ is the driving function of $\bs{\gamma}_{[\bs{0}, \bs{T}]}^{(n)}$ for $n\in \Z_+ \cup \{\infty\}$. Then by Carathéodory convergence theorem (see e.g.~\cite[Theorem~1.8]{Pommerenke}), the process $\bs{\gamma}_{[\bs{0}, \bs{T}]}^{(n)}$ converges to $\bs{\gamma}_{[\bs{0}, \bs{T}]}^{(\infty)}$ in the Carathéodory sense. By Loewner-Kufarev theorem (see e.g.~\cite[Theorem~8.5]{Berestycki2011LecturesOS}), the process $\bs{\mslitdriv}_{[\bs{0},\bs{T}]}^{(n)}$ converges to $\bs{\mslitdriv}_{[\bs{0},\bs{T}]}^{(\infty)}$ in the sense of uniform convergence. Hence $\Phi_0^{\mu}(\bs{\gamma}_{[\bs{0}, \bs{T}]})$ and $\Phi_j(\bs{\gamma}_{[\bs{0}, \bs{T}]})$ are continuous for $1\le j\le 3$ and the condition~(a) in Lemma~\ref{lem::Varadhan} holds.
 	\item By Koebe's quarter theorem we have
 	\begin{equation} \label{eqn::Varadhan_application_radial_aux2}
 		\max_{1\le j\le n} T_j \le \mathfrak{g}_{\bs{T}}'\left(0\right) \le \log 4+\max_{1\le j\le n} T_j.
 	\end{equation}
 	Thus $(f_1(\kappa),\Phi_1(\bs{\gamma}_{[\bs{0}, \bs{T}]}))$ satisfies condition~(b) in Lemma~\ref{lem::Varadhan}. 
	\item We have $\covmap_{\bs{T},j}'\left(\xi_{T_j}^{j}\right)\le 1$ and $\blm_{\bs{T}}\ge 0$. Thus $(f_2(\kappa),\Phi_2(\bs{\gamma}_{[\bs{0}, \bs{T}]}))$ and $(f_3(\kappa),\Phi_3(\bs{\gamma}_{[\bs{0}, \bs{T}]}))$ satisfy condition~(c) in Lemma~\ref{lem::Varadhan}.

 	\item We verify the condition~(d) in Lemma~\ref{lem::Varadhan} with $\beta=2$. We emphasize that we need to keep track of $\mu$ in our calculation. We have
 	\begin{align} \label{eqn::Varadhan_application_radial_aux3}
 		2 \Phi_0^{\mu}(\bs{\gamma}_{[\bs{0}, \bs{T}]}) = & (n^2-4-\mu^2)\log \mathfrak{g}_{\bs{T}}'(0) - 2\LVradial{n}^{(\mu)} \left(\bs{\mslitdriv}_{\bs{T}}\right) + 2 \LVradial{n}^{(\mu)}(\bs{\theta}) + 3 \sum_{j=1}^n T_j \notag\\
 		= & \Phi_0^{2\mu}(\bs{\gamma}_{[\bs{0}, \bs{T}]}) 
 		- \LVradial{n}^{(0)} \left(\bs{\mslitdriv}_{\bs{T}}\right) + \LVradial{n}^{(0)} (\bs{\theta}) + \frac{1}{2}(n^2-4+2\mu^2) \log \mathfrak{g}_{\bs{T}}'(0)+ \frac{3}{2} \sum_{j=1}^n T_j \notag\\
 		\le & \Phi_0^{2\mu}(\bs{\gamma}_{[\bs{0}, \bs{T}]}) + \LVradial{n}^{(0)}(\bs{\theta})+ \frac{1}{2}(n^2+3n-4+2\mu^2) \left(\log 4 + \max_{1\le j\le n} T_j\right),
 	\end{align}
 	where the inequality is due to~\eqref{eqn::Varadhan_application_radial_aux2}. Now, we are ready to check the condition~(d) in Lemma~\ref{lem::Varadhan} with $\beta=2$:
 	\begin{align}
 		& \limsup_{\kappa\to 0} \kappa \log \mathsf{E}_{n}^{(\kappa)} \left[ \exp \left( \frac{1}{\kappa} \left( 2 \Phi_0^{\mu}(\gamma)+\sum_{j=1}^{3} f_j(\kappa) \Phi_j(\gamma) \right) \right) \one\{\gamma\in X'\} \right] \notag \\
 		\le & \LVradial{n}^{(\mu)}(\bs{\theta})+ \frac{1}{2}(n^2+3n-4+2\mu^2) \left(\log 4 + \max_{1\le j\le n} T_j\right) \nonumber \\
		&+  \limsup_{\kappa\to 0} \kappa \log \mathsf{E}_{n}^{(\kappa)} \left[ M_{\bs{T}}(\LZradial{n}^{(\kappa;2\mu)})/M_{\bs{0}}(\LZradial{n}^{(\kappa;2\mu)}) \right] \tag{due to~\eqref{eqn::Varadhan_application_radial_aux3}} \\
 		= & \LVradial{n}^{(\mu)}(\bs{\theta})+ \frac{1}{2}(n^2+3n-4+2\mu^2) \left(\log 4 + \max_{1\le j\le n} T_j\right) <\infty, \notag
 	\end{align}
 	where we use the fact that $\mathsf{E}_{n}^{(\kappa)} \left[ M_{\bs{T}}(\LZradial{n}^{(\kappa;2\mu)})/M_{\bs{0}}(\LZradial{n}^{(\kappa;2\mu)}) \right]=\E_{n\mathrm{\textnormal{-}rad}}^{(\kappa;2\mu)}[1]=1$.
 	\item The function $\Phi(\bs{\gamma}_{[\bs{0},\bs{T}]};\kappa)$ is continuous in $\kappa$ and we have 
 	\begin{align*} \label{eqn::Def_Phi0_radial}
 	\begin{split}
 		\Phi(\bs{\gamma}_{[\bs{0},\bs{T}]}; 0)
 		=&\frac{n^2-4-\mu^2}{2} \log \mathfrak{g}_{\bs{T}}'(0) - \LVradial{n}^{(\mu)} \left(\bs{\mslitdriv}_{\bs{T}}\right) + \LVradial{n}^{(\mu)} (\bs{\theta}) \\
		&+ \frac{3}{2} \sum_{j=1}^n T_j  - 12\blm_{\bs{T}} + 3 \sum_{j=1}^n \log \covmap_{\bs{T},j}'\left(\xi_{T_j}^{j}\right).	
 	\end{split}
 	\end{align*}
\end{itemize}
Combing the observations above and applying Lemma~\ref{lem::Varadhan}, the law of $\bs{\gamma}_{[\bs{0},\bs{T}]}$ under $\PPradial{n}^{(\kappa;\mu)}$ satisfies large deviation principle in the space $(\chamber_{\bs{T}},\dist_{\chamber})$ as $\kappa\to 0+$ with rate function  
\[\LJ_{\oplus n}(\bs{\gamma}_{[\bs{0},\bs{T}]})-\Phi(\bs{\gamma}_{[\bs{0},\bs{T}]},0)\stackrel{\eqref{eqn::multi_time_energy_radial_def}}{=} \rateradial{n}^{(\mu)}(\bs{\gamma}_{[\bs{0},\bs{T}]}) \]
as desired.
\medbreak
It remains to show that $\rateradial{n}^{(\mu)}$ is a good rate function in the space $(\chamber_{\bs{T}}, \dist_{\chamber})$, i.e. for any $C>0$, the level set $\{\bs{\gamma}_{[\bs{0},\bs{T}]}: \rateradial{n}^{(\mu)}(\bs{\gamma}_{[\bs{0},\bs{T}]})\le C\}$ is compact. 
To this end, we compare the function $\rateradial{n}^{(\mu)}$ with the good rate function $\LJ_{\oplus n}$ defined in~\eqref{eqn::LDP_finitetime_aux1}. Combining~(\ref{eqn::finite_energy_simple_aux1},\ref{eqn::finite_energy_simple_aux2},\ref{eqn::finite_energy_simple_aux3}) and letting $t_2=\cdots=t_n=0$, we have
\begin{align*}
	\rateradial{n}^{(\mu)}(\bs{\gamma}_{[\bs{0}, \bs{T}]}) \ge & \frac{1}{2} \int_0^{T_1}\left( \dot{\xi}_{t_1}^{1} \right)^2 \ud t_1 - \mu \left( \xi_{T_1}^{1} - \theta_1 \right) + 2 \sum_{j=2}^n \log \left( \frac{\sin \left(\frac{\theta_j-\theta_1}{2}\right)}{\sin \left(\frac{\covmap_{T_1}^{1}(\theta_j) - \xi_{T_1}^{1}}{2}\right)} \right)  \\
	\ge & \frac{1}{2} \int_0^{T_1}\left( \dot{\xi}_{t_1}^{1} \right)^2 \ud t_1 - \mu \left( \xi_{T_1}^{1} - \theta_1 \right) + 2 \sum_{j=2}^n \log \left( \sin \left(\frac{\theta_j-\theta_1}{2}\right) \right) \\
	= & \frac{1}{4} \int_0^{T_1}\left( \dot{\xi}_{t_1}^{1} \right)^2 \ud t_1 + \frac{1}{4} \int_0^{T_1} \left( \left( \dot{\xi}_{t_1}^{1} \right)^2 - 4\mu \dot{\xi}_{t_1}^{1} \right) \ud t_1 + 2 \sum_{j=2}^n \log \left( \sin \left(\frac{\theta_j-\theta_1}{2}\right) \right) \notag \\
	\ge & \frac{1}{4} \int_0^{T_1}\left( \dot{\xi}_{t_1}^{1} \right)^2 \ud t_1 -T_1 \mu^2 + 2 \sum_{j=2}^n \log \left( \sin \left(\frac{\theta_j-\theta_1}{2}\right) \right),
 \end{align*}
which implies
\begin{equation} \label{eqn::LDP_finitetime_aux5}
 	\LJ_{\oplus n}(\bs{\gamma}_{[\bs{0}, \bs{T}]}) = \frac{1}{2} \sum_{j=1}^n \int_0^{T_j} \left( \dot{\xi}_{t_j}^j \right)^2 \ud t_j \le 2n \rateradial{n}^{(\mu)}(\bs{\gamma}_{[\bs{0}, \bs{T}]}) + 2 \mu^2 \sum_{j=1}^{n} T_j - 8 \sum_{1\le i<j\le n} \log \left(\sin \left(\frac{\theta_j-\theta_i}{2}\right)\right),
\end{equation}
by symmetry. Assume $\left\{ \bs{\gamma}_{[\bs{0},\bs{T}]}^{(k)} \right\}_{k\ge 1}\subset \chamber_{\bs{T}}$ is a sequence such that $\rateradial{n}^{(\mu)}(\bs{\gamma}_{[\bs{0},\bs{T}]}^{(k)})\le C$ for $k\ge 1$. Since $\LJ_{\oplus n}$ is a good rate function in the space $\left(\prod_{j=1}^{n} \chamber_{T_j}(\U;\ee^{\ii\theta_j};0),\dist_{\chamber}\right)$, combining with~\eqref{eqn::LDP_finitetime_aux5}, we can pass to a subsequence, still denoted by $\left\{ \bs{\gamma}_{[\bs{0},\bs{T}]}^{(k)} \right\}_{k\ge 1}$, which converges to some element $\bs{\gamma}_{[\bs{0},\bs{T}]}^{(\infty)}\in \overline{\chamber}_{\bs{T}}$. Recalling~\eqref{eqn::multi_time_energy_radial_def}, we have
\begin{align}\label{eqn::LDP_finitetime_aux3bis}
 	\rateradial{n}^{(\mu)}(\bs{\gamma}_{[\bs{0}, \bs{T}]})=\LJ_{\oplus n}(\bs{\gamma}_{[\bs{0}, \bs{T}]})
 	&+12\blm_{\bs{T}}-\frac{3}{2}\sum_{j=1}^n T_j-\frac{1}{2}(p^2-4-\mu^2)\log \mathfrak{g}'_{\bs{T}}(0)-3\sum_{j=1}^n \log\covmap'_{\bs{T}, j}\left(\xi^j_{T_j}\right)\notag\\
 	&+\LVradial{n}^{(\mu)}\left(\bs{\mslitdriv}_{\bs{T}}\right)-\LVradial{n}^{(\mu)}(\bs{\theta}). 
\end{align}
On the one hand, the good rate function $\LJ_{\oplus n}$ is lower semicontinuous with respect to $\bs{\gamma}_{[\bs{0},\bs{T}]}$. On the other hand, the other terms in the RHS of~\eqref{eqn::LDP_finitetime_aux3bis}, except the first term, are continuous with respect to $\bs{\gamma}$ due to Carathéodory convergence theorem (see e.g.~\cite[Theorem~1.8]{Pommerenke}) and Loewner–Kufarev theorem (see e.g.~\cite[Theorem~8.5]{Berestycki2011LecturesOS}). By~\eqref{eqn::LDP_finitetime_aux3bis}, we have
\begin{equation} \label{eqn::LDP_finitetime_aux9}
 	\rateradial{n}^{(\mu)}(\bs{\gamma}_{[\bs{0},\bs{T}]}^{(\infty)}) \le \liminf_{k\to \infty} \rateradial{n}^{(\mu)}(\bs{\gamma}_{[\bs{0},\bs{T}]}^{(k)}) \le C.
\end{equation}
Eq.~\eqref{eqn::LDP_finitetime_aux9} implies that $\bs{\gamma}_{[\bs{0},\bs{T}]}^{(\infty)}$ has finite $n$-time-energy and Lemma~\ref{lem::finite_energy_simple} implies that $\bs{\gamma}_{[\bs{0},\bs{T}]}^{(\infty)}\in \chamber_{\bs{T}}$. Thus the level set of $\rateradial{n}^{(\mu)}$ is compact and $\rateradial{n}^{(\mu)}$ is good.	
\end{proof}

\subsection{Large deviation in infinite time: proof of Theorem~\ref{thm::radialSLE_LDP}}
\label{subsec:::LDP_infinite_time_multiradial}
We will prove Theorem~\ref{thm::radialSLE_LDP} using Proposition~\ref{prop::finitetime_LDP_radial}, return estimate for radial $\SLE$ in Proposition~\ref{prop::radialSLE_return} and the contraction principle in Lemma~\ref{lem::generalized_contraction}.

Recall the definition of projective system introduced in Section~\ref{subsec::LDP_infinite_time_halfwatermelon}. We introduce the projective system that we will use. To simplify our notation, we write $\chamber_{\bs{t}}=\chamber_{\bs{t}}(\Omega;\bs{x};z)$. The collection $(\chamber_{\bs{t}})_{\bs{t}\in [\bs{0},\bs{\infty})}$ of $n$-time parameterized simple curves, equipped with the metric $\dist_{\chamber}$ in~\eqref{eqn::para_curve_metric}, forms a projective system with restriction maps 
\begin{align*} \label{eqn::restriction_maps}
	p_{\bs{s},\bs{t}}:&\chamber_{\bs{t}} \to \chamber_{\bs{s}} , \qquad \bs{\gamma}_{[\bs{0},\bs{t}]} \mapsto \bs{\gamma}_{[\bs{0},\bs{s}]}, \qquad \text{ for } \bs{s} \le \bs{t},
\end{align*}
which are continuous and satisfy $p_{\bs{s},\bs{r}}=p_{\bs{s},\bs{t}} \circ p_{\bs{t},\bs{r}}$ whenever $\bs{s}\le \bs{t}\le\bs{r}$. Its projective limit,
\[
\overleftarrow{\chamber}:=\varprojlim \chamber_{\bs{t}}=\left\{ \left( \bs{\gamma}_{[\bs{0},\bs{t}]} \right)_{\bs{t}\in [\bs{0},\bs{\infty})}: \bs{\gamma}_{[\bs{0},\bs{s}]}=p_{\bs{s},\bs{t}} \left( \bs{\gamma}_{[\bs{0},\bs{t}]} \right) \text{ for all } \bs{s}\le\bs{t} \right\} \subset \prod_{\bs{t}\in [\bs{0},\bs{\infty})} \chamber_{\bs{t}},
\]
is equipped with the topology induced by the product topological space $\prod_{\bs{t}\in [\bs{0},\bs{\infty})} \chamber_{\bs{t}}$. The projections are given by 
\begin{equation} \label{eqn::projection_radial}
	p_{\bs{s}}:\overleftarrow{\chamber} \to \chamber_{\bs{s}}, \qquad \left( \bs{\gamma}_{[\bs{0},\bs{t}]} \right)_{\bs{t}\in [\bs{0},\bs{\infty})} \mapsto \bs{\gamma}_{[\bs{0},\bs{s}]}.
\end{equation}

\begin{lemma} \label{lem::LDP_projection_radial}
Fix $\mu\in\R, n\ge 2$ and $n$-polygon $(\Omega;\bs{x})$ with $z\in \Omega$. The family $\{\PPradial{n}^{(\kappa; \mu)}(\Omega;\bs{x};z)\}_{\kappa\in (0,4]}$ of laws of $n$-radial $\SLE_{\kappa}^{\mu}$ satisfies large deviation principle in $\overleftarrow{\chamber}$ with the topology induced by the product topological space $\prod_{\bs{t}\in [\bs{0},\bs{\infty})} \chamber_{\bs{t}}$ as $\kappa\to 0+$ with good rate function $\rateradial{n}^{(\mu)}(\Omega;\bs{x};z;\cdot)$ in~\eqref{eqn::multi_time_energy_radial}. 
\end{lemma}
\begin{proof}
This is a combination of Proposition~\ref{prop::finitetime_LDP_radial} and Lemma~\ref{lem::Dawson-Gartner}.
\end{proof}

 Recall we equip $\chamberradial{n}(\U;\ee^{\ii\bs{\theta}}; 0)$ with the metric $\dist_{\chamber}$ defined in~\eqref{eqn::dist_chamber}. The projections~\eqref{eqn::projection_radial} induce a continuous bijection
\begin{equation*}
	\iota:\chamberradial{n}(\U;\ee^{\ii\bs{\theta}}; 0) \to \overleftarrow{\chamber}, \qquad \bs{\gamma} \mapsto \left( \bs{\gamma}_{[\bs{0},\bs{t}]} \right)_{\bs{t}\in [\bs{0},\bs{\infty})}.
\end{equation*}
However, the map $\iota$ is not a homeomorphism from $\chamberradial{n}(\U;\ee^{\ii\bs{\theta}}; 0)$ to $\overleftarrow{\chamber}$ (see~\cite[Remark~4.2]{AbuzaidPeltolaLargeDeviationCapacityParameterization}). To apply generalized contraction principle, we define
\begin{equation*}\label{eqn::exittime_radial}
	\tau_{p}^j:=\inf \left\{ t\ge 0: \left| \gamma_t^j \right|=\ee^{-p} \right\}, \qquad \overline{\U}_{p}=\ee^{-p} \overline{\U}, 
\end{equation*}
and we consider a smaller space:
\begin{equation*}
	G_{\overrightarrow{N}}:=\bigcap_{p=1}^{\infty} \bigcap_{j=1}^n  \left\{ \bs{\gamma}\in \chamberradial{n}(\U;\ee^{\ii\bs{\theta}}; 0): \gamma^j_{[\tau_{N_p}^j,\infty)} \subset \overline{\U}_p \right\},
\end{equation*}
for any sequence of integers $\overrightarrow{N}=(N_p)_{p\in \N}$ with $N_p>p$.

\begin{lemma}\label{lem::homeomorphisim_GN}
For any sequence of integers $\overrightarrow{N}=(N_p)_{p\in \N}$ with $N_p>p$, the continuous bijection
\begin{equation*}
	\iota|_{G_{\overrightarrow{N}}}: \chamberradial{n}(\U;\ee^{\ii\bs{\theta}}; 0) \to \overleftarrow{\chamber}, \qquad \bs{\gamma} \mapsto \left( \bs{\gamma}_{[\bs{0},\bs{t}]} \right)_{\bs{t}\in [\bs{0},\bs{\infty})},
\end{equation*}
is a homeomorphsim and the set $\iota\left(G_{\overrightarrow{N}}\right)$ is closed.
\end{lemma}
\begin{proof}
This can be proved by the same proof for Lemma~\ref{lem::homeomorphisim_FN}, with~\eqref{eqn::homeomorphism_aux1} replaced by the following estimate:
\[
\dist_{\chamber} \left(\bs{\gamma}^{(\ell)},\bs{\gamma}\right) \le \dist_{\chamber} \left( \bs{\gamma}_{[\bs{0},\bs{\tau}_{N_p}]}^{(\ell)}, \bs{\gamma}_{[\bs{0},\bs{\tau}_{N_p}]}  \right) +2\ee^{-p}.
\]
\end{proof}

\begin{proof}[Proof of Theorem~\ref{thm::radialSLE_LDP}]
This can be proved by the same proof for Theorem~\ref{thm::halfwatermelon_LDP}, with Proposition~\ref{prop::SLE_return_chordal} replaced by Proposition~\ref{prop::radialSLE_return}, Lemma~\ref{lem::LDP_projection} replaced by Lemma~\ref{lem::LDP_projection_radial}, and  Lemma~\ref{lem::homeomorphisim_FN} replaced by Lemma~\ref{lem::homeomorphisim_GN}.
\end{proof}

\subsection{Large deviation for radial SLE with spiral: proof of Proposition~\ref{prop::radialSLE_LDP}}
\label{subsec::radial_rho_LDP}
In this section, we prove Proposition~\ref{prop::radialSLE_LDP}. The proof is similar to that of Theorem~\ref{thm::radialSLE_LDP} in Sections~\ref{subsec:::LDP_finite_time_multiradial}-\ref{subsec:::LDP_infinite_time_multiradial}: we first prove the large deviation principle for finite-time curves (see Lemma~\ref{lem::rhoLR_finite_LDP_radial}), and then extend it to infinite-time curves using the return estimate in Proposition~\ref{prop::radialSLE_return}.

\medbreak 

The following lemma gives an equivalent definition of the energy $\rateradial{1}^{(\bs{\rho};\mu)}(\Omega;\bs{x};z;\gamma_{[0,T]})$ defined in \eqref{eqn::rho_energy_radial}. It will be used in the proof of Proposition~\ref{prop::radialSLE_LDP}. 

\begin{lemma} \label{lem::rho_energy_radial_expansion}
Assume the same notation as in Definition~\ref{def::rho_energy_radial}. For $\gamma\in \overline{\chamber}(\Omega;x_1;z)$, we have
\begin{align}\label{eqn::rho_energy_radial_expansion}
	&\rateradial{1}^{(\bs{\rho};\mu)}(\Omega;\bs{x};z;\gamma_{[0,T]}) \nonumber\\
	 = & \frac{1}{2} \int_{0}^{T} \left(\dot{\xi}_t \right)^2 \ud t - \Bigg( \frac{\bar{\rho}(\bar{\rho}+4)-4\mu^2}{8} T + \sum_{j=2}^{n} \frac{\rho_j(\rho_j+4)}{4} \log \covmap'_T(\theta_j)+ \mu \xi_T + \sum_{j=2}^{n} \frac{\mu \rho_j}{2} (\covmap_T(\theta_j) - \theta_j) \notag \\
	&+ \sum_{j=2}^{n} \rho_j \log \left| \frac{\sin \left( \frac{\covmap_T(\theta_j)-\xi_T}{2} \right)}{\sin \left( \frac{\theta_j-\theta_1}{2} \right)}\right| + \sum_{2\le j<\ell\le n} \frac{\rho_j \rho_{\ell}}{2} \log \left| \frac{\sin \left( \frac{\covmap_T(\theta_\ell)-\covmap_T(\theta_j)}{2} \right)}{\sin \left( \frac{\theta_\ell-\theta_j}{2} \right)}\right| \Bigg).
\end{align}
\end{lemma}

\begin{proof}
Expanding~\eqref{eqn::rho_energy_radial_noAC}, we obtain
\begin{align} \label{eqn::rho_energy_radial_expansion_aux1}
\begin{split}
	\rateradial{1}^{(\bs{\rho};\mu)}(\Omega;\bs{x};z;\gamma_{[0,T]}) = & \frac{1}{2} \int_{0}^{T} \left( \dot{\xi}_t \right)^2 \ud t + \frac{1}{2} \mu^2 T  - \mu \xi_T + \sum_{j=2}^{n} \frac{\rho_j^2}{8} \int_{0}^{T}  \cot \left( \frac{\xi_t-\covmap_t(\theta_j)}{2} \right)^2 \ud t \\
	& + \sum_{2\le j<\ell \le n} \frac{\rho_j \rho_\ell}{4} \int_{0}^{T} \cot \left( \frac{\xi_t-\covmap_t(\theta_j)}{2} \right)  \cot \left( \frac{\xi_t-\covmap_t(\theta_\ell)}{2} \right) \ud t \\
	& + \sum_{j=2}^{n} \frac{\mu \rho_j}{2} \int_{0}^{T} \cot \left( \frac{\xi_t - \covmap_t(\theta_j)}{2} \right) \ud t - \sum_{j=2}^{n} \frac{\rho_j}{2} \int_{0}^{T} \dot{\xi}_t \cot \left( \frac{\xi_t - \covmap_t(\theta_j)}{2} \right) \ud t.	
\end{split}
\end{align}
Standard calculations of the radial Loewner equation gives 
\begin{equation*}
\begin{split}	
	&\partial_t \log \sin \left( \frac{\covmap_t(\theta_j)-\xi_t}{2} \right) = \frac{1}{2} \cot \left( \frac{\covmap_t(\theta_j)-\xi_t}{2} \right)^2 - \frac{1}{2} \dot{\xi}_t \cot \left( \frac{\covmap_t(\theta_j)-\xi_t}{2} \right), \quad \partial_t \covmap_{t} (\theta_j) = \cot \left( \frac{\covmap_t(\theta_j)-\xi_t}{2} \right), \\
	&\partial_t \log \sin \left( \frac{\covmap_t(\theta_\ell)-\covmap_t(\theta_j)}{2} \right) = - \frac{\cos \left( \frac{\covmap_t(\theta_\ell)-\covmap_t(\theta_j)}{2} \right)}{2\sin \left( \frac{\covmap_t(\theta_\ell)-\xi_t}{2} \right) \sin \left( \frac{\covmap_t(\theta_j)-\xi_t}{2} \right)}, \quad \partial_t \log \covmap'_{t} (\theta_j) = - \frac{1}{2} \csc \left( \frac{\covmap_t(\theta_j)-\xi_t}{2} \right)^2.
\end{split}	
\end{equation*}
From these, we have
\begin{equation*}
	\begin{split}
		\cot^2\left( \frac{\covmap_t(\theta_j)-\xi_t}{2} \right) =& -1-2 \partial_t \log h_t'(\theta_j),\\
		\dot{\xi}_t \cot \left(\frac{h_t\left(\theta_j\right)-\xi_t}{2}\right)=&-1-2 \partial_t \log h_t^{\prime}\left(\theta_j\right)-2 \partial_t \log \sin \left(\frac{h_t\left(\theta_j\right)-\xi_t}{2}\right), \\
		\cot \left(\frac{h_t\left(\theta_{\ell}\right)-\xi_t}{2}\right) \cot \left(\frac{h_t\left(\theta_j\right)-\xi_t}{2}\right)=&-1-2 \partial_t \log \sin \left(\frac{h_t\left(\theta_{\ell}\right)-h_t\left(\theta_j\right)}{2}\right).
	\end{split}
\end{equation*}
Plugging these into~\eqref{eqn::rho_energy_radial_expansion_aux1}, we obtain~\eqref{eqn::rho_energy_radial_expansion}.
\end{proof}

\begin{lemma}	\label{lem::rhoLR_finite_LDP_radial}
Fix $n\ge 1$ and $n$-polygon $(\Omega;\bs{x})$ with $z\in \Omega$. Fix $\mu\in \R$ and $\bs{\rho}\in\R_{\ge 0}^{n-1}$ and $t>0$. 
Suppose $\gamma\sim\PPradial{1}^{(\kappa;\bs{\rho};\mu)}(\Omega; \bs{x}; z)$ is radial $\SLE_{\kappa}^{\mu}(\bs{\rho})$ in $(\Omega; \bs{x}; z)$. 
Then the family of laws of $\gamma_{[0,t]}$ under $\PPradial{1}^{(\kappa;\bs{\rho};\mu)}(\Omega; \bs{x}; z)$ satisfies large deviation principle in the space $(\chamber_t(\Omega;x_1;z),\dist_{\chamber})$ as $\kappa\to 0+$ with good rate function $\rateradial{1}^{(\bs{\rho};\mu)}(\Omega;\bs{x};z;\cdot)$ defined in~\eqref{eqn::rho_energy_radial_noAC}.
\end{lemma}

\begin{proof}
It suffices to show the conclusion for $(\Omega;\bs{x};z)=(\U;\ee^{\ii \bs{\theta}};0)$ and we eliminate them from the notation. 
Recall that the law of radial $\SLE_{\kappa}^{\mu}(\bs{\rho})$ in $(\U;\ee^{\ii\bs{\theta}};0)$ is the same as the law of radial $\SLE_{\kappa}$ in $(\U;\ee^{\ii\theta_1};0)$ tilted by $M_t$ given in~\eqref{eqn::SLEkapparho_martingale_radial}, up to the first time $\ee^{\ii\theta_2}$ or $\ee^{\ii\theta_n}$ is swallowed. For $\gamma_{[0,t]}\in \overline{\chamber}_t(\U;\ee^{\ii\theta_1};0)$, we define
\begin{align*}
	\Phi(\gamma_{[0,t]};\kappa):=& \kappa \log\frac{M_t}{M_0} =\frac{\bar{\rho}(\bar{\rho}+4)-4\mu^2}{8} t + \mu \xi_t + \sum_{j=2}^{n} \frac{\mu \rho_j}{2} (\covmap_t(\theta_j)  - \theta_j) + \sum_{j=2}^n \frac{\rho_j(\rho_j+4-\kappa)}{4} \log \covmap_t'(\theta_j) \\
	&+\sum_{j=2}^{n} \rho_j \log \left| \frac{\sin \left( \frac{\covmap_t(\theta_j)-\xi_t}{2} \right)}{\sin \left( \frac{\theta_j-\theta_1}{2} \right)}\right|
	 + \sum_{2\le j<\ell\le n} \frac{\rho_j \rho_{\ell}}{2} \log \left| \frac{\sin \left( \frac{\covmap_t(\theta_\ell)-\covmap_t(\theta_j)}{2} \right)}{\sin \left( \frac{\theta_\ell-\theta_j}{2} \right)}\right|.
\end{align*}
By Lemma~\ref{lem::SLEkapparho_martingale_radial}, for any subset $A\subset \chamber_{t}(\U;\ee^{\ii\theta_1};0)$, we have
\[
\PPradial{1}^{(\kappa;\bs{\rho};\mu)}[A] = \E^{(\kappa)} \left[ \exp\left( \frac{1}{\kappa} \Phi(\gamma_{[0,t]};\kappa) \right) \one\{\gamma_{[0,t]}\in A\} \right],
\]
where $\E^{(\kappa)}$ is the expectation with respect to the law of radial $\SLE_{\kappa}$ in $(\U;\ee^{\ii\theta_1};0)$. Applying the generalization of Varadhan's lemma (see Lemma~\ref{lem::Varadhan} and Remark~\ref{rmk::Varadhan}), we conclude that the family of laws of $\gamma_{[0,t]}$ under $ \PPradial{1}^{(\kappa;\bs{\rho};\mu)}$ satisfies large deviation principle in the space $(\chamber_t(\Omega;\bs{x};z),\dist_{\chamber})$ with good rate function
\begin{equation*}
	\rateradial{1}(\gamma_{[0,t]}) - \Phi(\gamma_{[0,t]};0) \stackrel{\eqref{eqn::rho_energy_radial_expansion}}{=} \rateradial{1}^{(\bs{\rho};\mu)}(\gamma_{[0,t]}),
\end{equation*}
with the same argument as in the proof of Proposition~\ref{prop::finitetime_LDP_radial}.
\end{proof}

\begin{proof}[Proof of Proposition~\ref{prop::radialSLE_LDP}]
It suffices to show the conclusion for $(\Omega;\bs{x};z)=(\U;\ee^{\ii \bs{\theta}};0)$ and we eliminate them from the notation. By Lemma~\ref{lem::rhoLR_finite_LDP_radial} and Lemma~\ref{lem::Dawson-Gartner}, we obtain large deviation principle of $\PPradial{1}^{(\kappa;\bs{\rho};\mu)}$ on the projective limit $\overleftarrow{\chamber}$ with the topology induced by the product topological space $\prod_{t\in [0,\infty)} \chamber_{t}$ as $\kappa\to 0+$ with good rate function $\rateradial{1}^{(\bs{\rho};\mu)}(\Omega;\bs{x};z;\gamma)$ in~\eqref{eqn::rho_energy_radial}. Combining this with Lemma~\ref{prop::radialSLE_return}, we complete the proof with the same argument as in the proof of Theorem~\ref{thm::radialSLE_LDP}.
\end{proof}

\subsection{Common-time parameter and proof of Proposition~\ref{prop::LDP_DysonBM_radial}}
\label{subsec::common_time_radial}
\paragraph*{Common-time parameter (radial).}
Fix $n\ge 1$ and $\bs{\theta}=(\theta_1, \ldots, \theta_n)\in \LX_n^{\U}$. Consider an $n$-tuple $\bs{\gamma}_{[\bs{0},\bs{t}]} := ( \gamma^{1}_{[0,t_1]}, \ldots, \gamma^{n}_{[0,t_n]} ) \in \chamber_{\bs{t}}(\U;\ee^{\ii\bs{\theta}};0)$ parameterized by $\bs{t}=(t_1,\ldots,t_n)\in [0,\infty)^n$.
Recall that $\mathfrak{g}_{\bs{t}}$ is the conformal map from $\smash{\U \setminus \cup_{j=1}^n \gamma^{j}_{[0,t_j]}}$ onto $\U$ normalized at the origin, i.e., 
\[
	\mathfrak{g}_{\bs{t}}(0)=0, \quad \log \mathfrak{g}_{\bs{t}}'(0)=\aleph_{\bs{t}}>0.
\] 
A standard calculation gives
\begin{equation*}
	\ud \aleph_{\bs{t}}=\sum_{j=1}^{n} \covmap'_{\bs{t},j} ( \xi_{t_j}^{j} )^2 \ud t_j.
\end{equation*}
We say that the $n$-tuple $\bs{\gamma}_{[\bs{0},\bs{t}]}$ is parameterized by common-time parameter if
\begin{equation*}
	\partial_{t_j} \aleph_{(t,\ldots,t)}=1, \quad \text{for } 1\le j\le n.
\end{equation*}
It is proved in~\cite[Lemma~3.2]{HealeyLawlerNSidedRadialSLE} that common-time parameter exists for continuous disjoint simple curves $\bs{\gamma}_{[\bs{0},\bs{t}]}$. The relation between common-time parameter and multi-time parameter is given by
\begin{equation} \label{eqn::common-multi}
	\ud t_j = \covmap_{\bs{t},j}'(W_{t_j}^j)^{-2} \ud t, \quad \text{ for } 1\le j\le n,
\end{equation}
Moreover,
\begin{equation*}
	t\le t_j(t) \le nt, \quad \text{for } 1\le j\le n.
\end{equation*}
Under common-time parameter, we denote the $n$-tuple by $\bs{\gamma}_{[0,t]}=\bs{\gamma}_{[(0,\ldots,0),(t,\ldots,t)]}$. We define the following normalized conformal transformations:
\begin{itemize}
	\item $\mathfrak{g}_t^j$ is the conformal map from $\U\setminus \gamma_{[0,t]}^j$ onto $\U$ with $\mathfrak{g}_t^j(0)=0$ and $\left(\mathfrak{g}_t^j \right)'(0)>0$, where $1\le j\le n$.
	\item $\mathfrak{g}_{t,j}$ is the conformal map from $\U\setminus \mathfrak{g}_{t}^j (\cup_{i \neq j} \gamma_{[0,t]}^i)$ onto $\U$ with $\mathfrak{g}_{t,j}(0)=0$ and $\mathfrak{g}'_{t,j}(0)>0$, where $1\le j\le n$.
	\item $\mathfrak{g}_t$ is the conformal map from $\U\setminus \cup_{j=1}^n \gamma_{[0,t]}^j$ onto $\U$ with $\mathfrak{g}_t(0)=0$ and $\mathfrak{g}'_t(0)=\ee^{nt}>0$, where $1\le j\le n$. 
\end{itemize}
Let $\covmap_{t}^{j},\covmap_{t},\covmap_{t,j}$ be the covering maps of $\mathfrak{g}_{t}^{j},\mathfrak{g}_{t},\mathfrak{g}_{t,j}$, respectively. Denote by $\xi^j$ the driving function of $\gamma^j$ and denote the driving function of the $n$-tuple $\bs{\gamma}_{[0,t]}$, started from $\bs{\mslitdriv}_0=(\theta_1,\ldots,\theta_n)\in\LX_n^{\U}$, by
\begin{equation*}
	\bs{\mslitdriv}_t=(\mslitdriv_t^1,\ldots,\mslitdriv_t^n), \quad \text{with } \mslitdriv_t^j=\covmap_{t,j}\left(\xi_t^j\right), \quad\text{for } 1\le j\le n.
\end{equation*}

\begin{lemma} \label{lem::DysonBM_driving_function_radial}
Under common-time parameter, the driving function $\bs{\mslitdriv}_t$ of $n$-radial $\SLE_{\kappa}^{\mu}$ satisfies the SDE of Dyson circular ensemble with parameter $\beta = 8/\kappa$ in~\eqref{eqn::DysonBM_radial}.	
\end{lemma}
\begin{proof}
Under multi-time parameter, Lemma~\ref{lem::nradial_mart} and~\cite[Eq.~(2.9)]{HuangPeltolaWuMultiradialSLEResamplingBP} (see also~\cite[Eq.~(3.14,3.15)]{KrusellWangWuCommutationRelation}) imply that the driving function $\bs{\mslitdriv}_{\bs{t}}$ of $n$-radial $\SLE_{\kappa}^{\mu}$ satisfies that, for $1\le j\le n$,
\begin{equation} \label{eqn::DysonBM_driving_function_radial_aux1}
	\ud \mslitdriv_{\bs{t}}^j = \sqrt{\kappa} \covmap'_{\bs{t}, j} (\xi_{t_j}^j) \ud B_t^j + \mu \covmap'_{\bs{t}, j} (\xi_{t_j}^j)^2 \ud t_j + \sum_{i \neq j} \cot\left( \frac{\mslitdriv_{\bs{t}}^j- \mslitdriv_{\bs{t}}^i}{2} \right) \left( \covmap'_{\bs{t}, j} (\xi_{t_j}^j)^2 \ud t_j + \covmap'_{\bs{t}, i} (\xi_{t_i}^i)^2 \ud t_i \right).
\end{equation}
Changing multi-time parameter into common-time parameter, plugging~\eqref{eqn::common-multi} into~\eqref{eqn::DysonBM_driving_function_radial_aux1}, we obtain~\eqref{eqn::DysonBM_radial} as desired.
\end{proof}

\begin{proof}[Proof of Proposition~\ref{prop::LDP_DysonBM_radial}]
From~\eqref{eqn::expansion_mslitdriv} and~\eqref{eqn::multi_time_energy_radial_def2}, we have
\begin{align*}
	\ud \rateradial{n}^{(\mu)}(\U;\ee^{\ii\bs{\theta}};0;\bs{\gamma}_{[\bs{0},\bs{t}]})=&\frac{1}{2}\sum_{j=1}^{n} \bigg( \dot{\xi}_{t_j}^{j} - \covmap'_{\bs{t},j}(\xi_{t_j}^{j}) \bigg( \mu+\sum_{i\neq j} \cot \bigg(\frac{\mslitdriv_{\bs{t}}^{j} - \mslitdriv_{\bs{t}}^{i}}{2}\bigg) \bigg) 
	- 3 \, \frac{\covmap''_{\bs{t},j}(\xi_{t_j}^{j})}{\covmap'_{\bs{t},j}(\xi_{t_j}^{j})}\bigg)^2 \ud t_j\\
	=&\frac{1}{2}\sum_{j=1}^{n} \left( \left(\covmap'_{\bs{t},j}\left(\xi_{t_j}^j\right)\right)^{-2}\frac{\partial}{\partial t_j}\mslitdriv_{\bs{t}}^{j}-  \mu-\sum_{i\neq j} \cot\left(\frac{\mslitdriv_{\bs{t}}^{j}-\mslitdriv_{\bs{t}}^{i}}{2}\right)\right)^2 \left(\covmap'_{\bs{t},j}\left(\xi_{t_j}^j\right)\right)^{2}\ud t_j. 
\end{align*}
Using the relation between common-time parameter and multi-time parameter~\eqref{eqn::common-multi}, we obtain from Proposition~\ref{prop::finitetime_LDP_radial} that under common-time parameter, for any finite time $T$, the family of laws of $\bs{\gamma}_{[0,T]}$ under $\PPradial{n}(\U;\ee^{\ii \bs{\theta}};0)$ satisfies large deviation principle in the space $(\chamber_{T}(\U;\ee^{\ii\bs{\theta}};0;\dist_{\chamber})$ as $\kappa\to 0+$ with good rate function
\begin{equation*}
	\rateradial{n}^{(\mu)}(\U;\ee^{\ii\bs{\theta}};0;\bs{\gamma}_{[0,t]})=\frac{1}{2}\int_{0}^{T}\sum_{j=1}^{n} \left( \dot{\mslitdriv}_t^j-  \mu- 2\sum_{i\neq j} \cot\left(\frac{\mslitdriv_{t}^j-\mslitdriv_{t}^{i}}{2}\right)  \right)^2 \ud t.
\end{equation*}
Here $\chamber_{T}$ is short for $\chamber_{(T,\ldots,T)}$ under common-time parameter. Finally,  by Loewner-Kufarev theorem (see e.g.~\cite[Theorem~8.5]{Berestycki2011LecturesOS}), the map $\bs{\gamma}_{[0,T]}\mapsto\bs{\mslitdriv}_{[0,T]}$ from $(\chamber_{T}(\U;\ee^{\ii\bs{\theta}};0), \dist_\chamber)$ to $\mathrm{C}([0,T],\LX_n^{\U})$ is continuous. Thus, by the contraction principle and Lemma~\ref{lem::DysonBM_driving_function_radial}, the family of laws of Dyson circular ensemble $\{\bs{\mslitdriv}^\kappa\}_{\kappa\in (0,4]}$ satisfies large deviation principle in the space $\mathrm{C}([0,T], \LX_n^{\U})$ as $\kappa\to 0+$ with good rate function~\eqref{eqn::rate_DysonBM_radial} as desired.
\end{proof}

\subsection{Boundary perturbation: proof of Proposition~\ref{prop::bp_radial}}
\label{subsec::bp_radial_proof}
In this section, we prove Proposition~\ref{prop::bp_radial}. We first recall the boundary perturbation of $n$-radial $\SLE_{\kappa}$ in Lemma~\ref{lem::bp_multiradialkappa}, and then combine it with Theorem~\ref{thm::radialSLE_LDP} to prove Proposition~\ref{prop::bp_radial}.
The boundary perturbation of radial $\SLE_{\kappa}$ is originally established in~\cite{JahangoshahiLawlerMultiplepathsSLE}, and later extended to multi-radial $\SLE_{\kappa}$ with spiral in~\cite{HuangPeltolaWuMultiradialSLEResamplingBP}.

\begin{lemma}[{\cite[Proposition~1.4]{HuangPeltolaWuMultiradialSLEResamplingBP}}] 
\label{lem::bp_multiradialkappa}
Fix $\kappa\in (0,4]$, $\mu\in\R$, $n\ge 1$ and nice $n$-polygon $(\Omega; \bs{x})=(\Omega; x_1, \ldots, x_n)$ with $z\in\Omega$.
Let $U\subset\Omega$ be a simply connected subdomain which coincides with $\Omega$ in a neighborhood of $\{x_1, \ldots, x_n, z\}$. Then, the $n$-radial $\SLE_{\kappa}^{\mu}$ probability measure 
in the smaller polygon $(U; \bs{x}; z)$ is absolutely continuous with respect to that in $(\Omega; \bs{x}; z)$, 
with Radon-Nikodym derivative 
\begin{align*}\label{eqn:multiradial_bp}
	\frac{\ud \PPradial{n}^{(\kappa;\mu)} (U;\bs{x};z) }{\ud \PPradial{n}^{(\kappa;\mu)} (\Omega;\bs{x}; z) }(\bs{\gamma})=\frac{\LZradial{n}^{(\kappa;\mu)}(\Omega; \bs{x};z)}{\LZradial{n}^{(\kappa;\mu)}(U; \bs{x};z)} \, \one\{\bs{\gamma}\cap (\Omega\setminus U)=\emptyset\}\, \exp\Big(\frac{\mathfrak{c}}{2} \blm(\Omega; \bs{\gamma}, \Omega\setminus U)\Big) ,
\end{align*} 
where $\LZradial{n}^{(\kappa;\mu)}$ is partition function in~\eqref{eqn::nradial_pf}, and the constant
$\mathfrak{c}$ is the central charge defined in~\eqref{eqn::parameters_b_c}, and $\blm(\Omega; \bs{\gamma}, \Omega\setminus U)$ is the Brownian loop measure in $\Omega$ of those loops that intersect both $\bs{\gamma}$ and $\Omega\setminus U$ (see~\eqref{eqn::blm_def}). 
\end{lemma}

\begin{proof}[Proof of Proposition~\ref{prop::bp_radial}]
Comparing~\eqref{eqn::LVradial_def} and~\eqref{eqn::nradial_pf}, we have 
\begin{equation}\label{eqn::bp_radial_aux0}
\kappa\log\LZradial{n}^{(\kappa;\mu)}(\Omega; \bs{x};z)=-\LVradial{n}^{(\mu)}(\Omega; \bs{x};z)-\frac{\kappa}{2}\sum_{j=1}^n\log\Poisson(\Omega; x_j; z)+\frac{\kappa(\kappa-8)}{8}\log\CR(\Omega; z). 
\end{equation}
Note that $\chamberradial{n}(U;\bs{x};z)$ is an open subset of $\chamberradial{n}(\Omega;\bs{x};z)$. 
For $\bs{\gamma}\in \chamberradial{n}(U;\bs{x};z)$, we define
\begin{align*}\Phi(\bs{\gamma};\kappa) :=&\kappa \log \frac{\ud \PPradial{n}^{(\kappa;\mu)}(U; \bs{x}; z) }{\ud \PPradial{n}^{(\kappa;\mu)}(\Omega; \bs{x}; z) }(\bs{\gamma}) \\
=& \underbrace{-\LVradial{n}^{(\mu)}(\Omega; \bs{x}; z) + \LVradial{n}^{(\mu)}(U; \bs{x}; z)}_{\Phi_0(\bs{\gamma}):=} \underbrace{- \blm(\Omega; \bs{\gamma}, \Omega\setminus U)}_{\Phi_1(\bs{\gamma}):=} \underbrace{\frac{(6-\kappa)(8-3\kappa)}{4}}_{f_1(\kappa):=} \\
	& \underbrace{-\frac{\kappa}{2}}_{f_2(\kappa):=} \underbrace{\sum_{1\le \ell\le n} \left( \log\Poisson(\Omega; x_{\ell}, z) - \log\Poisson(U; x_{\ell}, z) \right)}_{\Phi_2(\bs{\gamma}):=}+\underbrace{\frac{\kappa(\kappa-8)}{8}}_{f_3(\kappa)}\underbrace{\left(\log\CR(\Omega; z)-\log\CR(U; z)\right)}_{\Phi_3(\gamma)},
\end{align*}
where the second equation is due to~\eqref{eqn::bp_radial_aux0}. 
For any subset $A\subset \chamberradial{n}(U;\bs{x};z)$, 
\begin{equation}\label{eqn::bp_radial_aux1}
	\PPradial{n}^{(\kappa;\mu)}(U;\bs{x};z)[A]= \Eradial{n}^{(\kappa;\mu)} (\Omega;\bs{x};z)\left[ \exp\left( \frac{1}{\kappa} \Phi(\bs{\gamma};\kappa) \right) \one\{\bs{\gamma}\in A\}\right].
\end{equation}
To apply Lemma~\ref{lem::Varadhan} and Remark~\ref{rmk::Varadhan}, note that $\Phi_0(\bs{\gamma}),\Phi_2(\bs{\gamma}), \Phi_3(\bs{\gamma})$ are constants and $\Phi_1(\bs{\gamma})$ is continuous. Thus conditions~(a) and (d') of Lemma~\ref{lem::Varadhan} holds; $(f_1(\kappa),\Phi_1(\bs{\gamma}))$ satisfies condition~(c) of Lemma~\ref{lem::Varadhan}; both $(f_2(\kappa),\Phi_2(\bs{\gamma}))$ and $(f_3(\kappa), \Phi_3(\bs{\gamma}))$ satisfy condition~(b) of Lemma~\ref{lem::Varadhan}. 
Combining with Theorem~\ref{thm::radialSLE_LDP},  the RHS of~\eqref{eqn::bp_radial_aux1} satisfies large deviation with rate function
\[\rateradial{n}^{(\mu)}(\Omega;\bs{x};z;\bs{\gamma})-\lim_{\kappa\to 0+} \Phi(\bs{\gamma};\kappa)= \rateradial{n}^{(\mu)}(\Omega;\bs{x};z;\bs{\gamma})+ \LVradial{n}^{(\mu)}(\Omega; \bs{x}; z) - \LVradial{n}^{(\mu)}(U; \bs{x}; z) + 12 \blm(\Omega;\bs{\gamma},\Omega\setminus U). \] 
Thus we obtain
\begin{equation*}
	\rateradial{n}^{(\mu)}(U;\bs{x};z;\bs{\gamma})=\rateradial{n}^{(\mu)}(\Omega;\bs{x};z;\bs{\gamma})+ \LVradial{n}^{(\mu)}(\Omega; \bs{x}; z) - \LVradial{n}^{(\mu)}(U; \bs{x}; z) + 12 \blm(\Omega;\bs{\gamma},\Omega\setminus U),
\end{equation*}
which gives~\eqref{eqn::bp_radial} as desired.
\end{proof}

\appendix
\section{Generalized Varadhan's lemma}
\label{appendix::Varadhan}
In this section, we prove a generalization of Varadhan's lemma. This result is instrumental in establishing large deviation principles for processes constructed by weighting a reference process with a $\kappa$-dependent local martingale under mild conditions. We apply this result to the chordal case in Propositions~\ref{prop::finite_watermelon_LDP},~\ref{prop::LDP_SLEkapparho}, and~\ref{prop::bp_chordal}, and to the radial case in Propositions~\ref{prop::finitetime_LDP_radial},~\ref{prop::radialSLE_LDP}, and~\ref{prop::bp_radial}.
\medbreak
First, we recall the Varadhan's lemma (see e.g.~\cite[Lemma~4.3.4 and Lemma~4.3.6]{DemboZeitouniLargeDeviations})
\begin{lemma}[Varadhan's lemma] \label{lem::initial_Varadhan}
Suppose that the probability measures $\left\{ \PP^{(\kappa)} \right\}_{\kappa>0}$ satisfies the large deviation principle in the Hausdorff topology space $X$ with good rate function $\LI$. Let $\Phi:X\to \R$ be a continuous function.
Then for any open subset $O\subset X$, we have 
\begin{equation*} \label{eqn::initial_Varadhan_O}
	\limsup_{\kappa\to 0} \kappa \log \E^{(\kappa)} \left[ \exp \left( \frac{1}{\kappa} \Phi(\eta) \right) \one \left\{ \eta\in O \right\} \right] \ge -\inf_{\eta \in O} \left( \LI(\eta)-\Phi(\eta) \right)
\end{equation*}
Furthermore, if $\Phi$ satisfies the following moment condition for some $\beta>1$:
\begin{equation} \label{eqn::initial_moment_cond}
	\limsup_{\kappa\to 0} \kappa \log \E^{(\kappa)} \left[ \exp \left( \frac{\beta}{\kappa} \Phi(\eta) \right) \right]<\infty.
\end{equation}
Then for any closed subset $F\subset X$, we have 
\begin{equation*} \label{eqn::initial_Varadhan_F} 
	\limsup_{\kappa\to 0} \kappa \log \E^{(\kappa)} \left[ \exp \left( \frac{1}{\kappa} \Phi(\eta) \right) \one \left\{ \eta\in F \right\} \right] \le -\inf_{\eta \in F} \left( \LI(\eta)-\Phi(\eta) \right).
\end{equation*}
\end{lemma}

\begin{lemma}[Generalized Varadhan's lemma] \label{lem::Varadhan}
Suppose that the probability measures $\left\{ \PP^{(\kappa)} \right\}_{\kappa>0}$ satisfies the large deviation principle in the Hausdorff topology space $X$ with good rate function $\LI$. Suppose $X'\subset X$ is an open topological subspace of $X$. 
Suppose 
\begin{equation} \label{eqn::Def_Phi_gamma_kappa}
	\Phi(\eta; \kappa)=\Phi_0(\eta)+f_1(\kappa)\Phi_1(\eta)+f_2(\kappa)\Phi_2(\eta)
\end{equation}
satisfying the following conditions (a,b,c,d) or (a,b,c,d'): 
\begin{enumerate}
	\item[(a)] $\Phi_0$ is continuous on $X'$; 
	\item[(b)] $\Phi_1$ is continuous on $X'$ and there exists a constant $A_1>0$ such that $|\Phi_1(\eta)|\le A_1$; $f_1$ is continuous on $[0,\kappa_0]$ and $f_1(0)=0$; we denote $A'_1:=\sup_{\kappa\in [0,\kappa_0]} |f_1(\kappa)|$;
	\item[(c)] $\Phi_2$ is continuous on $X'$ and $\Phi_2(\eta)\le 0$ for $\eta\in X'$; $f_2$ is continuous and decreasing on $[0,\kappa_0]$ with $f_2(0)>0$;
	\item[(d)] there exists $\beta>1$ such that: 
	\begin{equation*} \label{eqn::moment_cond}
		\limsup_{\kappa\to 0} \kappa \log \E^{(\kappa)} \left[ \exp \left( \frac{1}{\kappa} \left( \beta \Phi_0(\eta)+f_1(\kappa)\Phi_1(\eta)+f_2(\kappa)\Phi_2(\eta) \right) \right) \one\{\eta\in X'\} \right]<\infty.
	\end{equation*}
	\item[(d')] there exists a constant $A_0>0$ such that $\Phi(\eta,\kappa)\le A_0$ for all $\eta\in X'$ and $\kappa\in [0,\kappa_0]$.
\end{enumerate}
Then for any closed subset $F\subset X'$ and open subset $O\subset X'$, we have 
\begin{align}
	&\limsup_{\kappa\to 0} \kappa \log \E^{(\kappa)} \left[ \exp \left( \frac{1}{\kappa} \Phi(\eta; \kappa) \right) \one \left\{ \eta\in F \right\} \right] \le -\inf_{\eta \in F} \left( \LI(\eta)-\Phi(\eta,0) \right), \label{eqn::Varadhan_F} \\
	&\limsup_{\kappa\to 0} \kappa \log \E^{(\kappa)} \left[ \exp \left( \frac{1}{\kappa} \Phi(\eta; \kappa) \right) \one \left\{ \eta\in O \right\} \right] \ge -\inf_{\eta \in O} \left( \LI(\eta)-\Phi(\eta,0) \right) \label{eqn::Varadhan_O}.
\end{align}
\end{lemma}

\begin{proof}
Note that the conditions (a,b,c,d') are stronger than the conditions (a,b,c,d). Thus we directly use conditions (a,b,c,d) to prove the lemma.
	\medbreak
First, we prove~\eqref{eqn::Varadhan_F}. Let $\overline{F}$ be the closure of $F$ in $X$. For $\eps,M>0$, define
\begin{equation*}
	\Phi^{M,\eps}(\eta):=\begin{cases}
		\max \left\{-M,\Phi_0(\eta)+\eps +\Phi_2(\eta) (f_2(0)-\eps)  \right\}, & \eta\in X', \\
		-M, & \eta \in X\setminus X'.
	\end{cases}
\end{equation*}
By conditions~(b) and (c), there exists $\kappa_1\in (0,\kappa_0)$ such that $\Phi(\eta; \kappa)\le \Phi^{M,\eps}(\eta)$ for $\eta\in X'$ and $\kappa\in [0,\kappa_1]$. 
	\medbreak
Note that $\Phi^{M,\eps}(\eta)$ satisfies the moment condition~\eqref{eqn::initial_moment_cond} when $\eps<(1-1/\beta) f_2(0)$, because
\begin{align}
	& \limsup_{\kappa\to 0} \kappa \log \E^{(\kappa)} \left[ \exp \left( \frac{\beta}{\kappa} \Phi^{M,\eps}(\eta) \right) \right] \notag\\
	\le & \limsup_{\kappa\to 0} \kappa \log \left( \ee^{-M/\kappa} + \E^{(\kappa)} \left[ \exp \left( \frac{\beta}{\kappa} \left( \Phi_0(\eta)+ \Phi_2(\eta) (f_2(0)-\eps) + \eps \right) \right) \one \left\{ \eta\in X' \right\} \right] \right) \notag \\
	\le & \limsup_{\kappa\to 0} \kappa \log \left( \ee^{-M/\kappa} + \ee^{\eps \beta/\kappa} \E^{(\kappa)} \left[ \exp \left( \frac{1}{\kappa} \left( \beta \Phi_0(\eta) + \Phi_2(\eta) f_2(0) \right) \right) \one \left\{ \eta\in X' \right\} \right] \right) \notag\\
	\le & \limsup_{\kappa\to 0} \kappa \log \left( \ee^{-M/\kappa} + \ee^{(\eps \beta+A_1 A'_1)/\kappa} \E^{(\kappa)} \left[ \exp \left( \frac{1}{\kappa} \left( \beta \Phi_0(\eta)+f_1(\kappa)\Phi_1(\eta)+f_2(\kappa)\Phi_2(\eta) \right) \right) \one \left\{ \eta\in X' \right\} \right] \right) \tag{due to conditions~(b) and~(c)}\\
	< & \infty. \tag{due to condition~(d)}
\end{align}
We have
\begin{align}
	&\limsup_{\kappa\to 0} \kappa \log \E^{(\kappa)} \left[ \exp \left( \frac{1}{\kappa} \Phi(\eta; \kappa) \right) \one \left\{ \eta\in F \right\} \right] \notag \\
	\le & \limsup_{\kappa\to 0} \kappa \log \E^{(\kappa)} \left[ \exp \left( \frac{1}{\kappa} \Phi^{M,\eps}(\eta) \right) \one \left\{ \eta\in \overline{F}  \right\} \right]  \notag \\
	\le & -\inf_{\eta\in \overline{F}} \left\{ \LI(\eta)-\Phi^{M,\eps}(\eta) \right\}  \tag{due to~\eqref{eqn::initial_Varadhan_F}} \\
	\to & -\inf_{\eta\in F} \left( \LI(\eta)-\Phi(\eta,0) \right) \quad \text{ as } \eps\downarrow 0, M\uparrow\infty, \notag
\end{align}
which implies~\eqref{eqn::Varadhan_F}.
	\medbreak
Next, we prove~\eqref{eqn::Varadhan_O}. We have
\begin{align}
	&\liminf_{\kappa\to 0} \kappa \log \E^{(\kappa)} \left[ \exp \left( \frac{1}{\kappa} \Phi(\eta; \kappa) \right) \one \left\{ \eta\in O \right\} \right] \notag \\
	\ge & \liminf_{\kappa\to 0} \kappa \left( -A_1 A'_1 + \log \E^{(\kappa)} \left[ \exp \left( \frac{1}{\kappa} \Phi(\eta,0) \right) \one \left\{ \eta\in O \right\} \right] \right) \tag{due to conditions~(b) and~(c)} \\
	\ge & -\inf_{\eta \in O} \left( \LI(\eta)-\Phi(\eta,0) \right), \tag{due to~\eqref{eqn::initial_Varadhan_O}}
\end{align}
which implies~\eqref{eqn::Varadhan_O}.
\end{proof}

\begin{remark} \label{rmk::Varadhan}
In fact, the decomposition of $\Phi(\eta; \kappa)$ in~\eqref{eqn::Def_Phi_gamma_kappa} can be more general. Assume there exists a positive integer $m$ such that
\begin{equation*}
	\Phi(\eta; \kappa)=\Phi_0(\eta)+\sum_{j=1}^{m} f_j(\kappa)\Phi_j(\eta).
\end{equation*}
The condition~(d) becomes that there exists $\beta>1$ that 
\begin{equation}\label{eqn::conditiond_general}
	\limsup_{\kappa\to 0} \kappa \log \E^{(\kappa)} \left[ \exp \left( \frac{1}{\kappa} \left( \beta \Phi_0(\eta)+\sum_{j=1}^{m} f_j(\kappa)\Phi_j(\eta) \right) \right) \one\{\eta\in X'\} \right]<\infty,
\end{equation}
and the condition~(d') does not change.
Assume $\Phi_0(\eta)$ satisfies condition~(a); $(f_j(\kappa),\Phi_j(\kappa))$ satisfies condition~(b) or condition~(c) for $1\le j\le m$; and $\Phi(\eta; \kappa)$ satisfies condition~\eqref{eqn::conditiond_general}. Then~\eqref{eqn::Varadhan_F} and~\eqref{eqn::Varadhan_O} hold. The proof is the same as the proof of Lemma~\ref{lem::Varadhan}.
\end{remark}


\end{document}